\input epsf

\documentclass[final]{siamltex}

\usepackage{latexsym,amsmath,amssymb,amsfonts,amscd,multirow}
\usepackage{epsfig,subfigure,verbatim,epstopdf,graphics}
\usepackage{color}

\newtheorem{thm}{Theorem}[section]
\newtheorem{cor}[thm]{Corollary}
\newtheorem{lem}[thm]{Lemma}

\newtheorem{rem}[thm]{Remark}

\newcommand{\beq}{\begin{equation}}
\newcommand{\eeq}{\end{equation}}
\newcommand{\beqa}{\begin{eqnarray}}
\newcommand{\eeqa}{\end{eqnarray}}
\newcommand{\bit}{\begin{itemize}}
\newcommand{\eit}{\end{itemize}}
\newcommand{\bedef}{\begin{defn}}
\newcommand{\edefn}{\end{defn}}
\newcommand{\bpro}{\begin{prop}}
\newcommand{\epro}{\end{prop}}

\newcommand{\df}{\partial}

\newcommand{\bx}{{\bf x}}
\newcommand{\bE}{{\bf E}}
\newcommand{\bB}{{\bf B}}
\newcommand{\bJ}{{\bf J}}

\newcommand{\bU}{{\bf U}}
\newcommand{\bV}{{\bf V}}
\newcommand{\bn}{{\bf n}}
\newcommand{\mE}{{\mathcal E}}
\newcommand{\mT}{{\mathcal T}}
\newcommand{\mG}{{\mathcal G}}
\newcommand{\mU}{{\mathcal U}}

\newcommand{\Ox}{{\Omega_x}}
\newcommand{\Oxi}{{\Omega_\xi}}
\newcommand{\Kx}{{K_x}}
\newcommand{\Kxi}{{K_\xi}}

\newcommand{\veps}{{\varepsilon}}
\newcommand{\bveps}{{\mbox{\boldmath{$\varepsilon$}}}}
\newcommand{\bzeta}{{\mbox{\boldmath{$\zeta$}}}}
\newcommand{\bPi}{{\mbox{\boldmath{$\Pi$}}}}

\title
{Discontinuous Galerkin Methods for the Vlasov-Maxwell Equations}

\author{Yingda Cheng\thanks{Department of Mathematics, Michigan State University,
East Lansing, MI 48824 U.S.A. {\tt ycheng@math.msu.edu}}
  \and
  Irene M. Gamba\thanks{Department of Mathematics and ICES, University of Texas at Austin,
Austin, TX 78712 U.S.A.
 {\tt gamba@math.utexas.edu}}
   \and
 Fengyan Li\thanks{Department of Mathematical Sciences, Rensselaer Polytechnic Institute, Troy, NY 12180, U.S.A. 
 {\tt lif@rpi.edu}}
 \and
 Philip J. Morrison\thanks{Department of Physics and Institute for Fusion Studies, University of Texas at
Austin, Austin, TX 78712 U.S.A.
 {\tt morrison@physics.utexas.edu}}}

\date{\today}

\begin{document}


\maketitle

\begin{abstract}
Discontinuous Galerkin methods are developed for solving  the Vlasov-Maxwell system, methods that are designed to be systematically as accurate as one wants with provable conservation of mass and possibly  total energy.  Such properties in general are hard to achieve within other numerical method frameworks for simulating the Vlasov-Maxwell system. The proposed scheme employs discontinuous Galerkin discretizations for both the Vlasov and the Maxwell  equations, resulting in a consistent description of the distribution  function and electromagnetic fields.  It is proven, up to some boundary effects, that  charge is conserved and the total energy can be preserved with suitable choices of the numerical flux for the Maxwell  equations and the underlying approximation spaces.  Error estimates are  established for  several flux choices.  The scheme is tested on the streaming Weibel instability:   the  order of accuracy  and  conservation properties of the proposed method are verified.
\end{abstract}




\section{Introduction}
In this paper, we consider the Vlasov-Maxwell (VM) system, the most important  equation for the  modeling of collisionless magnetized plasmas.  In particular, we study the evolution of a single species of nonrelativistic electrons under the self-consistent electromagnetic field while the ions are treated as uniform fixed background. Under the scaling of the characteristic time by the inverse of the plasma frequency $\omega_p^{-1}$, length by the Debye length $\lambda_D$, and  electric and magnetic fields by  $-m c \omega_p/e$ (with $m$ the electron mass, $c$ the speed of light, and $e$ the electron charge),  the  dimensionless form of the VM system is  
\begin{subequations}
\begin{align}
\partial_t f & + \xi \cdot \nabla_{\bx} f  +  (\bE + \xi \times \bB) \cdot \nabla_\xi f = 0~, \label{eq:vlasov}\\
\frac{\df \bE}{\df t}& = \nabla_{\bx}  \times \bB -  \bJ, \qquad
\frac{\df \bB}{\df t} = -  \nabla_{\bx}   \times \bE~,  \quad\label{eq:max:2} \\
\nabla_\bx \cdot \bE &= \rho-\rho_i, \qquad
\nabla_\bx \cdot \bB = 0~, \quad
\label{eq:max:4}
\end{align}
\end{subequations}
with
\begin{equation*}
\rho(\bx, t)= \int_\Oxi f(\bx, \xi, t)d\xi,\qquad \bJ(\bx, t)=  \int_\Oxi f(\bx, \xi, t)\xi d\xi~, 
\end{equation*}
where the equations are defined on  $\Omega=\Ox\times\Oxi$,  $\mathbf{x} \in \Ox$ denotes position in  physical space, and $\xi \in \Oxi$  in  velocity space.  Here 
$f(\bx, \xi, t)\geq 0$ is the distribution function of electrons at position $\bx$ with velocity $\xi$ at time $t$,  $\bE(\bx, t)$ is the electric field, $\bB(\bx, t)$ is the magnetic field, $\rho(\bx, t)$ is the electron charge density, and $\bJ(\bx, t)$ is the current density.  The charge  density of background ions is denoted by $\rho_i$, which is chosen to satisfy total charge neutrality,  $\int_\Ox \left(\rho(\bx, t)-\rho_i \right)\,d\bx=0$.       {Periodic boundary conditions in ${\bf x}$-space are assumed and the initial conditions are denoted by  $f_0=f(\bx, \xi, 0)$, $\bE_0=\bE(\bx,0)$ and $\bB_0=\bB(\bx,0)$.   We also assume that the initial distribution function $f_0(\bx,\xi) \in H^m(\Omega) \cap L^1_2(\Oxi)$, i.e., is in a Sobolev space of order $m$ and  is integrable  with finite energy in $\xi$-space, where $L^{p}_{m}(\Oxi) \equiv \{ \psi: \int_{\Oxi} |\psi|^p(1+
|\xi|^2)^{m/2} \,d\xi < \infty \}$.   The initial fields $\bE_0(\bx)$ and  $\bB_0(\bx)$ are also assumed to be in $H^m(\Omega_x)$.}

The VM system has wide importance  in plasma physics for describing  space and laboratory plasmas, with application to fusion devices,  high-power microwave generators,  and  large scale particle accelerators. The computation of the initial boundary value problem associated to the VM system is quite challenging,  due to the high-dimensionality (6D+time) of the Vlasov equation, multiple temporal and spatial scales associated with various physical phenomena, nonlinearity,  and the conservation of physical quantities due to the Hamiltonian structure \cite{morrison_80, morrison_13} of the system. Particle-in-cell (PIC) methods  \cite{Birdsall_book1991, Hockney_book1981} have long been very popular  numerical tools, in which the particles are advanced in a Lagrangian framework, while the field equations are solved on a mesh.  This remains an active area of research \cite{shadwick}.  In recent years, there has been growing interest in computing the Vlasov equation in a deterministi
 c framework. In the context of 
the Vlasov-Poisson system,  semi-Lagrangian methods  \cite{chengknorr_76, Sonnendrucker_1999}, finite volume (flux balance) methods \cite{Boris_1976,Fijalkow_1999, 
Filbet_PFC_2001}, Fourier-Fourier spectral methods \cite{Klimas_1987, Klimas_1994}, and continuous finite element methods \cite{Zaki_1988_1, Zaki_1988_2} have been  proposed, among many others. In the context of VM simulations, Califano \emph{et al.} have    used a semi-Lagrangian approach to compute the  streaming Weibel (SW) instability \cite{califano1998ksw}, current filamentation instability \cite{mangeney2002nsi}, magnetic vortices \cite{califano1965ikp}, magnetic reconnection \cite{califano2001ffm}. Also, various methods have  been  proposed for the relativistic VM system \cite{Sircombe20094773, Besse20087889, Suzuki20101643, Huot2003512}.

In this paper, we propose the use  of discontinuous Galerkin (DG) methods for solving  the VM system. 
What motivates us to choose DG methods, besides their many widely recognized desirable properties, is that they can be designed systematically to be as accurate as one wants, meanwhile with provable conservation of mass and possibly also the total energy. This is in general hard to achieve within other numerical method frameworks for simulating the VM system. The proposed scheme employs DG discretizations for both the Vlasov and the Maxwell equations, resulting in a consistent description of the distribution  function and electromagnetic fields.
We will show that up to some boundary effects, depending on the size of the computational domain, the total charge (mass)  is conserved and the total energy can be preserved with  a suitable choice of the numerical flux for the Maxwell  equations and   underlying approximation spaces. Error estimates are further established for  several flux choices.
The DG scheme can be  implemented on both structured and unstructured meshes with provable accuracy and stability for many linear and nonlinear problems, it is advantageous in long time wave-like simulations because it has low dispersive and dissipative errors \cite{Ainsworth:2004}, and it is very suitable for adaptive and parallel implementations. 
 The original DG method  was introduced by Reed and Hill \cite{Reed_hill_73} for a neutron transport equation.  Lesaint and Raviart \cite{Lesaint_r_74} performed the first error estimates for the original DG method, while  Cockburn and Shu in a series of papers
\cite{Cockburn_1991_MMNA_RK, Cockburn_1989_MC_RK_DG,Cockburn_1989_JCP_RK_localDG, Cockburn_1990_MC_RK_DG,Cockburn_1998_JCP_RK} developed the Runge-Kutta DG (RKDG) methods for hyperbolic equations. RKDG methods have  been used to simulate the Vlasov-Poisson system in plasmas \cite{Heath, Heath_thesis, Cheng_Gamba_Morrison} and for a  gravitational infinite homogeneous stellar system \cite{Cheng_jeans}. Some theoretical aspects about stability, accuracy and conservation of these methods in their semi-discrete form are discussed in \cite{Heath_thesis, Ayuso2009, Ayuso2010}.
Recently,  semi-Lagrangian DG methods \cite{rossmanith_11, qiu_ppdg_11} were proposed for the Vlasov-Poisson system.  In \cite{Jacobs_Hesthaven_2006, Jacobs_Hesthaven_2009}, DG discretizations for  Maxwell's equations were coupled with PIC methods to solve the VM system. In a recent work \cite{heyang},  error estimates of fully discrete RKDG methods are studied for the VM system.

The rest of the paper is  organized as follows: in Section \ref{nummeth}, we describe the numerical algorithm. In Section \ref{conserve},   conservation and the stability are established for the method. In Section \ref{error}, we provide the error estimates of the scheme. Section \ref{numres} is devoted to discussion of simulation results. We conclude with a few remarks in Section \ref{conclu}.

\section{Numerical Methods}
\label{nummeth}

In this section, we will introduce the DG algorithm for the VM system.   We consider  an infinite, homogeneous plasma, where all boundary conditions in $\bx$ are  periodic, {and $f(\mathbf{x},\xi,t)$ is assumed to be compactly supported in $\xi$. This assumption is consistent with the fact that the solution of the VM system is expected to decay at infinity in $\xi$-space, preserving integrability and its kinetic energy.}

Without loss of generality, we assume  $\Ox=(-L_x, L_x]^{d_x}$ and $\Oxi=[-L_\xi, L_\xi]^{d_\xi}$, where the velocity space domain $\Oxi$ is chosen large enough so that $f=0$ at and near the phase space boundaries. We take $d_x=d_\xi=3$ in the following sections, although the method and its analysis can be extended directly to the cases when $d_x$ and $d_\xi$ take any values from $\{1, 2, 3\}$.

In our analysis,  the assumption that  $f(\bx, \xi, t)$ remain  compactly  support in $\xi$,  given that it is initially so,  is an open question in the general setting.  The answer to this question is important for proving  the existence of a globally defined  classical solution, and its failure  could indicate the formation of  shock-like solutions  of the VM system.  Whether or not the three-dimensional VM system is globally well-posed as a Cauchy problem is a major open problem.  The limited results of global existence without uniqueness of weak solutions and  well-posedness and regularity of solutions assuming either some symmetry or near  neutrality constitute the present  extent of  knowledge  \cite{GlasseyARMA86,GlasseyCMP87, GlasseyCMP88, DipernaCPAM89,GlasseyCMP97, GlasseyARMA98II,GlasseyARMA98I}.

\subsection{Notations}

Throughout the paper,   standard notations will be used for the Sobolev spaces. Given a bounded domain $D\in {\mathbb R}^\star$ (with $\star=d_x, d_\xi$, or $ d_x+d_\xi$) and any nonnegative integer $m$, $H^m(D)$ denotes the $L^2$-Sobolev space of order $m$ with the standard Sobolev norm $||\cdot||_{m,D}$, and $W^{m,\infty}(D)$ denotes the $L^\infty$-Sobolev space of order $m$ with the standard Sobolev norm $||\cdot||_{m,\infty,D}$ and the semi-norm $|\cdot|_{m,\infty,D}$. When $m=0$, we also use $H^0(D)=L^2(D)$ and $W^{0,\infty}(D)=L^\infty(D)$.

Let $\mT_h^x=\{\Kx\}$ and $\mT_h^\xi=\{\Kxi\}$ be  partitions of $\Ox$ and $\Oxi$, respectively,  with $\Kx$ and $\Kxi$ being (rotated) Cartesian elements or simplices;  then $\mT_h=\{K: K=\Kx\times\Kxi, \forall \Kx\in\mT_h^x, \forall \Kxi\in\mT_h^\xi\}$ defines a partition of $\Omega$.
Let $\mE_x$ be the set of the edges of $\mT_h^x$ and  $\mE_\xi$ be the set of the edges of $\mT_h^\xi$;  then the edges of $\mT_h$ will be $\mE=\{\Kx\times e_\xi: \forall\Kx\in\mT_h^x, \forall e_\xi\in\mE_\xi\}\cup \{e_x\times\Kxi: \forall e_x\in\mE_x, \forall \Kxi\in\mT_h^\xi\}$. Here we take into account the periodic boundary condition in the $\bx$-direction when defining $\mE_x$ and $\mE$. Furthermore, $\mE_\xi=\mE_\xi^i\cup\mE_\xi^b$ with $\mE_\xi^i$ and $\mE_\xi^b$ being the set of interior and boundary edges of $\mT_h^\xi$, respectively.
In addition, we denote the mesh size of $\mT_h$ as $h=\textrm{max}(h_x, h_\xi)=\max_{K\in\mT_h}h_K$, where
$h_x=\max_{\Kx\in\mT_h^x} h_\Kx$ with $h_\Kx=\textrm{diam}(\Kx)$, $h_\xi=\max_{\Kxi\in\mT_h^\xi} h_\Kxi$ with $h_\Kxi=\textrm{diam}(\Kxi)$, and $h_K=\max(h_\Kx, h_\Kxi)$ for $K=\Kx\times\Kxi$.
When the mesh is refined, we assume both $ \frac{h_x}{h_{\xi,\min}}$ and $\frac{h_\xi}{h_{x,\min}}$ are uniformly bounded from above by a positive constant $\sigma_0$. Here $h_{x,\min}=\min_{\Kx\in\mT_h^x} h_\Kx$ and $h_{\xi,\min}=\min_{\Kxi\in\mT_h^\xi} h_\Kxi$. It is further assumed that $\{\mT_h^\star\}_h$ is shape-regular with $\star=x$ or $\xi$. That is, if $\rho_{K_\star}$ denotes the diameter of the largest sphere included in $K_\star$, there is
\begin{equation*}
\frac{h_{K_\star}}{\rho_{K_\star}}\leq \sigma_\star,\qquad \forall K_\star\in\mT_h^\star
\end{equation*}
for a positive constant $\sigma_\star$ independent of $h_\star$.

Next we define the discrete spaces
\begin{subequations}
\begin{align}
\mG_h^{k}&=\left\{g\in L^2(\Omega): g|_{K=\Kx\times\Kxi}\in P^k(\Kx\times\Kxi), \forall \Kx\in\mT_h^x, \forall \Kxi\in\mT_h^\xi \right\}~,\label{eq:sp:f}\\
              &=\left\{g\in L^2(\Omega): g|_K\in P^k(K), \forall K\in\mT_h \right\}~,\notag\\
\mU_h^r&=\left\{\bU\in [L^2(\Omega_x)]^{d_x}: \bU|_\Kx\in [P^r(\Kx)]^{d_x}, \forall \Kx\in\mT_h^x \right\}~,
\end{align}
\end{subequations}
where $P^r(D)$ denotes the set of polynomials of  total degree at most $r$ on $D$, and $k$ and $r$ are nonnegative integers.
Note the space $\mG_h^{k}$, which  we use to approximate $f$,  is called P-type, and it can be replaced by the tensor product of P-type spaces in $\bx$ and $\xi$,
\beq
\left\{g\in L^2(\Omega): g|_{K=\Kx\times\Kxi}\in P^k(\Kx)\times P^k(\Kxi), \forall \Kx\in\mT_h^x, \forall \Kxi\in\mT_h^\xi \right\}~,
\label{eq:sp:f:1}
\eeq
or by the tensor product space in each variable, which is called  Q-type 
\beq
\left\{g\in L^2(\Omega): g|_{K=\Kx\times\Kxi}\in Q^k(\Kx)\times Q^k(\Kxi), \forall \Kx\in\mT_h^x, \forall \Kxi\in\mT_h^\xi \right\}~.\label{eq:sp:f:3}\
\eeq
{Here $Q^r(D)$ denotes the set of polynomials of    degree at most $r$ in each variable on $D$. The numerical methods formulated in this paper, as well as the conservation, stability, and error estimates,  hold when any of the spaces above is used to approximate $f$. In our simulations of Section \ref{numres}, we use the P-type of  \eqref{eq:sp:f} as it is the smallest and therefore renders the most cost efficient algorithm. In fact,  the ratios of these three spaces defined in \eqref{eq:sp:f}, \eqref{eq:sp:f:1} and \eqref{eq:sp:f:3} are $\sum_{n=0}^k {{n+2d-1}\choose {2d-1}}: (\sum_{n=0}^k {{n+d-1}\choose {d-1}})^2: (k+1)^{2d}$ with $d_x=d_\xi=d$.


For piecewise  functions defined with respect to $\mT_h^x$ or $\mT_h^\xi$, we further introduce the jumps and averages as follows. For any edge $e=\{K_x^+\cap K_x^-\}\in\mE_x$, with $\bn_x^\pm$ as the outward unit normal to $\partial K_x^\pm$,
$g^\pm=g|_{K_x^\pm}$, and $\bU^\pm=\bU|_{K_x^\pm}$, the jumps across $e$ are defined  as
\begin{equation*}
[g]_x={g^+}{\bn_x^+}+{g^-}{\bn_x^-},\qquad [\bU]_x={\bU^+}\cdot{\bn_x^+}+{\bU^-}\cdot{\bn_x^-},\qquad [\bU]_\tau=\bU^+\times\bn_x^++\bU^-\times\bn_x^-
\end{equation*}
and the averages are
\begin{equation*}
\{g\}_x=\frac{1}{2}({g^+}+{g^-}),\qquad \{\bU\}_x=\frac{1}{2}({\bU^+}+{\bU^-}).
\end{equation*}

By replacing the subscript $x$ with $\xi$, one can define $[g]_\xi$, $[\bU]_\xi$, $\{g\}_\xi$, and $\{\bU\}_\xi$ for an interior edge of $\mT_h^\xi$ in $\mE^i_\xi$. For a boundary edge $e\in\mE^b_\xi$ with $\bn_\xi$ being the outward unit normal, we use
\begin{equation}
[g]_\xi={g}{\bn_\xi},\qquad
\{g\}_\xi=\frac{1}{2}g,\qquad \{\bU\}_\xi=\frac{1}{2}{\bU}~.
\label{eq:jump:ave:b}
\end{equation}
This is consistent with the fact that the exact solution $f$ is compactly supported in $\xi$.

For convenience, we introduce some shorthand notations,
$
\int_{\Omega_\star}=\int_{\mT_h^\star}=\sum_{K_\star\in\mT_h^\star}\int_{K_\star}, \qquad \int_{\Omega}=\int_{\mT_h}=\sum_{K\in\mT_h}\int_{K},\qquad \int_{\mE_\star}=\sum_{e\in\mE_\star}\int_{e}~,\notag
$
where again  $\star$ is $x$ or $\xi$.
%
 In addition,
$||g||_{0,\mE}=(||g||^2_{0,\mE_x\times\mT_h^\xi}+||g||^2_{0,\mT_h^x\times\mE_\xi})^{1/2}$ with
$
||g||_{0,\mE_x\times\mT_h^\xi}=\left(\int_{\mE_x}\int_{\mT^\xi_h} g^2 d\xi ds_x\right)^{1/2}, \qquad ||g||_{0,\mT_h^x\times\mE_\xi}=\left(\int_{\mT^x_h}\int_{\mE_\xi} g^2 ds_\xi d\bx\right)^{1/2}~,
$
and $||g||_{0,\mE_x}=\left(\int_{\mE_x} g^2 ds_x\right)^{1/2}$.
There are several equalities that will be used later, which  can be easily verified using  the definitions of averages and jumps.
\begin{subequations}
\beq
\frac{1}{2}[g^2]_\star =\{g\}_\star[g]_\star, \;\textrm{with}\;\star=x \;\textrm{or}\; \xi~,\label{eq:equality:0}
\eeq
\beq
[\bU\times\bV]_x+\{\bV\}_x\cdot [\bU]_\tau-\{\bU\}_x\cdot [\bV]_\tau =0~,\label{eq:equality:1}
\eeq
\beq
[\bU\times\bV]_x+{\bV^+}\cdot [\bU]_\tau-{\bU^-}\cdot [\bV]_\tau =0,\qquad
[\bU\times\bV]_x+{\bV^-}\cdot [\bU]_\tau-{\bU^+}\cdot [\bV]_\tau =0~.\label{eq:equality:2}
\eeq
\end{subequations}

We end this subsection by  summarizing  some standard approximation properties of the above discrete spaces, as well as some inverse inequalities  \cite{Ciarlet:book}.  For any nonnegative integer $m$, let $\Pi^m$ be the $L^2$ projection onto $\mG^m_h$, and  $\bPi_x^m$ be the $L^2$ projection onto $\mU^m_h$,   then
\begin{lem}[{Approximation properties}]
\label{lem:appr}
There exists a constant $C>0$, such that  for any $g\in H^{m+1}(\Omega)$ and  $\bU\in [H^{m+1}(\Ox)]^{d_x}$, the following hold: 
\begin{align*}
||g-\Pi^m g||_{0,K}+h_K^{1/2}||g-\Pi^m g||_{0,\partial K}&\leq C h_K^{m+1}||g||_{m+1,K},\qquad \forall K\in\mT_h~,\\
||\bU-\bPi_x^m \bU||_{0,\Kx}+h_\Kx^{1/2}||\bU-\bPi_x^m \bU||_{0,\partial\Kx}&\leq C h_\Kx^{m+1}||\bU||_{m+1,\Kx}, \qquad \forall \Kx\in\mT_h^x~,\\
||\bU-\bPi_x^m \bU||_{0,\infty,\Kx}&\leq C h_\Kx^{m+1}||\bU||_{m+1,\infty,\Kx},\qquad \forall \Kx\in\mT_h^x~, 
\end{align*}
where the constant $C$ is independent of the mesh sizes $h_K$ and $h_\Kx$, but  depends on $m$ and the shape regularity parameters $\sigma_x$ and $\sigma_\xi$ of the mesh.
\end{lem}

\begin{lem}[{Inverse inequality}]
\label{lem:inverse}
There exists a constant $C>0$, such that  for any $g\in P^m(K)$ or $P^m(K_x)\times P^m(K_\xi)$ with $K=(K_x\times K_\xi)\in\mT_h$, and for any $\bU\in [P^m(K_x)]^{d_x}$, the following hold: 
\begin{equation*}
||\nabla_\bx g||_{0,K}\leq C h_\Kx^{-1}||g||_{0,K},\qquad ||\nabla_\xi g||_{0,K}\leq C h_\Kxi^{-1}||g||_{0,K},
\end{equation*}
\begin{equation*}
|| \bU||_{0,\infty,K_x}\leq C h_\Kx^{-{d_x}/2}||\bU||_{0,K_x},\qquad || \bU||_{0,\partial K_x}\leq C h_\Kx^{-1/2}||\bU||_{0,K_x}~, 
\end{equation*}
where the constant $C$ is independent of the mesh sizes $h_\Kx$, $h_\Kxi$, but depends on $m$ and the shape regularity parameters $\sigma_x$ and $\sigma_\xi$ of the mesh.
\end{lem}

\subsection{The Semi-Discrete DG Methods}
\label{sec:DGmethod}

On the PDE level, the two equations in \eqref{eq:max:4} involving the divergence of the magnetic and electric fields can be derived from the remaining part of the VM system;  therefore,  the numerical methods proposed in this section are formulated for the VM system without \eqref{eq:max:4}. We want to stress that even though in principle the initial satisfaction of these constraints is sufficient for their satisfaction for all time, in certain circumstance one may need to consider explicitly  such divergence conditions  in order to produce  physically relevant numerical simulations \cite{Munz:2000, barth2006role}.

Given $k, r\geq 0$, the semi-discrete DG methods for the VM system are defined by the  following procedure: for any $K=\Kx\times\Kxi\in\mT_h$, look for $f_h\in\mG_h^k$, $\bE_h, \bB_h\in\mU_h^r$, such that for any $g\in\mG_h^k$, $\bU, \bV\in\mU_h^r$, 
\begin{subequations}
\begin{align}
\int_K\df_t f_h g d\bx d\xi
&- \int_K f_h\xi\cdot\nabla_\bx g d\bx d\xi 
- \int_K f_h(\bE_h+\xi\times\bB_h)\cdot\nabla_\xi g d\bx d\xi\notag\\
&+ \int_{\Kxi}\int_{\df\Kx} \widehat{f_h \xi\cdot \bn_x} g ds_x d\xi + \int_{\Kx} \int_{\df\Kxi} \widehat{(f_h (\bE_h+\xi\times \bB_h)\cdot \bn_\xi)} g ds_\xi dx
=0~,\label{eq:scheme:1}\\
\int_\Kx\df_t\bE_h\cdot\bU d\bx
&=\int_\Kx\bB_h\cdot\nabla\times\bU d\bx+\int_{\df\Kx}\widehat{\bn_x\times\bB_h}\cdot \bU ds_x
   -\int_\Kx\bJ_h\cdot\bU d\bx~,\label{eq:scheme:2}\\
\int_\Kx\df_t\bB_h\cdot\bV d\bx
&=-\int_\Kx\bE_h\cdot\nabla\times\bV d\bx-\int_{\df\Kx}\widehat{\bn_x\times\bE_h}\cdot \bV ds_x~,\label{eq:scheme:3}
\end{align}
\end{subequations}
with
\beq
\label{eq:J}
{\bJ}_h(\bx, t)=\int_{\mT_h^\xi} f_h(\bx, \xi, t)\xi d\xi~.
\eeq
Here $\bn_x$ and $\bn_\xi$ are outward unit normals of $\df\Kx$ and $\df\Kxi$, respectively. All  `hat'  functions are numerical fluxes that are determined by   upwinding, i.e., 
\begin{subequations}
\begin{align}
\widehat{f_h \xi\cdot \bn_x}:&=\widetilde{f_h \xi}\cdot \bn_x=\left(\{f_h\xi\}_x+\frac{|\xi\cdot\bn_x|}{2}[f_h]_x\right)\cdot\bn_x~,\label{eq:flux:1}\\
\widehat{f_h (\bE_h+\xi\times \bB_h)\cdot \bn_\xi}:&=\widetilde{f_h (\bE_h+\xi\times \bB_h)}\cdot\bn_\xi\notag\\
&=\left(\{f_h(\bE_h+\xi\times\bB_h)\}_\xi+\frac{|(\bE_h+\xi\times\bB_h)\cdot\bn_\xi|}{2}[f_h]_\xi\right)\cdot\bn_\xi~,\label{eq:flux:3}\\
\widehat{\bn_x\times\bE_h}:&=\bn_x\times\widetilde{\bE_h}
=\bn_x\times \left(\{\bE_h\}_x+\frac{1}{2}[\bB_h]_\tau\right)~,\label{eq:flux:4}\\
\widehat{\bn_x\times\bB_h}:&=\bn_x\times\widetilde{\bB_h}
=\bn_x\times \left(\{\bB_h\}_x-\frac{1}{2}[\bE_h]_\tau\right)~,
\label{eq:flux:5}
\end{align}
\end{subequations}
where these relations define the meaning of `tilde'. 

For the Maxwell part, we also consider two other numerical fluxes: central flux and alternating flux, which are defined by 
\begin{subequations}
\begin{align}
&\mbox{Central flux:}\qquad \widetilde{\bE_h}=\{\bE_h\}, \;\widetilde{\bB_h}=\{\bB_h\}~,\label{eq:flux:6}\\
&\mbox{Alternating flux:}\qquad \widetilde{\bE_h}=\bE_h^+, \; \widetilde{\bB_h}=\bB_h^-, \;\mbox{or}\;
\widetilde{\bE_h}=\bE_h^-,\; \widetilde{\bB_h}=\bB_h^+~.\label{eq:flux:7}
\end{align}
\end{subequations}

Upon summing up \eqref{eq:scheme:1} with respect to $K\in\mT_h$ and  similarly summing  \eqref{eq:scheme:2} and \eqref{eq:scheme:3} with respect to $K_x\in\mT_h^x$, the numerical method becomes the following: look for $f_h\in\mG_h^k$, $\bE_h, \bB_h\in\mU_h^r$, such that
\begin{subequations}
\begin{align}
a_h(f_h,\bE_h,\bB_h; g)&=0~,\label{eq:scheme:1:a}\\
b_h(\bE_h,\bB_h; \bU, \bV)&=l_h(\bJ_h;\bU)~,\label{eq:scheme:2:a}
\end{align}
\end{subequations}
for any $g\in\mG_h^k$, $\bU, \bV\in\mU_h^r$, where
\begin{align*}
a_h(f_h,\bE_h,\bB_h; g)=&a_{h,1}(f_h;g)+a_{h,2}(f_h, \bE_h, \bB_h; g)~,\qquad l_h(\bJ_h; \bU)=-\int_{\mT_h^x}\bJ_h\cdot\bU d\bx~\\
b_h(\bE_h, \bB_h; \bU, \bV)=
&\int_{\mT_h^x}\df_t\bE_h\cdot\bU d\bx
-\int_{\mT_h^x}\bB_h\cdot\nabla\times\bU d\bx-\int_{\mE_x}\widetilde{\bB_h}\cdot [\bU]_\tau ds_x\\
&+\int_{\mT_h^x}\df_t\bB_h\cdot\bV d\bx
+\int_{\mT_h^x}\bE_h\cdot\nabla\times\bV d\bx+\int_{\mE_x}\widetilde{\bE_h}\cdot[\bV]_\tau ds_x~,
\end{align*}
and
\begin{align*}
a_{h,1}(f_h;g)&=\int_{\mT_h}\df_t f_h g d\bx d\xi- \int_{\mT_h} f_h\xi\cdot\nabla_\bx g d\bx d\xi +\int_{\mT_h^\xi}\int_{\mE_x} \widetilde{f_h \xi}\cdot [g]_x ds_x d\xi~,\\
a_{h,2}(f_h,\bE_h, \bB_h; g)&=- \int_{\mT_h} f_h(\bE_h+\xi\times\bB_h)\cdot\nabla_\xi g d\bx d\xi+\int_{\mT_h^x}\int_{\mE_\xi} \widetilde{f_h (\bE_h+\xi\times\bB_h)}\cdot [g]_\xi ds_\xi d\bx~.
\end{align*}
Note,  $a_h$ is linear with respect to $f_h$ and $g$, yet it is in general nonlinear with respect to $\bE_h$ and $\bB_h$ due to \eqref{eq:flux:3}.
Recall,  the exact solution $f$ has compact support in $\xi$; therefore,  the numerical fluxes of  \eqref{eq:flux:1}-\eqref{eq:flux:5} and  \eqref{eq:flux:6} and \eqref{eq:flux:7} are consistent and, consequently,  so is the proposed numerical method. That is, the exact solution $(f, \bE, \bB)$ satisfies
\begin{align*}
a_h(f, \bE, \bB; g)&=0, \qquad \forall g\in \mG_h^k~,\\
b_h(\bE, \bB; \bU, \bV)&=l_h(\bJ; \bU), \qquad \forall \bU, \bV \in \mU_h^r~.
\end{align*}

\subsection{Temporal Discretizations}

We use total variation diminishing (TVD) high-order Runge-Kutta
methods to solve the method of lines  ODE resulting from the semi-discrete DG scheme, $\frac{d}{dt}G_h=R(G_h)$.
Such time stepping methods are convex combinations of the Euler
forward time discretization. The commonly used third-order TVD
Runge-Kutta method  is given by
\begin{eqnarray}
\label{3rk}
G_h^{(1)}&=&G_h^n+\triangle t R(G_h^n) \nonumber \\
G_h^{(2)}&=&\frac{3}{4}G_h^n+\frac{1}{4}G_h^{(1)}+
\frac{1}{4}\triangle t R(G_h^{(1)}) \nonumber \\
G_h^{n+1}&=&\frac{1}{3}G_h^n+\frac{2}{3}G_h^{(2)}+
\frac{2}{3}\triangle t R(G_h^{(2)}),
\end{eqnarray}
where $G_h^n$ represents a numerical approximation of the solution at discrete time $t_n$.  A detailed description of the TVD Runge-Kutta method can be found in \cite{Shu_1988_JCP_NonOscill}; see also \cite{Gottlieb_1998_MC_RK} and \cite{Gottlieb_2001_SIAM_Stabi}  for strong-stability-perserving
methods.

\section{Conservation and Stability}
\label{conserve}

In this section, we will establish  conservation and stability properties of  the semi-discrete DG methods.
In particular, we prove that subject to boundary effects, the total charge  (mass) is always conserved. As for the total energy of the system,   conservation depends on the choice of numerical fluxes for the Maxwell equations.  We also show that $f_h$ is $L^2$ stable, which facilitates the error analysis of Section \ref{error}.    

\begin{lem}[{Mass conservation}]
The numerical solution $f_h\in\mG_h^k$ with $k\geq 0$ satisfies
\begin{equation}
\frac{d}{dt}\int_{\mT_h} f_h d\bx d\xi+\Theta_{h,1}(t)=0~,
\label{eq:MassC}
\end{equation}
where 
\begin{equation*}
\Theta_{h,1}(t)=\int_{\mT_h^x}\int_{\mE_\xi^b} f_h \max((\bE_h+\xi\times\bB_h)\cdot\bn_\xi, 0)ds_\xi d\bx~.
\end{equation*}
Equivalently, with $\rho_h(\bx, t)=\int_{\mT_h^\xi} f_h(\bx,\xi,t) d\xi$, for any $T>0$, the following holds: 
\begin{align}
\int_{\mT_h^x} \rho_h(\bx, T)d\bx+\int_0^T\Theta_{h,1}(t)dt
=\int_{\mT_h^x} \rho_h(\bx, 0) d\bx~.
\label{eq:MassC:1}
\end{align}
\label{lem:MassC}
\end{lem}

 \begin{proof}
 Let $g(\bx, \xi)=1$. Note that $g\in\mathcal{G}_h^k$, for any $k\geq 0$,  is continuous  and
$\nabla_\bx g=0$. Taking  this $g$ as the test function in \eqref{eq:scheme:1:a}, one has
\begin{align*}
\frac{d}{dt}\int_{\mT_h} f_h d\bx d\xi+\int_{\mT_h^x}\int_{\mE_\xi^b} \widetilde{f_h (\bE_h+\xi\times\bB_h)}\cdot [g]_\xi ds_\xi d\bx=0~.
\end{align*}
With the numerical flux of \eqref{eq:flux:3} and the average and jump across $\mE_\xi^b$ of  \eqref{eq:jump:ave:b}, the second term above becomes
\begin{align}
&\int_{\mT_h^x}\int_{\mE_\xi^b} \widetilde{f_h (\bE_h+\xi\times\bB_h)}\cdot \bn_\xi ds_\xi d\bx\\ =&\int_{\mT_h^x}\int_{\mE_\xi^b} \frac{f_h}{2} \left((\bE_h+\xi\times\bB_h)\cdot\bn_\xi+|(\bE_h+\xi\times\bB_h)\cdot\bn_\xi|\right)ds_\xi d\bx=\Theta_{h,1}(t)~,
\end{align}
and this gives \eqref{eq:MassC}. Integrating in time from $0$ to $T$ gives \eqref{eq:MassC:1}.
\end{proof}

\begin{lem}[{Energy conservation 1}]
\label{lem:EneC}
For $k\geq 2$, $r\geq 0$,  the numerical solution $f_h\in\mG_h^k$, $\bE_h, \bB_h\in\mU_h^r$ with the upwind numerical fluxes \eqref{eq:flux:1}-\eqref{eq:flux:5} satisfies
\begin{equation}
\frac{d}{dt}\left(\int_{\mT_h} f_h|\xi|^2 d\bx d\xi+\int_{\mT_h^x} (|\bE_h|^2+|\bB_h|^2) d\bx\right)
+\Theta_{h,2}(t)+\Theta_{h,3}(t)=0~,
\label{eq:EneC}
\end{equation}
with
\begin{equation*}
\Theta_{h,2}(t)=\int_{\mE_x}\left( |[\bE_h]_\tau|^2+|[\bB_h]_\tau|^2 \right) ds_x~, \qquad \Theta_{h,3}(t) 
=\int_{\mT_h^x}\int_{\mE_\xi^b}f_h|\xi|^2 \max((\bE_h+\xi\times\bB_h)\cdot\bn_\xi, 0)  ds_\xi d\bx~.
\end{equation*}
%
%
\end{lem}
\begin{proof}~
{\sf Step 1:} Let $g(\bx, \xi)=|\xi|^2$. Note that $g\in\mathcal{G}_h^k$ for $k\geq 2$ and it is continuous. In addition,
$\nabla_\bx g=0$, $\nabla_\xi g=2\xi$, and $\xi\times\bU\cdot\nabla_\xi g=0$ for any function $\bU$. Taking this $g$ as the test function in \eqref{eq:scheme:1:a}, one has
\begin{align*}
& \frac{d}{dt}\int_{\mT_h} f_h|\xi|^2 d\bx d\xi
=2\int_{\mT_h} f_h\bE_h\cdot\xi d\bx d\xi- \int_{\mT_h^x}\int_{\mE_\xi^b} \widetilde{f_h (\bE_h+\xi\times\bB_h)}\cdot [|\xi|^2]_\xi ds_\xi d\bx\\
&=
2\int_{\mT_h^x}\bE_h\cdot\left(\int_{\mT_h^\xi} f_h\xi  d\xi \right)d\bx-
\int_{\mT_h^x}\int_{\mE_\xi^b}\left(\frac{1}{2}(\bE_h+\xi\times\bB_h)f_h+\frac{|(\bE_h+\xi\times\bB_h)\cdot\bn_\xi|}{2}f_h\bn_\xi\right)\cdot(|\xi|^2 \bn_\xi)ds_\xi d\bx\\
&=2 \int_{\mT_h^x}\bE_h\cdot\bJ_h d\bx-
\int_{\mT_h^x}\int_{\mE_\xi^b}\frac{f_h}{2}\left((\bE_h+\xi\times\bB_h)\cdot\bn_\xi+|(\bE_h+\xi\times\bB_h)\cdot\bn_\xi|\right)  |\xi|^2 ds_\xi d\bx\\
&=2 \int_{\mT_h^x}\bE_h\cdot\bJ_h d\bx-
\int_{\mT_h^x}\int_{\mE_\xi^b} f_h|\xi|^2  \max((\bE_h+\xi\times\bB_h)\cdot\bn_\xi, 0) ds_\xi d\bx
 \end{align*}
{\sf Step 2:} With $\bU=\bE_h$ and $\bV=\bB_h$, \eqref{eq:scheme:2:a} becomes
\begin{align*}
-\int_{\mT_h^x}\bJ_h\cdot\bE_h d\bx
&=\frac{1}{2}\frac{d}{dt}\int_{\mT_h^x} |\bE_h|^2 d\bx
-\int_{\mT_h^x}\bB_h\cdot\nabla\times\bE_h d\bx-\int_{\mE_x}\widetilde{\bB_h}\cdot [\bE_h]_\tau ds_x\\
&+\frac{1}{2}\frac{d}{dt}\int_{\mT_h^x} |\bB_h|^2 d\bx+\int_{\mT_h^x}\bE_h\cdot\nabla\times\bB_h d\bx+\int_{\mE_x}\widetilde{\bE_h}\cdot [\bB_h]_\tau ds_x~,\\
&=\frac{1}{2}\frac{d}{dt}\int_{\mT_h^x} \left(|\bE_h|^2+|\bB_h|^2\right) d\bx
-\int_{\mE_x}\left([\bE_h\times\bB_h]_x+\widetilde{\bB_h}\cdot [\bE_h]_\tau-\widetilde{\bE_h}\cdot [\bB_h]_\tau \right)ds_x~,\\
&=\frac{1}{2}\frac{d}{dt}\int_{\mT_h^x} \left(|\bE_h|^2+|\bB_h|^2\right) d\bx
 +\frac{1}{2}\int_{\mE_x}\left( |[\bE_h]_\tau|^2+|[\bB_h]_\tau|^2 \right) ds_x~.
\end{align*}
The last equality uses the formulas of the upwind fluxes \eqref{eq:flux:4}-\eqref{eq:flux:5} as well as \eqref{eq:equality:1}.

Combining  the results in previous two steps, one concludes  \eqref{eq:EneC}.
\end{proof}

\smallskip
\begin{cor}[{Energy conservation 2}]
For $k\geq 2$, $r\geq 0$ and   the numerical solution $f_h\in\mG_h^k$, $\bE_h, \bB_h\in\mU_h^r$ with the upwind numerical flux \eqref{eq:flux:1}-\eqref{eq:flux:3} for the Vlasov part, and with either the central or alternating flux of  \eqref{eq:flux:6}-\eqref{eq:flux:7} for   the Maxwell part,  the following holds:
\begin{equation*}
\frac{d}{dt}\left(\int_{\mT_h} f_h|\xi|^2 d\bx d\xi+\int_{\mT_h^x} (|\bE_h|^2+|\bB_h|^2) d\bx\right)+\Theta_{h,3}(t)=0~.
\end{equation*}
\label{col:energy}
\end{cor}
\begin{proof}
The proof proceeds the same way as for Lemma \ref{lem:EneC}. The only difference is that here the equalities \eqref{eq:equality:1}-\eqref{eq:equality:2} give 
\begin{equation*}
[\bE_h\times\bB_h]_x+\widetilde{\bB_h}\cdot [\bE_h]_\tau-\widetilde{\bE_h}\cdot [\bB_h]_\tau=0~,
\end{equation*}
 which holds for $\widetilde{\bE_h}$ and $\widetilde{\bB_h}$ defined in  the central or alternating flux in the Maxwell solver. 
\end{proof}

With either the central or alternating flux for the Maxwell solver, the energy does not change due to the tangential jump of the magnetic and electric fields as in Lemma \ref{lem:EneC}. This, on the other hand,  may have some effect on the accuracy of the methods (See Sections \ref{error} and \ref{numres} and also \cite{Ainsworth:2004}).

\medskip
{
\begin{rem}
In Lemma 3.2, the conservation error term satisfies 
$\Theta_{h,2}\ge  0$ with equality depending on the choice of numerical
fluxes for the Maxwell discretization.  In addition,  the error
conservation terms $\Theta_{h,1}$ in Lemma 3.1  and  $\Theta_{h,3}$ in
Lemmas 3.2 and Corollary 3.3 both depend on the numerical solution $f_h$
on the outflow portion of the computational boundary in $\xi$-space,
which is determined by the numerical electric and magnetic fields. Hence,
for the case of periodic boundary conditions in $x$-space for both the
Vlasov and Maxwell's equations, it can be easily shown that these error
terms $\Theta_{h,i} \approx 0,$ for $i = 1$ and $3$, by choosing the
computational domain in $\xi$-space sufficiently large.
\end{rem}
}

\medskip
\begin{rem}
Energy conservation holds as long as $|\xi|^2\in \mG^k_h$. Indeed, for $k<2$, the energy conservation results of Lemma \ref{lem:EneC} and Corollary \ref{col:energy} can be obtained if one replaces $\mG^k_h$ with ${\tilde\mG}^k_h=\mG^k_h\oplus\{|\xi|^2\}=\{g+c|\xi|^2,\; \forall g\in \mG^k_h, \forall c\in\mathbb{R}\}$.
\end{rem}

\bigskip


Finally, we can obtain the $L^2$-stability result for $f_h$, a result that  is independent of choice of numerical flux  in the Maxwell solver.  This result will  be used in the error analysis of Section \ref{error}. 
\begin{lem}[{$L^2$-stability of $f_h$}]
For $k\geq 0$, the numerical solution $f_h\in\mG_h^k$ satisfies
\begin{align}
\label{eq:L2Stab}
\frac{d}{dt}\left(\int_{\mT_h} |f_h|^2 d\bx d\xi\right)&+\int_{\mT_h^\xi}\int_{\mE_x} |\xi\cdot\bn_x||[f_h]_x|^2 ds_x d\xi\\
&+\int_{\mT_h^x}\int_{\mE_\xi} |(\bE_h+\xi\times\bB_h)\cdot\bn_\xi||[f_h]_\xi|^2 ds_\xi d\bx=0~.\notag
\end{align}
\label{lem:stability}
\end{lem}
\begin{proof}
Taking  $g=f_h$ in \eqref{eq:scheme:1:a}, one gets
\beq
\label{eq:L2Stab:1}
\frac{1}{2}\frac{d}{dt}\left(\int_{\mT_h} |f_h|^2 d\bx d\xi\right)
+R_1+R_2=0~,
\eeq
with
\begin{align*}
R_1 =- \int_{\mT_h} f_h\xi\cdot\nabla_\bx f_h d\bx d\xi+ \int_{\mT_h^\xi}\int_{\mE_x} \widetilde{f_h \xi}\cdot [f_h]_x ds_xd\xi,
\qquad R_2=a_{h,2}(f_h,\bE_h, \bB_h; f_h)~.
\end{align*}
Observe
\begin{align*}
R_1 &=- \int_{\mT_h^\xi}\sum_{\Kx\in\mT_h^x}\int_\Kx \xi \cdot\nabla_\bx \left(\frac{f_h^2}{2}\right) d\bx d\xi+ \int_{\mT_h^\xi} \int_{\mE_x} \widetilde{f_h \xi}\cdot [f_h]_x ds_x d\xi~, \\
&=- \int_{\mT_h^\xi}\sum_{\Kx\in\mT_h^x}\int_{\df\Kx} \xi \cdot \bn_x \left(\frac{f_h^2}{2}\right) ds_x d\xi+ \int_{\mT_h^\xi} \int_{\mE_x} \widetilde{f_h \xi}\cdot [f_h]_x ds_x d\xi~, \\
&=- \int_{\mT_h^\xi}\int_{\mE_x}\frac{1}{2} [\xi{f_h^2}]_x ds_x d\xi+ \int_{\mT_h^\xi}\int_{\mE_x} \widetilde{f_h \xi}\cdot [f_h]_x ds_x d\xi~, \\
&= \int_{\mT_h^\xi}\int_{\mE_x} \left(-\frac{1}{2}[\xi{f_h^2}]_x + \{f_h\xi\}_x\cdot[f_h]_x +\frac{1}{2}|\xi\cdot\bn_x|[f_h]_x\cdot [f_h]_x \right)ds_xd\xi~,\\
&= \int_{\mT_h^\xi}\int_{\mE_x} \left( (-\frac{1}{2}[{f_h^2}]_x + \{f_h\}_x [f_h]_x)\cdot\xi +\frac{1}{2}|\xi\cdot\bn_x||[f_h]_x|^2 \right)ds_xd\xi~,\\
&= \frac{1}{2}\int_{\mT_h^\xi}\int_{\mE_x} |\xi\cdot\bn_x||[f_h]_x|^2 ds_xd\xi~,
\end{align*}
where the fourth equality uses the definition of the numerical flux \eqref{eq:flux:1} and the last one is due to \eqref{eq:equality:0}. Similarly,
\begin{align*}
R_2&=- \int_{\mT_h^x}\sum_{\Kxi\in\mT_h^\xi}\int_\Kxi (\bE_h+\xi\times\bB_h)\cdot\nabla_\xi \left(\frac{f_h^2}{2}\right) d\xi d\bx+\int_{\mT_h^x}\int_{\mE_\xi} \widetilde{f_h (\bE_h+\xi\times \bB_h)}\cdot [f_h]_\xi ds_\xi d\bx~,\\
&=\int_{\mT_h^x}\int_{\mE_\xi} \left(-\frac{1}{2}[(\bE_h+\xi\times\bB_h) f_h^2]_\xi + \{f_h(\bE_h+\xi\times\bB_h)\}_\xi\cdot [f_h]_\xi+\frac{1}{2}|(\bE_h+\xi\times\bB_h)\cdot\bn_\xi|[f_h]_\xi\cdot [f_h]_\xi\right) ds_\xi d\bx~,\\
&=\int_{\mT_h^x}\int_{\mE_\xi} \left( (-\frac{1}{2}[f_h^2]_\xi + \{f_h\}_\xi\cdot [f_h]_\xi)\cdot (\bE_h+\xi\times\bB_h) +\frac{1}{2}|(\bE_h+\xi\times\bB_h)\cdot\bn_\xi||[f_h]_\xi|^2\right) ds_\xi d\bx~,\\
&=\frac{1}{2}\int_{\mT_h^x}\int_{\mE_\xi} |(\bE_h+\xi\times\bB_h)\cdot\bn_\xi||[f_h]_\xi|^2
ds_\xi d\bx~, 
\end{align*}
where the second equality is due to $\nabla_\xi\cdot(\bE_h+\xi\times\bB_h)=0$ and the definition of the numerical flux in \eqref{eq:flux:3}, and the third equality  uses  \eqref{eq:equality:0} and $\bE_h+\xi\times\bB_h$ being continuous in $\xi$. With \eqref{eq:L2Stab:1}, we conclude $L^2$ stability \eqref{eq:L2Stab}.
\end{proof}

\section{Error Estimates}
\label{error}

In this section, we  establish  error estimates at any given time $T>0$ for our semi-discrete DG methods described in  Section \ref{sec:DGmethod}. It is assumed that the discrete  spaces have the same degree, i.e.,  $k=r$, and that the exact solution satisfies $f \in C^1([0,T]; H^{k+1}(\Omega)\cap W^{1,\infty}(\Omega))$ and $\bE,\; \bB \in C^0([0,T]; [H^{k+1}(\Ox)]^{d_x}\cap [W^{1,\infty}(\Ox)]^{d_x})$. Also, periodic  boundary conditions in $\bx$  and  compact support for $f$ in $\xi$ are assumed.   To prevent the proliferation of constants,  we use $A\lesssim B$ to represent the inequality $A\leq (\textrm{constant}) B$, where the positive constant is independent of the mesh size $h$, $h_x$, and $h_\xi$, but  it can depend on the polynomial degree $k$, mesh parameters $\sigma_0, \sigma_x$ and $\sigma_\xi$, and domain parameters $L_x$ and $L_\xi$.

Defining  $\zeta_h=\Pi^k f-f$ and $\veps_h=\Pi^k f-f_h$, it follows that  $f-f_h=\veps_h-\zeta_h$.  
Analogously, if $\bzeta_h^E=\bPi_x^k \bE-\bE$, $\bzeta_h^B=\bPi_x^k \bB-\bB$,  $\bveps_h^E=\bPi_x^k\bE-\bE_h$ and $\bveps_h^B=\bPi_x^k\bB-\bB_h$, then $\bE-\bE_h=\bveps^E_h-\bzeta^E_h$ and $\bB-\bB_h=\bveps^B_h-\bzeta^B_h$. With the approximation  results of Lemma \ref{lem:appr}, we have 
\begin{equation}
||\zeta_h||_{0,\Omega}\lesssim h^{k+1} ||f||_{k+1,\Omega},\qquad ||\bzeta_h^B||_{0,\Ox}\lesssim h_x^{k+1} ||\bB ||_{k+1,\Ox},\qquad ||\bzeta_h^E||_{0,\Ox}\lesssim h_x^{k+1} ||\bE ||_{k+1,\Ox}~; 
\label{eq:110}
\end{equation}
therefore, we only need to estimate $\veps_h$, $\bveps^E_h$ and $\bveps^B_h$. The remainder of this section is organized as follows:  we first  state Lemmas \ref{lem:vepsh} and \ref{lem:vepshEB}, with which the main error estimate is established in Theorem \ref{thm:mainresult} for the proposed semi-discrete DG method with the upwind numerical fluxes.
Then,  the  proofs of Lemmas \ref{lem:vepsh} and \ref{lem:vepshEB} will be given in subsections \ref{sec:err:part1} and \ref{sec:err:part2}.  Lastly,  for the proposed method using the central or alternating flux of  \eqref{eq:flux:6}-\eqref{eq:flux:7} for the Maxwell solver,   error estimates are given in Theorem  \ref{thm:error:2}.

\begin{lem} [{Estimate of $\veps_h$}] Based on the semi-discrete DG discretization for the Vlasov equation of  \eqref{eq:scheme:1:a} with the upwind flux \eqref{eq:flux:1}-\eqref{eq:flux:3}, we have
\begin{align}
\frac{d}{dt}&\left(\int_{\mT_h} |\veps_h|^2 d\bx d\xi\right)
+\int_{\mT_h^\xi}\int_{\mE_x} |\xi\cdot\bn_x||[\veps_h]_x|^2 ds_x d\xi
+\int_{\mT_h^x}\int_{\mE_\xi} (|(\bE_h+\xi\times\bB_h)\cdot\bn_\xi|)||[\veps_h]_\xi|^2 ds_\xi d\bx\notag\\
\lesssim
& \left(h^{k+1} \hat{\Lambda} +h^k||f||_{k+1,\Omega}(||\bveps^E_h ||_{0,\infty,\Ox}+||\bveps^B_h ||_{0,\infty,\Ox})+ |f|_{1,\infty,\Omega}(||\bveps^E_h||_{0,\Ox}+||\bveps^B_h||_{0,\Ox}) \right) ||\veps_h||_{0,\Omega}\notag\\
& + h^{k+\frac{1}{2}} ||f||_{k+1,\Omega} \left(||\bveps^B_h||_{0,\infty,\Ox}^{1/2}+||\bveps^E_h||_{0,\infty,\Ox}^{1/2}+||\bB||_{0,\infty,\Ox}^{1/2}+||\bE||_{0,\infty,\Ox}^{1/2}\right)\notag\\
&\hspace{1 cm} \times \left(\int_{\mT_h^x}\int_{\mE_\xi}  |(\bE_h+\xi\times\bB_h)\cdot\bn_\xi| |[\veps_h]|^2  ds_\xi dx\right)^{1/2}\notag\\
&+ h^{k+\frac{1}{2}}  ||f||_{k+1,\Omega} \left(\int_{\mT_h^\xi}\int_{\mE_x} |\xi\cdot\bn_x|
 |[\veps_h]_x|^2  ds_xd\xi\right)^{1/2}~,
\label{eq:lem:vepsh}
\end{align}
with
\begin{align*}
\hat{\Lambda}=||\df_t f||_{k+1,\Omega}&+\left(1+||\bE||_{1,\infty,\Ox}
+||\bB||_{1,\infty,\Ox} \right)||f||_{k+1,\Omega}\\
&+ \left(||\bE||_{k+1,\Ox}+||\bB||_{k+1,\Ox}\right) |f|_{1,\infty,\Omega}~.
\end{align*}
\label{lem:vepsh}
\end{lem}

\begin{lem}[{Estimate of $\bveps^E_h$ and $\bveps^B_h$}]
Based on the semi-discrete DG discretization for the Maxwell equations of  \eqref{eq:scheme:2:a} with the upwind flux \eqref{eq:flux:4}-\eqref{eq:flux:5}, we have
\begin{align}
&\frac{d}{dt}\int_{\mT_h^x} \left(|\bveps_h^E|^2+|\bveps_h^B|^2\right) d\bx
 +\int_{\mE_x}\left( |[\bveps_h^E]_\tau|^2+|[\bveps_h^B]_\tau|^2 \right) ds_x
\label{eq:lem:vepshEB}\\
&\lesssim (||\veps_h||_{0,\Omega}+h^{k+1}||f||_{k+1,\Omega})||\bveps_h^E||_{0,\Ox} + h_x^{k+\frac{1}{2}}(||\bE||_{k+1,\Ox}+||\bB||_{k+1,\Ox})\left(\int_{\mE_x} |[\bveps_h^E]_\tau|^2+|[\bveps_h^B]_\tau|^2  ds_x\right)^{1/2}~.\notag
\end{align}
\label{lem:vepshEB}
\end{lem}

\begin{thm}[{Error estimate 1}]  For $k\geq 2$, the semi-discrete DG method  of \eqref{eq:scheme:1:a}-\eqref{eq:scheme:2:a},  for the Vlasov-Maxwell equations with the upwind fluxes of \eqref{eq:flux:1}-\eqref{eq:flux:5},  has the following error estimate
\begin{equation}
||(f-f_h)(t)||^2_{0,\Omega}+||(\bE-\bE_h)(t)||^2_{0,\Ox}+||(\bB-\bB_h)(t)||^2_{0,\Ox} \leq C h^{2k+1}, \qquad \forall\; t\in [0,T]~.
\label{eq:mainresult}
\end{equation}
Here the constant $C$ depends on the upper bounds of $||\df_t f||_{k+1,\Omega}$, $||f||_{k+1,\Omega}$, $|f|_{1,\infty,\Omega}$, $||\bE||_{1,\infty,\Ox}$, $||\bB||_{1,\infty,\Ox}$, $||\bE||_{k+1,\Ox}$, $||\bB||_{k+1,\Ox}$ over the time interval $ [0, T]$, and  it also depends on the polynomial degree $k$, mesh parameters $\sigma_0, \sigma_x$ and $\sigma_\xi$, and domain parameters $L_x$ and $L_\xi$.
\label{thm:mainresult}
\end{thm}

\begin{proof}
With several applications of Cauchy-Schwarz inequality and
\begin{equation*}
\widetilde{\Lambda}=h^{1/2}\hat{\Lambda}+||f||_{k+1,\Omega}\left(1+ ||\bE||^{1/2}_{0,\infty,\Ox}+||\bB||^{1/2}_{0,\infty,\Ox}\right)~,
\end{equation*}
 Eq.~\eqref{eq:lem:vepsh} becomes
\begin{align*}
\frac{d}{dt}
&\left(\int_{\mT_h} |\veps_h|^2 d\bx d\xi\right)\\
\leq
& c \left(h^{2k+1} \widetilde{\Lambda}^2 +(h^k||f||_{k+1,\Omega}(||\bveps^E_h ||_{0,\infty,\Ox}+||\bveps^B_h ||_{0,\infty,\Ox})+ |f|_{1,\infty,\Omega}(||\bveps^E_h||_{0,\Ox}+||\bveps^B_h||_{0,\Ox}) )^2 \right.\\
&\left.+ h^{2k+1} ||f||^2_{k+1,\Omega} (||\bveps^E_h||_{0,\infty,\Ox}+||\bveps^B_h||_{0,\infty,\Ox})\right)+ ||\veps_h||^2_{0,\Omega}~\\
\leq & c \left( h^{2k+1} \widetilde{\Lambda}^2 + h^{2k}(1+h)||f||^2_{k+1,\Omega}(||\bveps^E_h ||^2_{0,\infty,\Ox}+||\bveps^B_h ||^2_{0,\infty,\Ox})+ |f|^2_{1,\infty,\Omega}(||\bveps^E_h||^2_{0,\Ox}+||\bveps^B_h||^2_{0,\Ox})\right)\notag\\
&\;\; + ||\veps_h||^2_{0,\Omega}~.
\end{align*}
Here and below, the constant $c>0$ only depends on $k$, mesh parameters $\sigma_0, \sigma_x$ and $\sigma_\xi$, and domain parameters $L_x$ and $L_\xi$.
Moreover, with the inverse inequality of  Lemma \ref{lem:inverse},  and $\frac{h_\xi}{h_{x,\min}}$ being uniformly bounded by $\sigma_0$ when the mesh is refined, we have 
\begin{equation}
h^{2k}(||\bveps^E_h ||^2_{0,\infty,\Ox}+||\bveps^B_h ||^2_{0,\infty,\Ox})\leq c h^{2k-d_x}(||\bveps^E_h ||^2_{0,\Ox}+||\bveps^B_h ||^2_{0,\Ox})
\label{eq:1001}
\end{equation}
and this leads to
\begin{align}
\frac{d}{dt}
&\left(\int_{\mT_h} |\veps_h|^2 d\bx d\xi\right) \label{eq:101}\\
\leq & c \left( h^{2k+1} \widetilde{\Lambda}^2 + (h^{2k-d_x}(1+h)||f||^2_{k+1,\Omega}+ |f|^2_{1,\infty,\Omega})(||\bveps^E_h||^2_{0,\Ox}+||\bveps^B_h||^2_{0,\Ox})\right) + ||\veps_h||^2_{0,\Omega}~.\notag
\end{align}
Recall $d_x=3$, then for $k\geq 2$, there is $2k-d_x\geq 0$ and therefore $h^{2k-d_x}<\infty$.
Similarly, with the Cauchy-Schwarz inequality, \eqref{eq:lem:vepshEB} becomes
\begin{align}
\frac{d}{dt}
&\int_{\mT_h^x} \left(|\bveps_h^E|^2+|\bveps_h^B|^2\right) d\bx\label{eq:102}\\
\leq
& c \left(||\veps_h||^2_{0,\Omega} + h^{2k+2}||f||^2_{k+1,\Omega}+h_x^{2k+1}(||\bE||^2_{k+1,\Ox}+||\bB||^2_{k+1,\Ox})\right)+ ||\bveps_h^E||^2_{0,\Ox}~.\notag
\end{align}
Now,   summing  up \eqref{eq:101} and \eqref{eq:102}, we get
\begin{equation*}
\frac{d}{dt}\left(\int_{\mT_h} |\veps_h|^2 d\bx d\xi +\int_{\mT_h^x} |\bveps_h^E|^2+|\bveps_h^B|^2 d\bx \right)
\leq \Lambda h^{2k+1} + \Theta\left(\int_{\mT_h} |\veps_h|^2 d\bx d\xi +\int_{\mT_h^x} |\bveps_h^E|^2+|\bveps_h^B|^2 d\bx\right)~.
\end{equation*}
Here $\Lambda$ depends on $(f, \bE, \bB)$ in their Sobolev norms $||\df_t f||_{k+1,\Omega}$, $||f||_{k+1,\Omega}$, $|f|_{1,\infty,\Omega}$, $||\bE||_{1,\infty,\Ox}$, $||\bB||_{1,\infty,\Ox}$, $||\bE||_{k+1,\Ox}$, $||\bB||_{k+1,\Ox}$
at time $t$, and $\Theta$ depends on $||f||_{k+1,\Omega}$ and $|f|_{1,\infty,\Omega}$ at time t. Both $\Lambda$ and $\Theta$ depend on the polynomial degree $k$, mesh parameters $\sigma_0, \sigma_x$ and $\sigma_\xi$, and domain parameters $L_x$ and $L_\xi$. Now with a standard application of  Gronwall's inequality, a triangle  inequality,  and the approximation results of  \eqref{eq:110}, we  conclude the error estimate \eqref{eq:mainresult}.
\end{proof}

\medskip

\begin{rem}
Theorem \ref{thm:mainresult} shows that the proposed methods are $(k+\frac{1}{2})$-th order accurate, which is standard for upwind DG methods applied to  hyperbolic problems on general meshes. The assumption on the polynomial degree $k\geq 2$ is due to the lack of the $L^\infty$ error estimate for the DG solutions to the Maxwell solver and the use of an inverse inequality in handling the nonlinear coupling (see \eqref{eq:1001}-\eqref{eq:102} in the proof of Theorem \ref{thm:mainresult}). If the computational domain in $\bx$ is one- or two-dimensional ($d_x=1$ or $2$), then  Theorem \ref{thm:mainresult} holds for $k\geq 1$.
\label{rem:aftermainresult}
\end{rem}

\bigskip

If the upwind numerical flux for the Maxwell solver \eqref{eq:scheme:2:a} is replaced by either the central or alternating flux \eqref{eq:flux:6}-\eqref{eq:flux:7}, we will have the estimates for
$\bveps^E_h$ and $\bveps^B_h$ in Lemma \ref{lem:vepshEB:2} instead, provided an additional assumption is made for the mesh when it is refined. That is, we need to assume there is a positive constant $\delta<1$ such that for any $\Kx\in\mT_h^x$,
\begin{equation}
\delta\leq \frac{h_{\Kx'}}{h_\Kx}\leq \frac{1}{\delta}
\label{eq:meshratio}
\end{equation}
where  $\Kx'$ is any element in $\mT_h^x$ satisfying $\Kx'\cap\Kx \ne \emptyset$.

\begin{lem}[{Estimate of $\bveps^E_h$ and $\bveps^B_h$ with the non-upwinding flux}]
Based on the semi-discrete DG discretization for the Maxwell equations of \eqref{eq:scheme:2:a}, with either the central or alternating flux of  \eqref{eq:flux:6}-\eqref{eq:flux:7}, we have
\begin{align}
\frac{d}{dt}\int_{\mT_h^x} \left(|\bveps_h^E|^2+|\bveps_h^B|^2\right) d\bx
\lesssim &(||\veps_h||_{0,\Omega}+h^{k+1}||f||_{k+1,\Omega})||\bveps_h^E||_{0,\Ox} \\
& + c(\delta)h_x^k(||\bE||_{k+1,\Ox}+||\bB||_{k+1,\Ox})\left(\int_{\mT_h^x}(|\bveps_h^E|^2+|\bveps_h^B|^2) d\bx\right)^{1/2}~.\notag
\end{align}
\label{lem:vepshEB:2}
\end{lem}
The proof of this Lemma is given in Subsection \ref{sec:err:part3}.
With Lemma \ref{lem:vepshEB:2} and a  proof similar to that of  Theorem \ref{thm:mainresult}, the following error estimates can be established, but  the proof is omitted.
\begin{thm}[{Error estimate 2}]
For $k\geq 2$, the semi-discrete DG method of  \eqref{eq:scheme:1:a}-\eqref{eq:scheme:2:a} for Vlasov-Maxwell equations  with the upwind numerical flux \eqref{eq:flux:1}-\eqref{eq:flux:3} for the Vlasov solver and either the central or alternating fluxes of  \eqref{eq:flux:6}-\eqref{eq:flux:7} for the Maxwell solver, has the following error estimate: 
\begin{equation}
||(f-f_h)(t)||^2_{0,\Omega}+||(\bE-\bE_h)(t)||^2_{0,\Ox}+||(\bB-\bB_h)(t)||^2_{0,\Ox} \leq C h^{2k}, \qquad \forall\; t\in [0,T]~.
\label{eq:mainresult:2}
\end{equation}
Besides the dependence as in  Theorem \ref{thm:mainresult}, the constant $C$ also depends on $\delta$ of  \eqref{eq:meshratio}.
\label{thm:error:2}
\end{thm}

 Theorem \ref{thm:error:2} indicates that with either the central or alternating numerical flux for the Maxwell solver, the proposed method will be $k$-th order accurate. Also, one can see easily that the accuracy can be improved to  $(k+\frac{1}{2})$-th order as in Theorem \ref{thm:mainresult} if the discrete space for Maxwell solver is one degree higher than that for the Vlasov equation, namely, $r=k+1$. This improvement will require higher regularity for the exact solution $\bE$ and $\bB$.

In \cite{Ayuso2010}, optimal error estimates were established for some DG methods solving the multi-dimensional Vlasov-Poisson problem on Cartesian meshes with tensor-structure discrete space, defined in \eqref{eq:sp:f:3}, and $k\geq 1$.  Some of the techniques in \cite{Ayuso2010} are used in our analysis. In the present work, we focus on the P-type space $\mG_h^{k}$ in \eqref{eq:sp:f} in the numerical section, as it renders better cost efficiency and can be used on more general meshes.
Our analysis is established only for $k\geq 2$ due to the lack of the $L^\infty$ error estimate of the DG solver for the Maxwell part which is of hyperbolic nature, as pointed out in Remark   \ref{rem:aftermainresult}.

In the next three subsections, we will provide the proofs of Lemmas \ref{lem:vepsh}, \ref{lem:vepshEB} and \ref{lem:vepshEB:2}.
\subsection{Proof of Lemma \ref{lem:vepsh}}
\label{sec:err:part1}

Since the proposed method is consistent,  the error equation is related  to the Vlasov solver,
\begin{equation}
a_h(f,  \bE,  \bB; g_h)-a_h(f_h, \bE_h, \bB_h; g_h)=0,\qquad\forall g_h\in\mG_h^k~.
\label{eq:er:1}
\end{equation}
Note, $\veps_h\in\mG_h^k$; by taking $g_h=\veps_h$ in \eqref{eq:er:1}, one has
\begin{equation}
a_h(\veps_h,  \bE_h,  \bB_h; \veps_h)= a_h(\Pi^k f, \bE_h, \bB_h; \veps_h)-a_h(f,\bE, \bB; \veps_h)~.
\label{eq:er:2}
\end{equation}
Following the same lines as in the proof of Lemma \ref{lem:stability}, we get
\begin{align}
a_h(\veps_h, \bE_h,  \bB_h; \veps_h)=&\frac{1}{2}\frac{d}{dt}\left(\int_{\mT_h} |\veps_h|^2 d\bx d\xi\right)
+\frac{1}{2}\int_{\mT_h^\xi}\int_{\mE_x} |\xi\cdot\bn_x||[\veps_h]_x|^2 ds_x d\xi
\label{eq:er:3}\\
&
+\frac{1}{2}\int_{\mT_h^x}\int_{\mE_\xi} |(\bE_h+\xi\times\bB_h)\cdot\bn_\xi||[\veps_h]_\xi|^2 ds_\xi d\bx~.\notag
\end{align}
Next we will estimate the remaining terms in \eqref{eq:er:2}. Note
\begin{align*}
a_h(\Pi^k f, \bE_h, \bB_h; \veps_h)-a_h(f, \bE, \bB; \veps_h)=T_1+T_2~,
\end{align*}
where
\begin{align*}
T_1&=a_{h,1}(\Pi^k f; \veps_h)-a_{h,1}(f; \veps_h)=a_{h,1}(\zeta_h; \veps_h)~,\\
T_2&=a_{h,2}(\Pi^k f, \bE_h, \bB_h; \veps_h)-a_{h,2}(f,\bE, \bB; \veps_h)~.
\end{align*}


\noindent{\sf Step 1:  estimate of $T_1$.} We start with
\begin{equation*}
T_1
=\int_{\mT_h}(\df_t \zeta_h) \veps_h d\bx d\xi- \int_{\mT_h} \zeta_h\xi\cdot\nabla_\bx \veps_h d\bx d\xi +\int_{\mT_h^\xi} \int_{\mE_x} \widetilde{\zeta_h \xi}\cdot [\veps_h]_x ds_x d\xi=T_{11}+T_{12}+T_{13}~.
\end{equation*}
It is easy to verify that $\df_t\Pi^k=\Pi^k\df_t$, and therefore $\df_t \zeta_h=\Pi^k (\df_t f)-(\df_t f)$. With the approximation result of  Lemma \ref{lem:appr}, we have
\begin{equation}
\label{eq:est:T11}
|T_{11}|=\left|\int_{\mT_h}(\df_t \zeta_h) \veps_h d\bx d\xi \right|
\leq ||\df_t \zeta_h||_{0,\Omega} ||\veps_h||_{0,\Omega}\lesssim h^{k+1} ||\df_t f||_{k+1,\Omega}||\veps_h||_{0,\Omega}~.
\end{equation}
Next,  let $\xi_0$ be the $L^2$ projection of the function $\xi$ onto the piecewise constant space with respect to $\mT_h^\xi$, then
\begin{equation}
T_{12}=- \int_{\mT_h} \zeta_h(\xi-\xi_0)\cdot\nabla_\bx \veps_h d\bx d\xi- \int_{\mT_h} \zeta_h\xi_0\cdot\nabla_\bx \veps_h d\bx d\xi~.\label{eq:T12}
\end{equation}
Since $\xi_0\cdot\nabla_\bx \veps_h\in \mG_h^k$ and $\zeta_h=\Pi^k f-f$ with $\Pi^k$ being the $L^2$ projection onto $\mG_h^k$, the second term in \eqref{eq:T12} vanishes. Hence
\begin{align}
|T_{12}|&\leq \int_{\mT_h} |\zeta_h(\xi-\xi_0)\cdot\nabla_\bx \veps_h| d\bx d\xi~,\notag\\
&\leq ||\xi-\xi_0||_{0,\infty,\Oxi}\sum_{\Kx\times\Kxi=K\in\mT_h}(h^{-1}_\Kx ||\zeta_h||_{0,K})(h_\Kx||\nabla_\bx\veps_h||_{0,K})~,\notag\\
&\lesssim ||\xi-\xi_0||_{0,\infty,\Oxi}\sum_{\Kx\times\Kxi=K\in\mT_h}h_{K}^{k+1}h^{-1}_\Kx ||f||_{k+1,K} ||\veps_h||_{0,K}~,\notag\\
&\lesssim h_\xi ||\xi||_{1,\infty,\Oxi} h^k ||f||_{k+1,\Omega} ||\veps_h||_{0,\Omega}~,\notag\\
&\lesssim h^{k+1} ||f||_{k+1,\Omega} ||\veps_h||_{0,\Omega}~.\label{eq:est:T12}
\end{align}
The third inequality above uses the approximating result of   Lemma \ref{lem:appr} and the inverse inequality of Lemma \ref{lem:inverse}.  The fourth inequality uses an approximation result similar to the last one of  Lemma \ref{lem:appr}, and $\frac{h_\xi}{h_{x,\min}}$ being uniformly bounded by $\sigma_0$ when the mesh is refined.

Next,
\begin{align*}
T_{13}&=\int_{\mT_h^\xi}\int_{\mE_x} \left(\{\zeta_h\}_x\xi+\frac{|\xi\cdot\bn_x|}{2}[\zeta_h]_x\right)\cdot [\veps_h]_x  ds_xd\xi~,\\
&=\int_{\mT_h^\xi}\int_{\mE_x} \left(\{\zeta_h\}_x(\xi\cdot\hat{\bn}_x)\hat{\bn}_x+\frac{|\xi\cdot\bn_x|}{2}[\zeta_h]_x\right)\cdot [\veps_h]_x  ds_xd\xi~,
\end{align*}
where $\hat{\bn}_x$ is the unit normal vector of an edge in $\mE_x$ with either orientation, that is $\hat{\bn}_x=\bn_x$, or $-\bn_x$. Then, 
\begin{align}
|T_{13}|&\leq
\int_{\mT_h^\xi}\int_{\mE_x} \left(|\xi\cdot\bn_x|(|\{\zeta_h\}_x| +\frac{|[\zeta_h]_x|}{2})\right)\cdot |[\veps_h]_x|  ds_x d\xi\notag\\
&\leq
\left( \int_{\mT_h^\xi}\int_{\mE_x} 2 (|\{\zeta_h\}_x|^2+(\frac{|[\zeta_h]_x|}{2})^2) |\xi\cdot\bn_x|ds_xd\xi \right)^{1/2}
 \left(\int_{\mT_h^\xi}\int_{\mE_x}  |\xi\cdot\bn_x|
 |[\veps_h]_x|^2  ds_xd\xi\right)^{1/2}\notag\\
 &=
\left( \int_{\mT_h^\xi}\int_{\mE_x} 2 |\xi\cdot\bn_x| |\{\zeta_h^2\}_x| ds_xd\xi \right)^{1/2}
 \left(\int_{\mT_h^\xi}\int_{\mE_x} |\xi\cdot\bn_x|
 |[\veps_h]_x|^2  ds_xd\xi\right)^{1/2}\notag\\
 &\lesssim ||\xi||^{1/2}_{0,\infty,\Oxi} ||\zeta_h||_{0,\mT_h^\xi\times {\mE_x}}  \left(\int_{\mT_h^\xi}\int_{\mE_x} |\xi\cdot\bn_x|
 |[\veps_h]_x|^2  ds_xd\xi\right)^{1/2}\notag\\
 & \lesssim h^{k+\frac{1}{2}} ||f||_{k+1,\Omega} \left(\int_{\mT_h^\xi}\int_{\mE_x} |\xi\cdot\bn_x|
 |[\veps_h]_x|^2  ds_xd\xi\right)^{1/2}~.\label{eq:est:T13}
\end{align}
The approximation results of Lemma \ref{lem:appr} are used for the last inequality.

\bigskip
\noindent{\sf Step 2: estimate of $T_2$.} Note, 
\begin{align*}
T_2&=a_{h,2}(\Pi^k f, \bE_h, \bB_h; \veps_h)-a_{h,2}(f,\bE, \bB; \veps_h)\\
   &=a_{h,2}(\zeta_h,  \bE_h, \bB_h; \veps_h)+a_{h,2}(f,\bE_h, \bB_h; \veps_h)-a_{h,2}(f,\bE, \bB; \veps_h)=T_{21}+T_{22}+T_{23}~,\\
\end{align*}
with
\begin{align*}
T_{21}&=- \int_{\mT_h} \zeta_h (\bE_h+\xi\times\bB_h)\cdot\nabla_\xi \veps_h d\bx d\xi, \qquad
T_{22}=\int_{\mT_h^x} \int_{\mE_\xi} \widetilde{\zeta_h (\bE_h+\xi\times\bB_h)}\cdot [\veps_h]_\xi ds_\xi dx,\\
T_{23}&=a_{h,2}(f,\bE_h, \bB_h; \veps_h)-a_{h,2}(f,\bE, \bB; \veps_h)~.
\end{align*}
For $T_{21}$, we proceed as for the estimate of  $T_{12}$. Let $\bE_0=\bPi_x^0\bE$, $\bB_0=\bPi_x^0\bB$ be the $L^2$ projection of $\bE$, $\bB$, respectively, onto the piecewise constant vector space with respect to $\mT_h^x$, then
\begin{align*}
\int_{\mT_h} \zeta_h (\bE_h+\xi\times\bB_h)\cdot\nabla_\xi \veps_h d\bx d\xi
=&\int_{\mT_h} \zeta_h (\bE_h-\bE_0+\xi\times (\bB_h-\bB_0))\cdot\nabla_\xi \veps_h d\bx d\xi\\
&+\int_{\mT_h} \zeta_h (\bE_0+\xi\times\bB_0)\cdot\nabla_\xi \veps_h d\bx d\xi~,
\end{align*}
and the second term above vanishes due to $(\bE_0+\xi\times\bB_0)\cdot\nabla_\xi \veps_h\in\mG_h^k$, and therefore
\begin{align*}
&|\int_{\mT_h} \zeta_h (\bE_h+\xi\times\bB_h)\cdot\nabla_\xi \veps_h d\bx d\xi|
\leq \int_{\mT_h} |\zeta_h (\bE_h-\bE_0+\xi\times(\bB_h-\bB_0))\cdot\nabla_\xi \veps_h| d\bx d\xi~,\\
&\leq(||\bE_h-\bE_0+\xi\times(\bB_h-\bB_0)||_{0,\infty,\Omega})\sum_{\Kx\times\Kxi=K\in\mT_h}(h^{-1}_\Kxi ||\zeta_h||_{0,K})(h_\Kxi||\nabla_\xi\veps_h||_{0,K})~,\\
&\lesssim (||\bE_h-\bE_0||_{0,\infty,\Ox}+||(\bB_h-\bB_0)||_{0,\infty,\Ox}) \sum_{\Kx\times\Kxi=K\in\mT_h}h_{K}^{k+1}h^{-1}_\Kxi||f||_{k+1,K} ||\veps_h||_{0,K}~,\\
&\lesssim h^k||f||_{k+1,\Omega}(||\bveps^E_h||_{0,\infty,\Ox}+||\bveps^B_h||_{0,\infty,\Ox}+
||\bPi^k_x\bE-\bE_0 ||_{0,\infty,\Ox}+||\bPi^k_x\bB-\bB_0 ||_{0,\infty,\Ox})||\veps_h||_{0,\Omega}~.
\end{align*}
Note that $\bPi_x^k\bE-\bE_0=\bPi_x^k(\bE-\bE_0)$, and $\bPi_x^k$ is bounded in any $L^p$-norm ($1\leq p\leq \infty$) \cite{Crouzeix-Thomee:1987, Ayuso2010},
then
\begin{equation*}
 ||\bPi_x^k\bE-\bE_0||_{0,\infty,\Ox}\lesssim ||\bE-\bE_0||_{0,\infty,\Ox}\lesssim h_x||\bE||_{1,\infty,\Ox}~,
\end{equation*}
and similarly $||\bPi_x^k\bB-\bB_0||_{0,\infty,\Ox}\lesssim h_x||\bB||_{1,\infty,\Ox}$.
Hence
\begin{align}
&\left|\int_{\mT_h} \zeta_h (\bE_h+\xi\times\bB_h)\cdot\nabla_\xi \veps_h d\bx d\xi\right|
\label{eq:est:T21}\\
\lesssim & h^k||f||_{k+1,\Omega}(||\bveps^E_h||_{0,\infty,\Ox}+||\bveps^B_h||_{0,\infty,\Ox}+h_x(||\bE||_{1,\infty,\Ox}+||\bB||_{1,\infty,\Ox}))||\veps_h||_{0,\Omega}~.\notag
\end{align}

For $T_{22}$, we follow the  estimate of $T_{13}$.  Note that $\bE_h$ and $\bB_h$ only depends on $\bx$, and $\xi$ is continuous,
\begin{align*}
&|\int_{\mT_h^x} \int_{\mE_\xi} \widetilde{\zeta_h (\bE_h+\xi\times\bB_h)}\cdot [\veps_h]_\xi ds_\xi dx|\\
&=| \int_{\mT_h^x} \int_{\mE_\xi} \left(\{\zeta_h(\bE_h+\xi\times\bB_h) \}_\xi+\frac{|(\bE_h+\xi\times\bB_h)\cdot\bn_\xi|}{2}[\zeta_h]_\xi\right)\cdot [\veps_h]_\xi
ds_\xi dx|~,\\
&=|\int_{\mT_h^x}\int_{\mE_\xi} \left(\{\zeta_h\}_\xi((\bE_h+\xi\times\bB_h)\cdot\hat{\bn}_\xi)\hat{\bn}_\xi+\frac{|(\bE_h+\xi\times\bB_h)\cdot\bn_\xi|}{2}[\zeta_h]_\xi\right)\cdot [\veps_h]_\xi  ds_\xi dx|~, \;\;\; \hat{\bn}_\xi={\bn}_\xi\;\textrm{or}\; -{\bn}_\xi\\
&\leq \int_{\mT_h^x}\int_{\mE_\xi} \left(|(\bE_h+\xi\times\bB_h)\cdot\bn_\xi|(|\{\zeta_h\}_\xi|+|\frac{[\zeta_h]_\xi}{2}|)\right)|[\veps_h]_\xi|  ds_\xi dx~,\\
&\leq \left(\int_{\mT_h^x}\int_{\mE_\xi} 2 |(\bE_h+\xi\times\bB_h)\cdot\bn_\xi| |\{\zeta_h^2\}|  ds_\xi dx\right)^{1/2}
\left(\int_{\mT_h^x}\int_{\mE_\xi} |(\bE_h+\xi\times\bB_h)\cdot\bn_\xi| |[\veps_h]|^2  ds_\xi dx\right)^{1/2}~. 
\end{align*}
In addition,
\begin{align*}
&\left(\int_{\mT_h^x}\int_{\mE_\xi} 2 |(\bE_h+\xi\times\bB_h)\cdot\bn_\xi| |\{\zeta_h^2\}|  ds_\xi dx\right)^{1/2}\\
&\lesssim ||\bE_h+\xi\times\bB_h||_{0,\infty,\Omega}^{1/2} ||\zeta_h||_{0,\mT_h^x\times {\mE_\xi}}\\
&\lesssim  ||\zeta_h||_{0,\mT_h^x\times {\mE_\xi}}(||\bE_h||_{0,\infty,\Ox}^{1/2}+||\bB_h||_{0,\infty,\Ox}^{1/2})\\
&\lesssim h^{k+\frac{1}{2}} ||f||_{k+1,\Omega} (||\bveps^E_h||_{0,\infty,\Ox}^{1/2}+||\bveps^B_h||_{0,\infty,\Ox}^{1/2}+||\bE||_{0,\infty,\Ox}^{1/2}+||\bB||_{0,\infty,\Ox}^{1/2})~, 
\end{align*}
and therefore
\begin{align}
T_{22}\lesssim h^{k+\frac{1}{2}} ||f||_{k+1,\Omega} &(||\bveps^E_h||_{0,\infty,\Ox}^{1/2}+||\bveps^B_h||_{0,\infty,\Ox}^{1/2}+||\bE||_{0,\infty,\Ox}^{1/2}+||\bB||_{0,\infty,\Ox}^{1/2})\label{eq:est:T22}\\
&\left(\int_{\mT_h^x}\int_{\mE_\xi} |(\bE_h+\xi\times\bB_h)\cdot\bn_\xi| |[\veps_h]|^2  ds_\xi dx\right)^{1/2}~.\notag
\end{align}

Finally, we   estimate $T_{23}$. Since $f$ is continuous in $\xi$, and $\nabla_\xi\cdot (\bE_h-\bE+\xi\times (\bB_h-\bB))=0$,
\begin{align*}
T_{23}&=a_{h,2}(f,\bE_h, \bB_h; \veps_h)-a_{h,2}(f,\bE, \bB; \veps_h)\\
&=- \int_{\mT_h} f (\bE_h-\bE+\xi\times(\bB_h-\bB))\cdot\nabla_\xi \veps_h d\bx d\xi+\int_{\mT_h^x} \int_{\mE_\xi} f (\bE_h-\bE+\xi\times(\bB_h-\bB))\cdot [\veps_h]_\xi  ds_\xi dx~,\\
&= \int_{\mT_h} \nabla_\xi f \cdot(\bE_h-\bE+\xi\times(\bB_h-\bB)) \veps_h d\bx d\xi~;
\end{align*}
therefore, 
\begin{align}
|T_{23}|&\leq ||\bE_h-\bE+\xi\times(\bB_h-\bB)||_{0,\Omega} |f|_{1,\infty,\Omega} ||\veps_h||_{0,\Omega}~,\notag\\
        &  \lesssim (||\bE_h-\bE||_{0,\Ox}+||(\bB_h-\bB)||_{0,\Ox}) |f|_{1,\infty,\Omega} ||\veps_h||_{0,\Omega} ~,\notag\\
&\lesssim  (||\bveps^E_h||_{0,\Ox}+||\bveps^B_h||_{0,\Ox}+||\bzeta^E_h||_{0,\Ox}+||\bzeta^B_h||_{0,\Ox})|f|_{1,\infty,\Omega} ||\veps_h||_{0,\Omega} ~,\notag\\
&\lesssim (||\bveps^E_h||_{0,\Ox}+||\bveps^B_h||_{0,\Ox}+h^{k+1}_x(||\bE||_{k+1,\Ox}+||\bB||_{k+1,\Ox}))|f|_{1,\infty,\Omega} ||\veps_h||_{0,\Omega}~.
\label{eq:est:T23}
\end{align}

Now we  combine the estimates of \eqref{eq:est:T11} and \eqref{eq:est:T12}-\eqref{eq:est:T23}, and get the result of   Lemma \ref{lem:vepsh}.

\subsection{Proof of Lemma \ref{lem:vepshEB}}
\label{sec:err:part2}

Since the proposed method is consistent, the error equation is related to the Maxwell solver,
\begin{equation}
b_h(\bE-\bE_h,  \bB-\bB_h; \bU, \bV)=l_h(\bJ-\bJ_h, \bU),\qquad\forall\; \bU, \bV\in\mU_h^k~.
\label{eq:er:1:Max}
\end{equation}
Taking the test functions in \eqref{eq:er:1:Max} to be $\bU=\bveps_h^E$ and $\bV=\bveps_h^B$  gives
\begin{equation}
b_h(\bveps_h^E, \bveps_h^B; \bveps_h^E, \bveps_h^B)= b_h(\bzeta_h^E, \bzeta_h^B; \bveps_h^E, \bveps_h^B)+l_h(\bJ-\bJ_h, \bveps_h^E)~.
\label{eq:er:2:Max}
\end{equation}
Following the same lines of {\sf Step 2} in the proof of Lemma \ref{lem:EneC},
\begin{equation}
b_h(\bveps_h^E, \bveps_h^B; \bveps_h^E, \bveps_h^B)=
\frac{1}{2}\frac{d}{dt}\int_{\mT_h^x} \left(|\bveps_h^E|^2+|\bveps_h^B|^2\right) d\bx
 +\frac{1}{2}\int_{\mE_x}\left( |[\bveps_h^E]_\tau|^2+|[\bveps_h^B]_\tau|^2 \right) ds_x~.
 \label{eq:est:EB:1}
\end{equation}

It remains to estimate the two terms on the right side of \eqref{eq:er:2:Max},
\begin{align}
&b_h(\bzeta_h^E, \bzeta_h^B; \bveps_h^E, \bveps_h^B)\notag\\
=&\int_{\mT_h^x}\df_t\bzeta^E_h\cdot\bveps_h^E d\bx
-\int_{\mT_h^x}\bzeta^B_h\cdot\nabla\times\bveps_h^E d\bx-\int_{\mE_x}\widetilde{\bzeta^B_h}\cdot [\bveps_h^E]_\tau ds_x\notag\\
&+\int_{\mT_h^x}\df_t\bzeta^B_h\cdot\bveps_h^B d\bx
+\int_{\mT_h^x}\bzeta^E_h\cdot\nabla\times\bveps_h^B d\bx+\int_{\mE_x}\widetilde{\bzeta^E_h}\cdot[\bveps_h^B]_\tau ds_x~,\label{eq:501}\\
=&-\int_{\mE_x}\widetilde{\bzeta^B_h}\cdot [\bveps_h^E]_\tau ds_x+\int_{\mE_x}\widetilde{\bzeta^E_h}\cdot[\bveps_h^B]_\tau ds_x~,\notag\\
\leq & \left(\int_{\mE_x} |\widetilde{\bzeta^B_h}|^2+|\widetilde{\bzeta^E_h}|^2  ds_x\right)^{1/2}\left(\int_{\mE_x} |[\bveps_h^E]_\tau|^2+|[\bveps_h^B]_\tau|^2  ds_x\right)^{1/2}~,\notag\\
\lesssim & \sum_{K_x\in\mT_h^x}(||\bzeta^E_h||_{0,\partial \Kx}+||\bzeta^B_h||_{0,\partial \Kx})\left(\int_{\mE_x} |[\bveps_h^E]_\tau|^2+|[\bveps_h^B]_\tau|^2  ds_x\right)^{1/2}~,\notag\\
\lesssim & h_x^{k+\frac{1}{2}}(||\bE||_{k+1,\Ox}+||\bB||_{k+1,\Ox})\left(\int_{\mE_x} |[\bveps_h^E]_\tau|^2+|[\bveps_h^B]_\tau|^2  ds_x\right)^{1/2}~.\notag
\end{align}
All of the volume integrals of  \eqref{eq:501} vanish due to  $\partial_t\bPi^k_x=\bPi^k_x\partial_t$ and $\bveps_h^E, \bveps_h^B, \nabla\times\bveps_h^E, \nabla\times\bveps_h^B\in \mU_h^k$. And,  for the last two inequalities, the definition of the numerical fluxes are used together with the approximation results of  Lemma \ref{lem:appr}. Finally,
\begin{align}
|l_h(\bJ-\bJ_h; \bveps_h^E)|&=|\int_{\mT_h^x}(\bJ-\bJ_h)\cdot\bveps_h^E d\bx|~,\notag\\
& \leq ||\bJ-\bJ_h||_{0,\Ox}||\bveps_h^E||_{0,\Ox}= ||\int_{\mT_\xi}(f-f_h)\xi d\xi||_{0,\Ox}||\bveps_h^E||_{0,\Ox}~,\notag\\
& \leq ||f-f_h||_{0,\Omega} ||\xi||_{0,\Oxi} ||\bveps_h^E||_{0,\Ox}~,\notag\\
& \lesssim (||\veps_h||_{0,\Omega}+||\zeta_h||_{0,\Omega})||\bveps_h^E||_{0,\Ox}\lesssim(||\veps_h||_{0,\Omega}+
h^{k+1}||f||_{k+1,\Omega})||\bveps_h^E||_{0,\Ox}~.
 \label{eq:est:EB:2}
\end{align} ~
Combining \eqref{eq:est:EB:1}-\eqref{eq:est:EB:2}, we  conclude Lemma \ref{lem:vepshEB}.

\subsection{Proof of Lemma \ref{lem:vepshEB:2}}
\label{sec:err:part3}

The proof proceeds in a manner similar to that of  Lemma \ref{lem:vepshEB} of Subsection \ref{sec:err:part2}. Based on the error equation \eqref{eq:er:1:Max},  related to the Maxwell solver with some specific test functions, we get \eqref{eq:er:2:Max}. With either the central or alternating flux of  \eqref{eq:flux:6}-\eqref{eq:flux:7}, we have 
\begin{equation*}
b_h(\bveps_h^E, \bveps_h^B; \bveps_h^E, \bveps_h^B)=
\frac{1}{2}\frac{d}{dt}\int_{\mT_h^x} \left(|\bveps_h^E|^2+|\bveps_h^B|^2\right) d\bx~.
\end{equation*}
The same estimate as that of  \eqref{eq:est:EB:2} can be obtained for the second term on the right of \eqref{eq:er:2:Max}. To estimate the first one,
\begin{align}
&b_h(\bzeta_h^E, \bzeta_h^B; \bveps_h^E, \bveps_h^B)\notag\\
=&\int_{\mT_h^x}\df_t\bzeta^E_h\cdot\bveps_h^E d\bx
-\int_{\mT_h^x}\bzeta^B_h\cdot\nabla\times\bveps_h^E d\bx-\int_{\mE_x}\widetilde{\bzeta^B_h}\cdot [\bveps_h^E]_\tau ds_x\notag\\
&+\int_{\mT_h^x}\df_t\bzeta^B_h\cdot\bveps_h^B d\bx
+\int_{\mT_h^x}\bzeta^E_h\cdot\nabla\times\bveps_h^B d\bx+\int_{\mE_x}\widetilde{\bzeta^E_h}\cdot[\bveps_h^B]_\tau ds_x\label{eq:502}\\
=&-\int_{\mE_x}\widetilde{\bzeta^B_h}\cdot [\bveps_h^E]_\tau ds_x+\int_{\mE_x}\widetilde{\bzeta^E_h}\cdot[\bveps_h^B]_\tau ds_x\notag\\
\leq & \left(\sum_{e\in\mE_x}\int_e h_\Kx^{-1}(|\widetilde{\bzeta^B_h}|^2+|\widetilde{\bzeta^E_h}|^2)  ds_x\right)^{1/2}\left(\sum_{e\in\mE_x}\int_e h_\Kx(|[\bveps_h^E]_\tau|^2+|[\bveps_h^B]_\tau|^2)  ds_x\right)^{1/2}\label{eq:503}\\
\lesssim & ~c(\delta) \left(\sum_{\Kx\in\mT_h^x}\int_{\partial\Kx} h_\Kx^{-1}(|\bzeta^B_h|^2+|\bzeta^E_h|^2)  ds_x\right)^{1/2}
\left(\sum_{\Kx\in\mT_h^x}\int_{\partial\Kx} h_\Kx(|\bveps_h^E|^2+|\bveps_h^B|^2)  ds_x\right)^{1/2}\label{eq:504}\\
\lesssim & ~c(\delta) \left(\sum_{\Kx\in\mT_h^x}h_\Kx^{2k}(||\bE||_{k+1,\Kx}^2+||\bB||_{k+1,\Kx}^2)\right)^{1/2}
\left(\sum_{\Kx\in\mT_h^x}(||\bveps_h^E||^2_{0,\Kx}+||\bveps_h^B||^2_{0,\Kx})  \right)^{1/2}\label{eq:505}\\
\lesssim & ~c(\delta) h_x^k(||\bE||_{k+1,\Ox}+||\bB||_{k+1,\Ox})\left(\int_{\mT_h^x} (|\bveps_h^E|^2+|\bveps_h^B|^2)  d\bx\right)^{1/2}~.\notag
\end{align}
As before, all volume integrals of  \eqref{eq:502} vanish due to  $\partial_t\bPi^k_x=\bPi^k_x\partial_t$ and $\bveps_h^E, \bveps_h^B, \nabla\times\bveps_h^E, \nabla\times\bveps_h^B\in \mU_h^k$. In \eqref{eq:503}, $\Kx$ is any element containing an edge $e$. To get \eqref{eq:504}, we use the definitions of the numerical fluxes, jumps, as well as the assumption \eqref{eq:meshratio} on the ratio of the neighboring mesh elements. Here $c(\delta)$ is a positive constant depending on  $\delta$. We obtain \eqref{eq:505} by applying an approximation result of  Lemma \ref{lem:appr} and an inverse inequality of Lemma \ref{lem:inverse}. From all the above, we  conclude Lemma \ref{lem:vepshEB:2}.

\section{Numerical results}
\label{numres}

In this section, we perform a detailed numerical study of the proposed scheme in the context of the streaming Weibel (SW) instability first analyzed in \cite{pegoraro96}.  The SW instability is  closely related to the Weibel instability of \cite{PhysRevLett.2.83}, but derives its free energy from transverse counter-streaming as opposed to  temperature anisotropy.  The SW instability and its Weibel counterpart have  been considered both  analytically and numerically in several papers (e.g.\ \cite{pegoraro96, califano1998ksw,califano1965ikp,califano2001ffm, lopa2009}) --  here we focus on comparison with the numerical results of  Califano \emph{et al.}  in \cite{califano1998ksw}.

%

We consider a reduced version of the Vlasov-Maxwell equations with one spatial variable, $x_2$,  and two velocity variables, $\xi_1$ and $\xi_2$,  The dependent  variables under consideration are the distribution function $f(x_2, \xi_1, \xi_2, t)$, a 2D electric field $\textbf{E}=(E_1(x_2, t),
E_2(x_2, t), 0)$ and a 1D magnetic field $\textbf{B}=(0, 0, B_3(x_2, t))$, and the reduced 
 Vlasov-Maxwell system is 
\begin{align}
f_t &+ \xi_2 f_{x_2} + (E_1 + \xi_2 B_3)f_{\xi_1} + (E_2 -
\xi_1 B_3 )f_{\xi_2} = 0~, \\
\frac{\df B_3}{\df t} &=  \frac{\df E_1}{\df x_2}, \quad
 \frac{\df E_1}{\df t} =  \frac{\df
B_3}{\df x_2} -  j_1, \quad \frac{\df E_2}{\df t} =  -  j_2~,
\end{align}
where 
\beq j_1=\int_{-\infty}^{\infty} \int_{-\infty}^{\infty}
f(x_2, \xi_1, \xi_2, t) \xi_1 \,d\xi_1 d\xi_2,\quad j_2=\int_{-\infty}^{\infty} \int_{-\infty}^{\infty} f(x_2, \xi_1,
\xi_2, t) \xi_2 \,d\xi_1 d\xi_2~.
\eeq






The initial conditions are given by
\begin{align}
f(x_2, \xi_1, \xi_2, 0)&=\frac{1}{\pi \beta}
e^{- \xi_2^2 /\beta} [\delta e^{- (\xi_1-v_{0,1})^2 /\beta}
+(1-\delta) e^{- (\xi_1+v_{0,2})^2 /\beta} ],\\
E_1(x_2, \xi_1, \xi_2, 0)&=E_2(x_2, \xi_1, \xi_2, 0)=0, \qquad B_3(x_2, \xi_1, \xi_2, 0)=b \sin(k_0 x_2)~,
\end{align}
which for $b=0$ is an equilibrium state composed of counter-streaming beams propagating perpendicular to the direction of inhomogeneity.   Following \cite{califano1998ksw}, we trigger the instability by taking  $\beta=0.01$, $b=0.001$ (the amplitude of the initial perturbation to the magnetic field).  Here, $\Ox=[0, L_y]$, where $L_y=2 \pi/ k_0$, and we set  $\Oxi=[-1.2, 1.2]^2$.  Two different sets of  parameters  will be considered,  
\begin{eqnarray}
{\rm \underline{choice\  1}:} \  \ \delta&=&0.5, v_{0,1}=v_{0,2}=0.3, k_0=0.2 
\nonumber\\
 {\rm  \underline{choice\  2}:} \ \ \delta&=&1/6,  v_{0,1}=0.5, v_{0,2}=0.1,  k_0=0.2~.
 \nonumber
\end{eqnarray}
For comparison, these are chosen  to correspond to runs of \cite{califano1998ksw}. 

\smallskip
\textbf{Accuracy test:} The VM system is time reversible, and this provides a way to test the accuracy of our scheme. In particular, let $f(\bx, \xi, 0), \bE(\bx, 0), \bB(\bx, 0) $ denote  the initial conditions for the VM system  and $f(\bx, \xi, T),  \bE(\bx, T), \bB(\bx, T)$  the solution  at $t=T$.  If we choose  $f(\bx, -\xi, T), \bE(\bx, T), -\bB(\bx, T)$ as the initial condition  at $t=0$, then at $t=T$ we theoretically must  recover $f(\bx, -\xi, 0), \bE(\bx, 0), -\bB(\bx, 0)$. In Tables \ref{errorupwind}, \ref{errorcentral}, we show the $L^2$ errors and orders of the numerical solutions with  three flux choices for the Maxwell's equations:  the upwind flux, the central flux, and one of the alternating fluxes $\widetilde{\bE_h}=\bE_h^+$ and $\widetilde{\bB_h}=\bB_h^-$. The parameters are those of choice 1, with symmetric counter-streaming.   In the numerical simulations, the third order TVD Runge Kutta time discretization is used, with the CFL number $C_{\rm cfl}=0.19
 $ for  the upwind and central 
fluxes, and  $C_{\rm cfl}=0.12$ for the alternating flux in $P^1$ and $P^2$ cases. For $P^3$, we take $\triangle t=O( \triangle x^{4/3})$ to ensure that the spatial and temporal accuracy is of the same order. From Tables \ref{errorupwind}, \ref{errorcentral}, we observe that the schemes with the upwind and alternating fluxes achieve optimal $(k+1)$-th order accuracy in approximating the solution, while for odd $k$, the central flux gives suboptimal approximation of some of the solution components.

\begin{table} [htb]
\begin{center}

\caption {Upwind flux for Maxwell's equations, $L^2$ errors and orders. Run to T=5 and back to $T=10$.}

\begin{small}
\begin{tabular}{|c|c|c|c|c|c|c|}
\hline  \multirow{2}{*}{Space} &\multirow{2}{*}{} & Mesh=$20^3$ & \multicolumn{2}{|c|}{Mesh=$40^3$} & \multicolumn{2}{|c|}{Mesh=$80^3$} \\
\cline{3-7}
 & &error& error&order& error&order \\
\hline
\multirow{4}{*}{$\mG_h^{1}, \mU_h^1$} & $f$ &0.18E+00& 0.50E-01&1.82& 0.13E-01&1.96 \\
\cline{2-7}
 & $B_3$ &0.26E-05& 0.66E-06&2.01&0.16E-06  &2.01\\
\cline{2-7}
 & $E_1$ &0.21E-05& 0.68E-06&1.61& 0.19E-06&1.81 \\
 \cline{2-7}
 & $E_2$ &0.10E-05& 0.22E-06&2.23& 0.22E-07&3.29 \\
 \hline
 \multirow{4}{*}{$\mG_h^{2}, \mU_h^2$} & $f$ &0.56E-01& 0.77E-02 &2.87& 0.10E-02&2.92 \\
\cline{2-7}
 & $B_3$ & 0.23E-06& 0.26E-07&3.12&0.32E-08   &3.06\\
\cline{2-7}
 & $E_1$ &0.16E-06& 0.16E-07 &3.32& 0.14E-08&3.54  \\
 \cline{2-7}
 & $E_2$ &0.16E-06& 0.22E-07&2.90& 0.15E-08&3.91 \\
 \hline
 \multirow{4}{*}{$\mG_h^{3}, \mU_h^3$} & $f$ &0.12E-01& 0.10E-02&3.56& 0.70E-04&3.90 \\
\cline{2-7}
 & $B_3$ &0.97E-07& 0.23E-08&5.37&0.12E-09  &4.34\\
\cline{2-7}
 & $E_1$ &0.19E-07& 0.27E-09&6.16& 0.57E-11&5.54 \\
 \cline{2-7}
 & $E_2$ &0.14E-07& 0.79E-09&4.11& 0.16E-10&5.64 \\
 \hline
\end{tabular}
\end{small}
\label{errorupwind}
\end{center}
\end{table}

\begin{table} [htb]
\begin{center}
\caption {Central and alternating fluxes for Maxwell's equations, $L^2$ errors and orders. Run to T=5 and back to $T=10$.}

\begin{small}
\begin{tabular}{|c|c|c|c|c|c|c|c|c|c|c|c|}
\hline &&\multicolumn{5}{|c|}{Central} &\multicolumn{5}{|c|}{Alternating} \\
\hline
&\multirow{2}{*}{} & Mesh=$20^3$ & \multicolumn{2}{|c|}{Mesh=$40^3$} & \multicolumn{2}{|c|}{Mesh=$80^3$} & Mesh=$20^3$ & \multicolumn{2}{|c|}{Mesh=$40^3$} & \multicolumn{2}{|c|}{Mesh=$80^3$}\\
\cline{3-12}
 & &error& error&order& error&order&error& error&order& error&order  \\
\hline
\multirow{2}{*}{$\mG_h^{1} $} & $f$ &0.18E+00& 0.50E-01&1.82& 0.13E-01&1.96&0.18E+00& 0.50E-01&1.82& 0.13E-01&1.96 \\
\cline{2-12}
 & $B_3$ &0.13E-04& 0.85E-05&0.66&0.50E-05  &0.75&0.29E-05& 0.78E-06&1.90&0.22E-06  &1.83\\
\cline{2-12}
\multirow{2}{*}{$\mU_h^1$}  & $E_1$ &0.19E-05& 0.13E-05&0.51& 0.58E-06&1.17 &0.24E-06& 0.35E-07&2.74& 0.22E-08&3.99\\
 \cline{2-12}
 & $E_2$ &0.92E-06& 0.19E-06&2.26& 0.20E-07&3.24&0.10E-05& 0.22E-06&2.23& 0.22E-07&3.29 \\
 \hline
\multirow{2}{*}{$\mG_h^{2} $}  & $f$ & 0.56E-01 & 0.77E-02&2.87& 0.10E-02&2.92&0.56E-01& 0.77E-02&2.87& 0.10E-02&2.92 \\
\cline{2-12}
 & $B_3$ &0.28E-06& 0.28E-07&3.34&0.32E-08  &3.15&0.28E-06&  0.22E-07 &3.70&0.18E-08  &3.63\\
\cline{2-12}
\multirow{2}{*}{$\mU_h^2$} & $E_1$ &0.18E-07& 0.56E-09&5.00& 0.88E-11&5.99 &0.32E-07& 0.30E-09&6.72&0.11E-10&4.84  \\
 \cline{2-12}
 & $E_2$ &0.16E-06& 0.22E-07&2.90& 0.15E-08&3.91 &0.16E-06&  0.22E-07 &2.90& 0.15E-08&3.91  \\
 \hline
\multirow{2}{*}{$\mG_h^3$}& $f$ &0.12E-01& 0.10E-02 &3.56& 0.70E-04 &3.90&0.12E-01& 0.10E-02&3.56& 0.70E-04 &3.90 \\
\cline{2-12}
 & $B_3$ &0.10E-06& 0.44E-08&4.57&0.16E-09  &4.81&0.10E-06& 0.24E-08&5.42&0.12E-09  &4.36\\
\cline{2-12}
\multirow{2}{*}{$\mU_h^3$} & $E_1$ &0.46E-07&  0.82E-10 &9.12& 0.30E-10& 1.45&0.98E-08& 0.10E-09& 6.60& 0.90E-12&6.80 \\
 \cline{2-12}
 & $E_2$ &0.14E-07&  0.79E-09&4.12&  0.16E-10  &5.65 &0.14E-07&  0.79E-09&4.11& 0.16E-10&5.64 \\
 \hline
\end{tabular}
\end{small}
\label{errorcentral}
\end{center}
\end{table}

\smallskip

\textbf{Conservation properties:} The purpose here is to validate our theoretical result about conservation through two numerical examples,  the symmetric case and the non-symmetric case. We first use parameter choice 1 as in the Califano \textit{et al.} \cite{califano1998ksw}, the symmetric case  with three different fluxes for  Maxwell's equations.  The results are illustrated in Figure \ref{masste1} }.  In all the plots, we have rescaled the macroscopic quantities by the physical domain size. For all three fluxes, the mass (charge) is well conserved.  The largest relative error for the charge for  all three fluxes is smaller than $4\times10^{-10}$. As for the total energy, we could observe relatively larger decay in the total energy from the simulation with the upwind flux compared to the one with the other two fluxes. This is expected from the analysis in Section \ref{conserve}. In fact, the largest relative error for the total energy is bounded by $1\times10^{-4}$ for th
 e upwind flux, and bounded by $1\
times10^{-7}$ for central and alternating fluxes. 

As for momentum conservation, it is well known that the two species VM system conserves the following expression for the total linear momentum:
 \beq
 \mathbf{P}= \int \xi f\, d\xi d\mathbf{x} + \int \mathbf{E}\times\mathbf{B} \, d\mathbf{x}~, 
 \label{eq:mom}
 \eeq
where the first term represents the momentum in the  particles while the second that of  the electromagnetic field.    In fact, this is true for the full energy-momentum and angular momentum tensors \cite{PM85}.   Each  component of the spatial integrand of (\ref{eq:mom}), the components of the momentum density, satisfies a conservation law, a result that relies on both species being dynamic and one that relies on the  constraint equations (\ref{eq:max:4}) being satisfied.  However, in this paper we have fixed the constant ion background by charge neutrality and, consequently,  momentum is not conserved in general.  This lack of conservation does not appear to be widely known,  but it is known that the enforcement of constraints may or may not results in the loss of conservation \cite{MoLeb}.  For example, the single species Vlasov-Poisson system with a fixed constant ion background does indeed conserve momentum.  However, for the streaming Weibel application, it is not diffi
 cult to show that the following 
component is conserved:
\beq
 P_1= \int \xi_1 f\, d{\xi}_1 d\xi_2 dx_2 + \int E_2B_3\, dx_2~,
 \eeq
while the component $P_2$ is not.  Since conservation of $P_1$ relies on the constraint equations and since our computational algorithm does not enforce these constraints, conservation of $P_1$ serves as a measure of the goodness of our method in maintaining the initial satisfaction of the constraints.    {}From Figure \ref{masste1}, we see that all three flux formulations conserve $P_1$ relatively well, but, as expected,   there is  a large accumulating error in $P_2$, particularly for the alternating flux case.

 Similarly, for a general VM system without constraints, the following expression for the total angular momentum is conserved:
 \beq
 \mathbf{L}= \int \mathbf{x}\times \xi \, f\, d\xi d\mathbf{x} 
 + \int \mathbf{x}\times(\mathbf{E}\times\mathbf{B}) \, d\mathbf{x}~. 
 \eeq
 However, because the SW application  breaks symmetry, there is no relevant component of the angular momentum  that is conserved for this problem, but for a more general application one may want to track its conservation. 
 
\smallskip
 
\textbf{Comparison and interpretation:} 
In Figure \ref{keeme1}, we plot the time evolution of the kinetic, electric,  and magnetic energies.  In particular, we  plot the separate components defined  by  $K_1= \frac{1}{2}\int f \xi_1^2 d \xi_1 d \xi_2 dx_2$,  $K_2 = \frac{1}{2}\int f \xi_2^2 d \xi_1 d \xi_2 dx_2$,  ${\rm E}_1 = \frac{1}{2}\int E_1^2 dx_2$,  and ${\rm E}_2 = \frac{1}{2}\int E_2^2 dx_2$.    Figure (a) shows for choice 1 the transference of kinetic energy from one component to the other with a deficit converted into field energy.   This deficit is consistent with energy conservation, as evidenced by Figure \ref{masste1}.  Observe the magnetic and inductive electric fields grow initially at a linear growth rate (comparable to that of Table I of \cite{califano1998ksw}).  Saturation occurs when the electric and magnetic energies simultaneously peak at around $t=70$ in agreement with  \cite{califano1998ksw};  however,  in our case we achieve equipartition at the peak, which may be due to better resolution.
    Here we have also shown the 
longitudinal component $E_2$, not shown in \cite{califano1998ksw}, which in Figure (b) is seen to grow at twice the growth rate.  This behavior was anticipated in \cite{98jpp} in the context of a two-fluid model  and seen in kinetic VM computations of the usual Weibel instability \cite{lopa2009}.    It is due to wave coupling and a modulation of the electron density induced by the spatial modulation of $B^2_3$.  The growth at twice the growth rate of the magnetic field $B_3$  is seen in Figure (b), and the density modulation, including the expected spikes,  is seen in Figure  \ref{density1}.  We have also calculated  the first four Log Fourier modes  of the fields $E_1$, $E_2$, $B_3$, and these are shown  in Figure \ref{logfm1}. 
Here, the $n$-th Log Fourier mode for a function $W(x,t)$ \cite{Heath} is defined as
$$
logF\!M_n(t)=\log_{10} \left(\frac{1}{L} \sqrt{\left|\int_0^L W(x, t) \sin(knx) \, dx \right|^2 +
\left|\int_0^L W(x, t) \cos(knx)\,  dx \right|^2} \right).
$$
In Figures \ref{contour1}   we plot the 2D contours  of $f$ at selected locations  $x_2$ and time $t$,  when the upwind flux is used in the Maxwell solver.  The times chosen correspond to those for the density of Figure \ref{density1}, and we see that at late times considerable fine structure is present, which is consistent with the Log Fourier plots.   For completeness, we also include in Figure \ref{befield1}  plots of the electric and magnetic fields at the final time for our three  fluxes.

\begin{figure}[htb]
  \begin{center}
    \subfigure[ Mass]{\includegraphics[width=3in,angle=0]{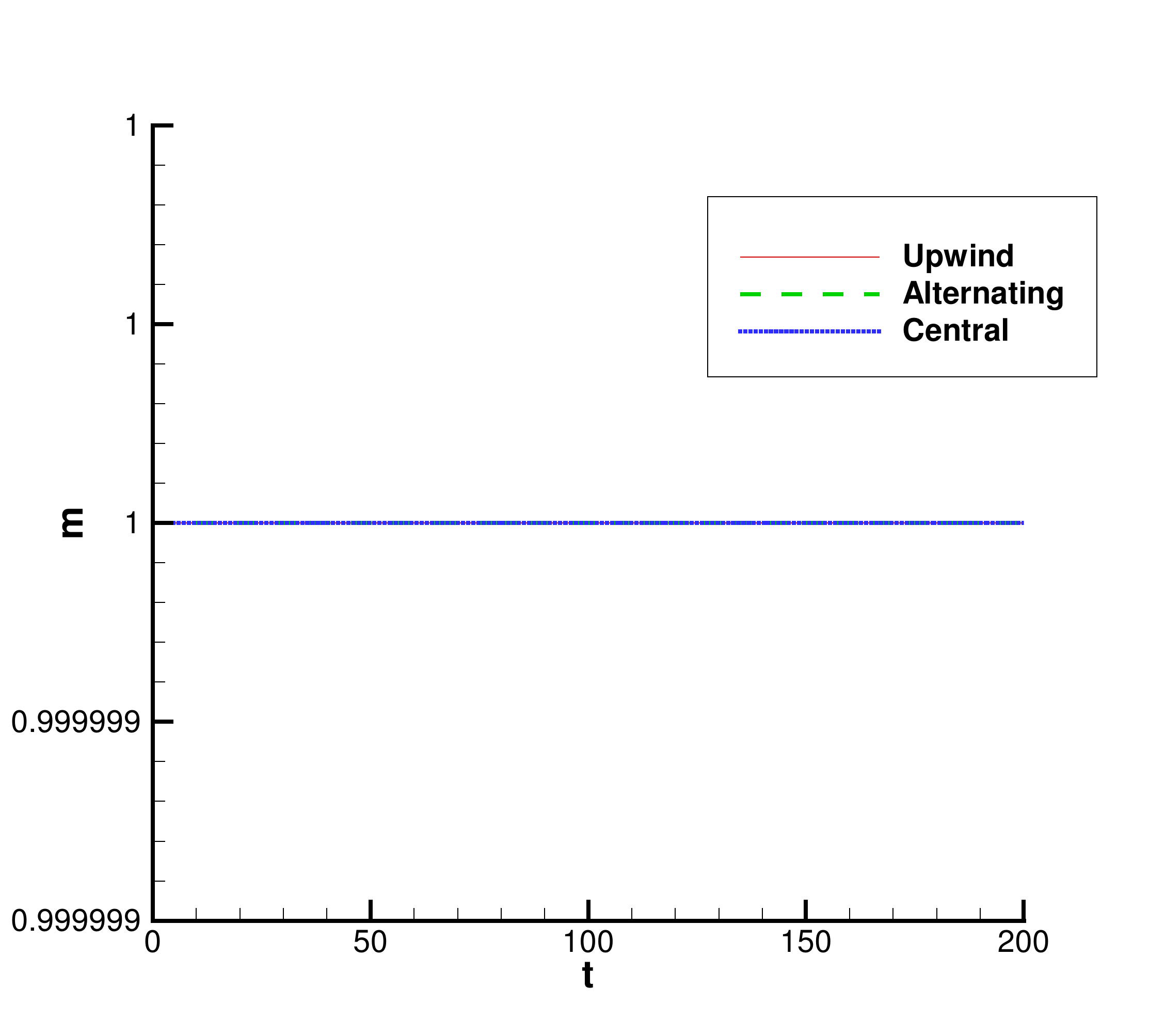}}
    \subfigure[Total energy]{\includegraphics[width=3in,angle=0]{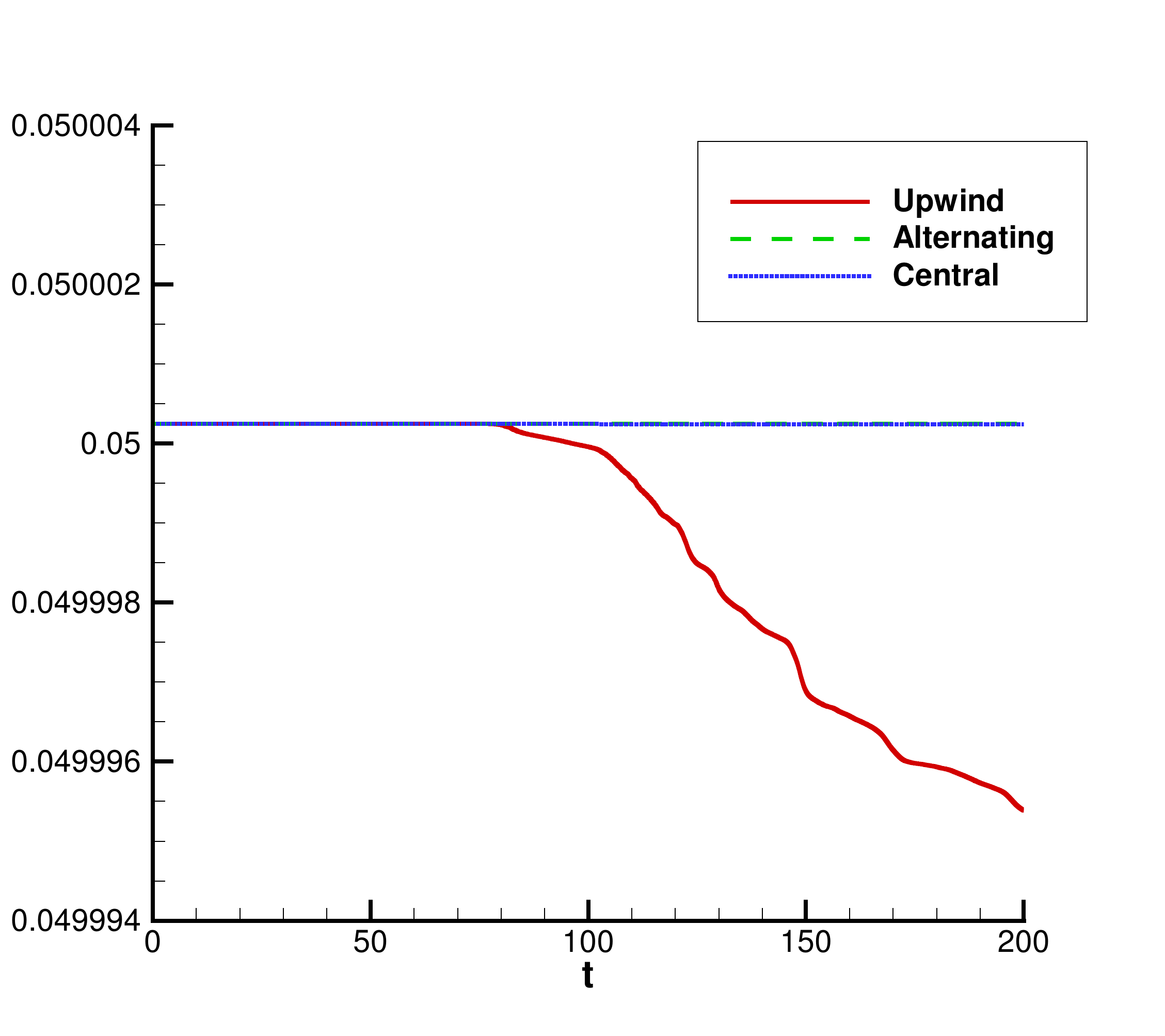}}
              \subfigure[Momentum $P_1$]{\includegraphics[width=3in,angle=0]{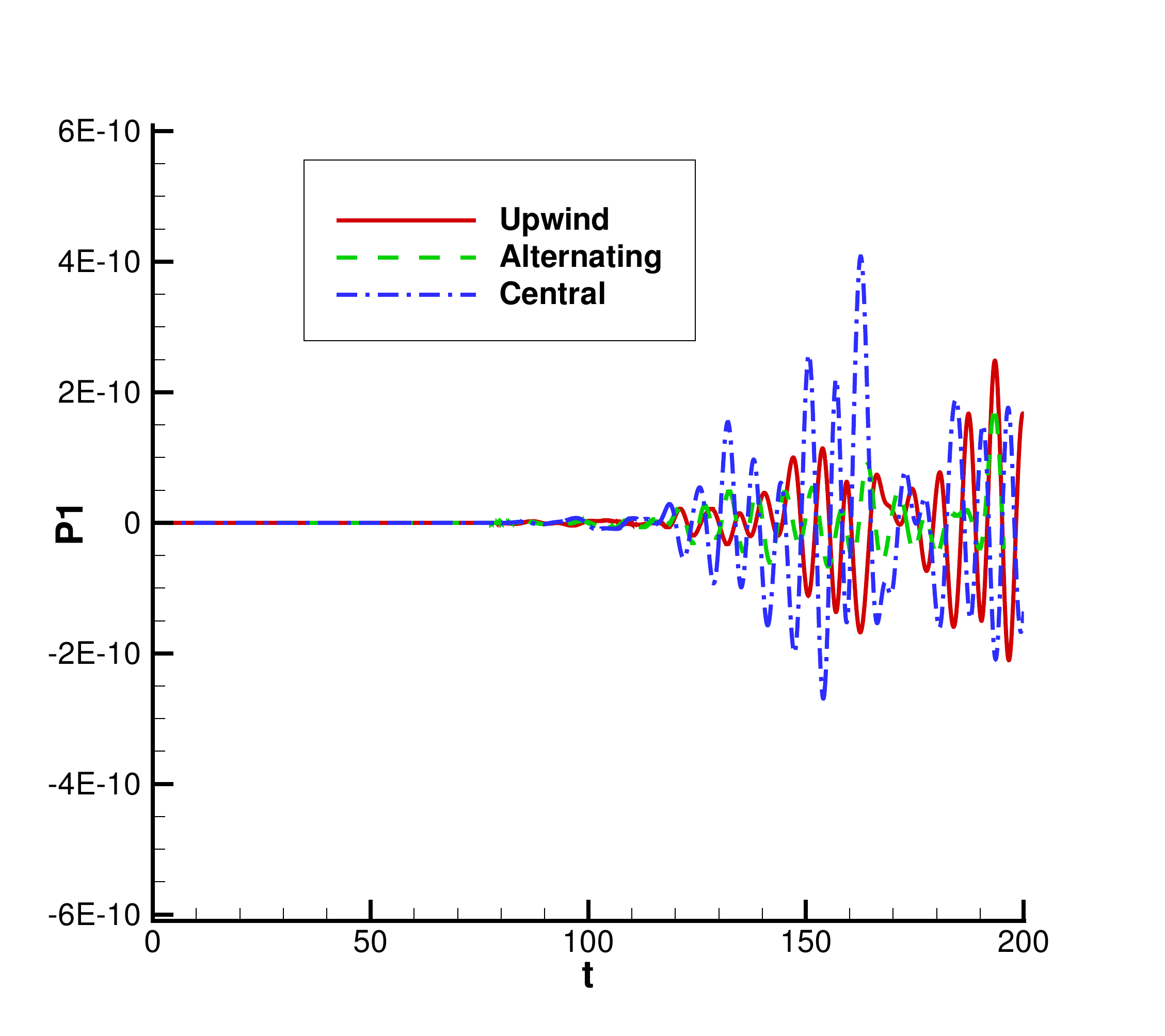}}
                       \subfigure[Momentum $P_2$]{\includegraphics[width=3in,angle=0]{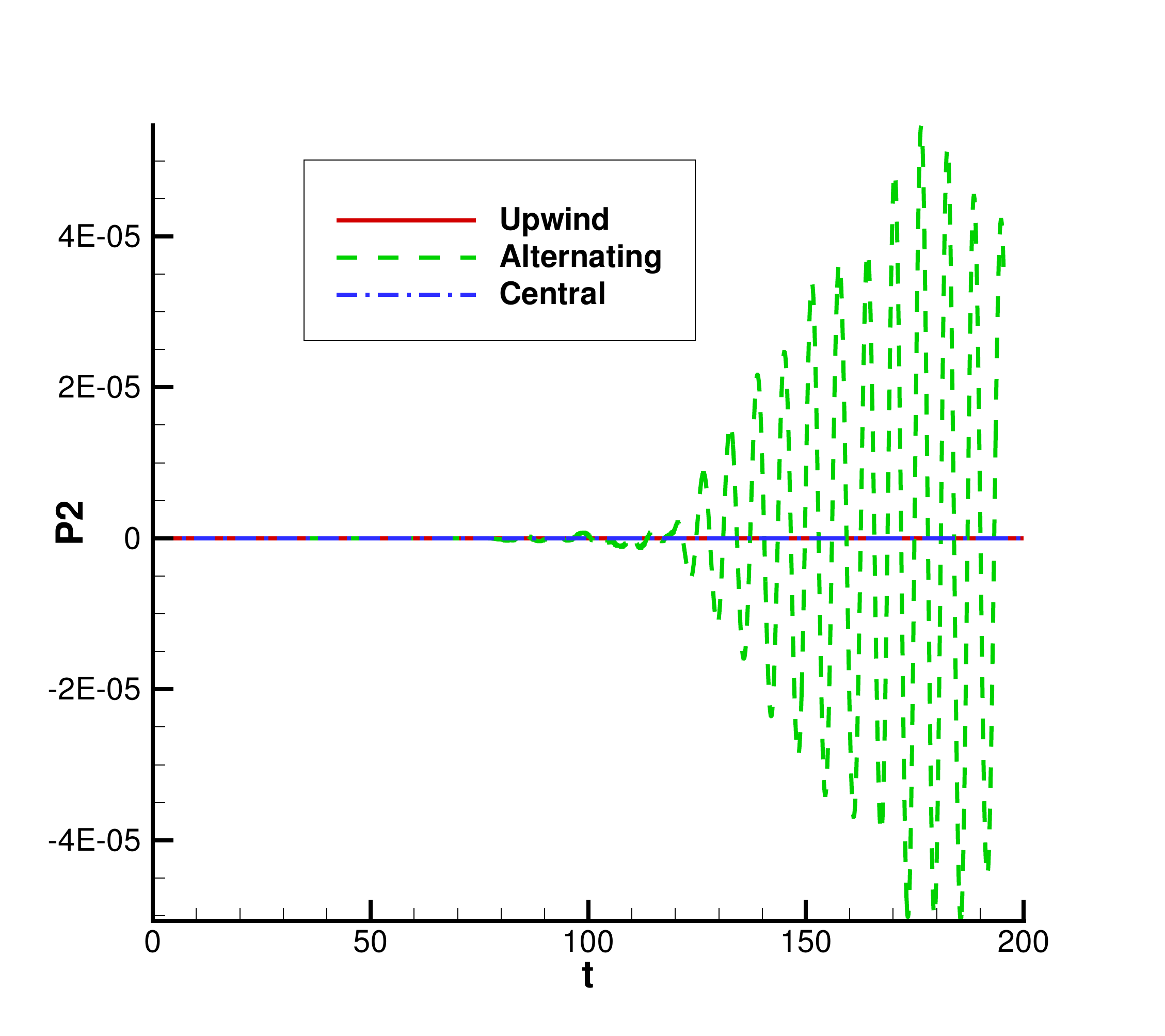}}
  \end{center}
  \caption{Streaming Weibel instability with parameter choice 1 as in Califano \textit{et al.} \cite{califano1998ksw} ($\delta=0.5, v_{0,1}=v_{0,2}=0.3, k_0=0.2$), the symmetric case. The mesh is $100^3$ with piecewise quadratic polynomials. Time evolution of mass, total energy, and momentum for  the three numerical fluxes for the Maxwell's equations.}
\label{masste1}
\end{figure}

\begin{figure}[htb]
  \begin{center}
    \subfigure[ Parameter Choice 1. Electric field, upwind flux]{\includegraphics[width=3in,angle=0]{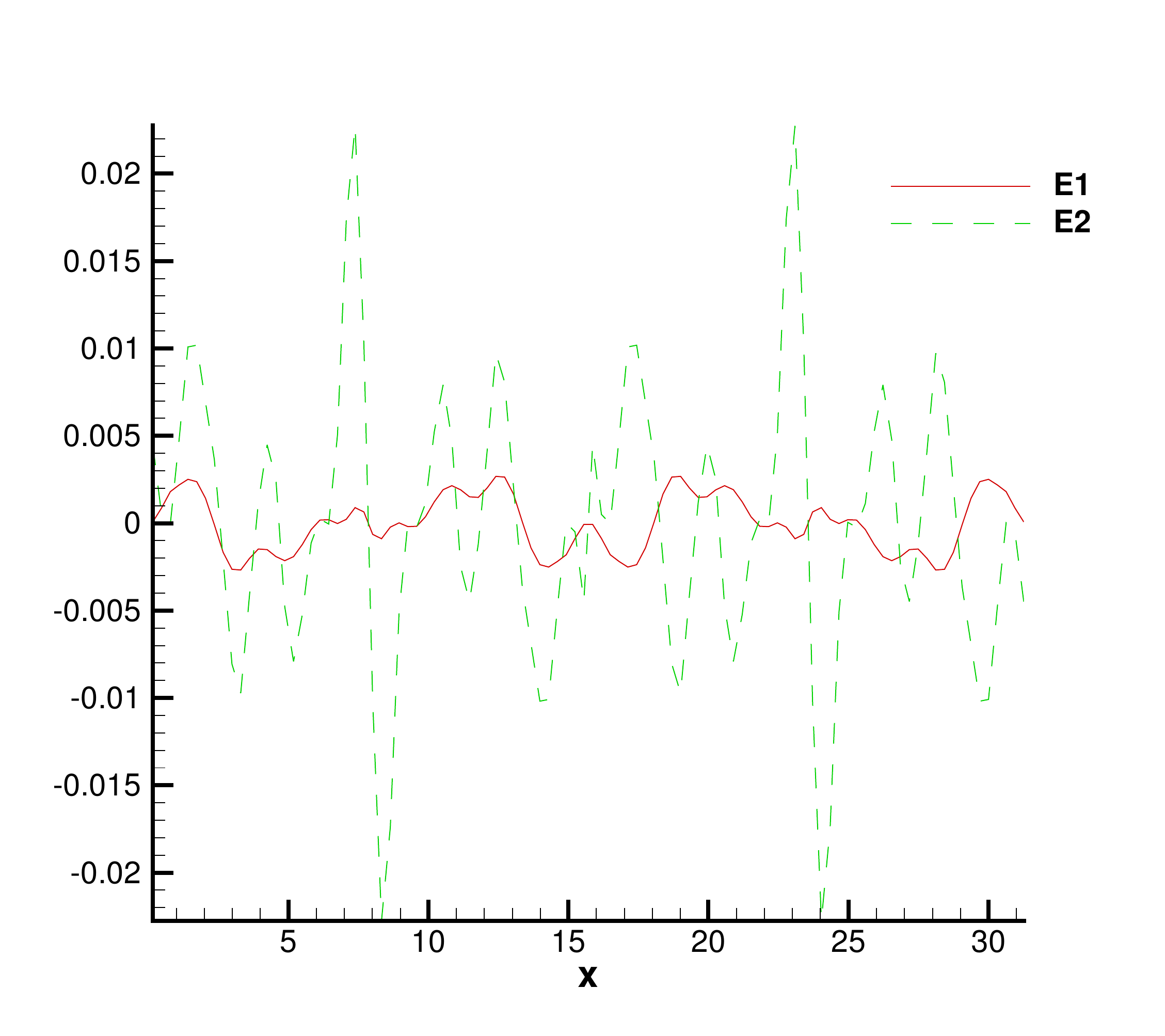}}
        \subfigure[ Parameter Choice 2.  Electric field, upwind flux]{\includegraphics[width=3in,angle=0]{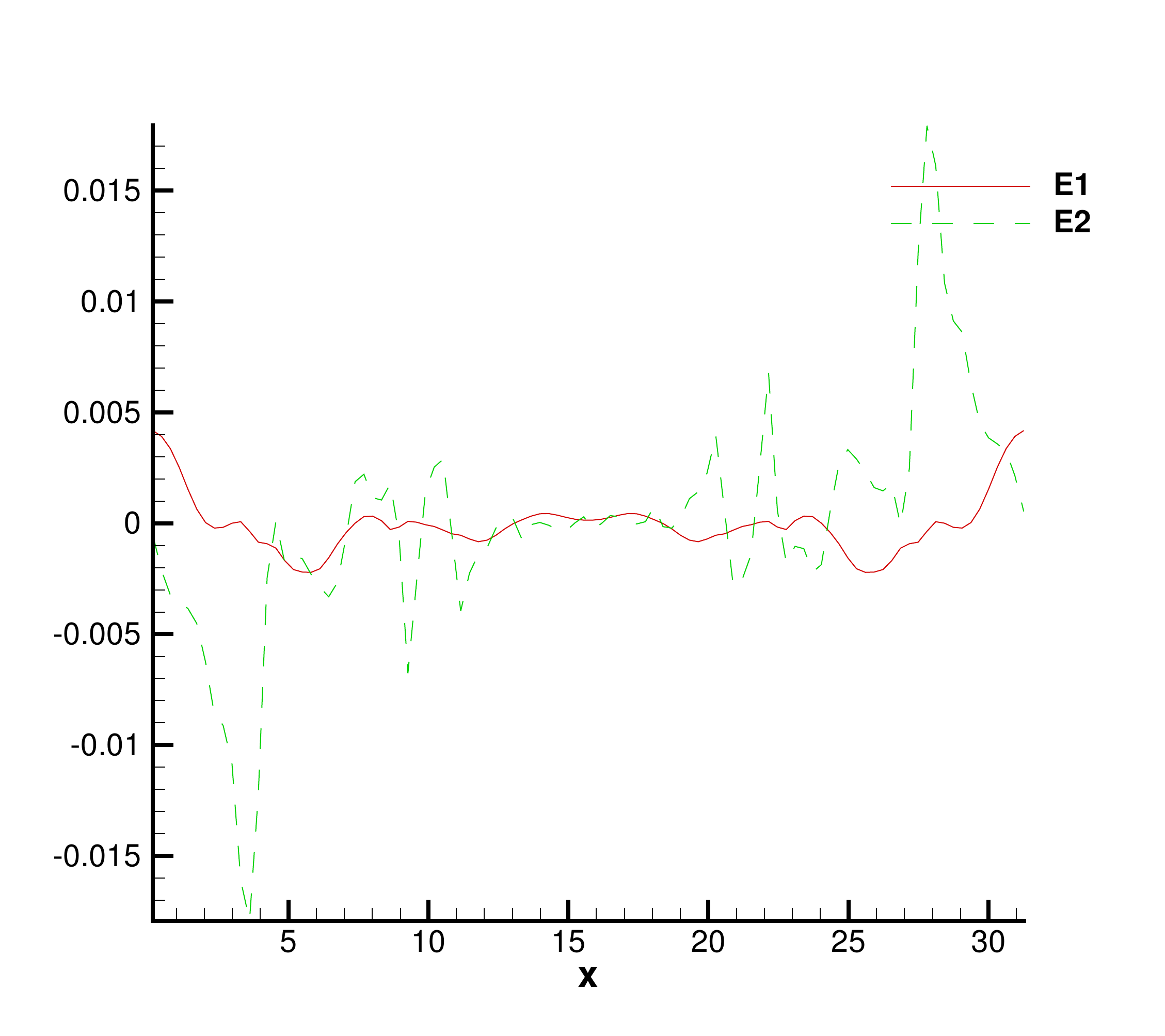}}
            \subfigure[Parameter Choice 1.  Magnetic field, upwind flux]{\includegraphics[width=3in,angle=0]{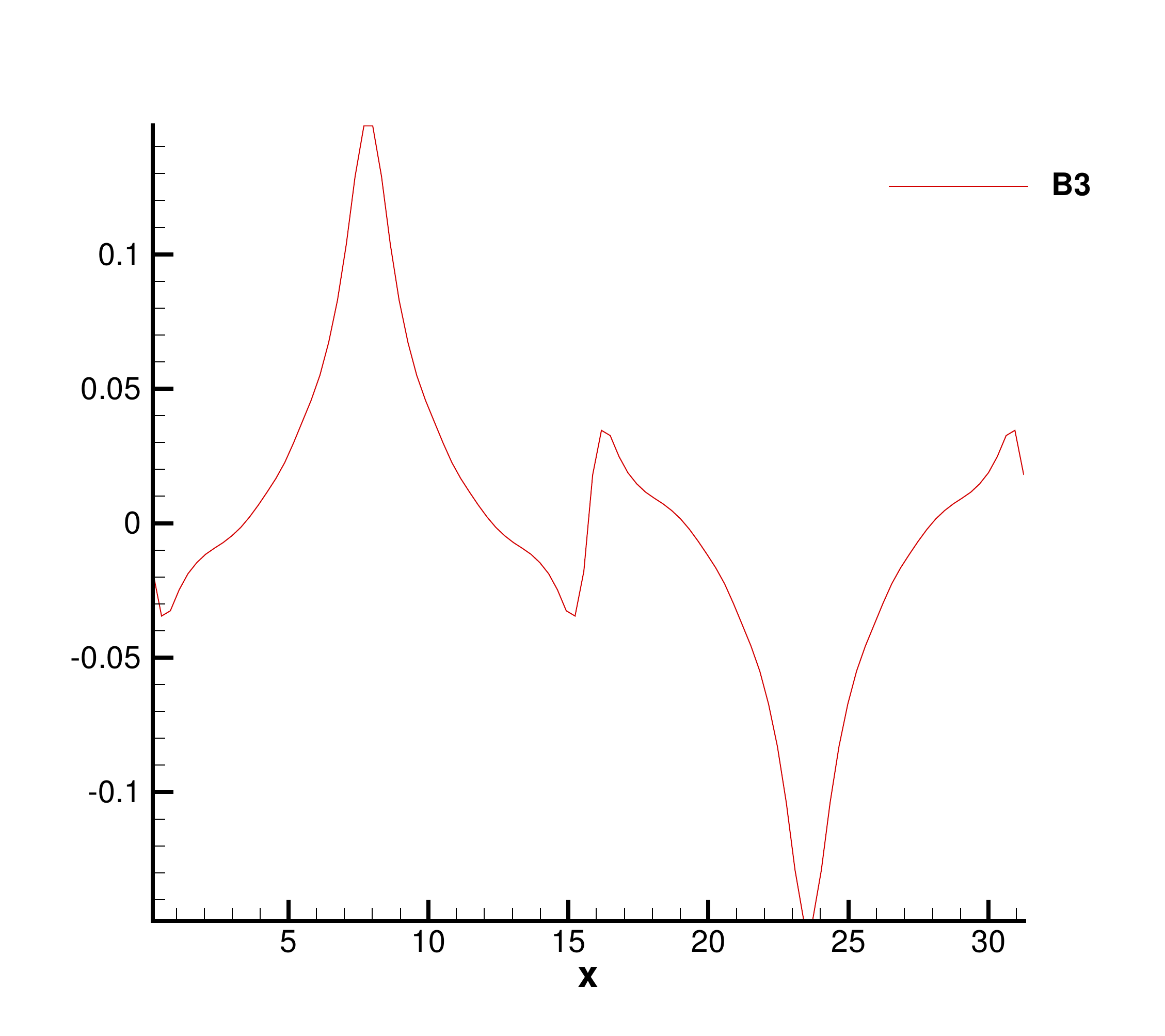}}
    \subfigure[Parameter Choice 2.  Magnetic field, upwind flux]{\includegraphics[width=3in,angle=0]{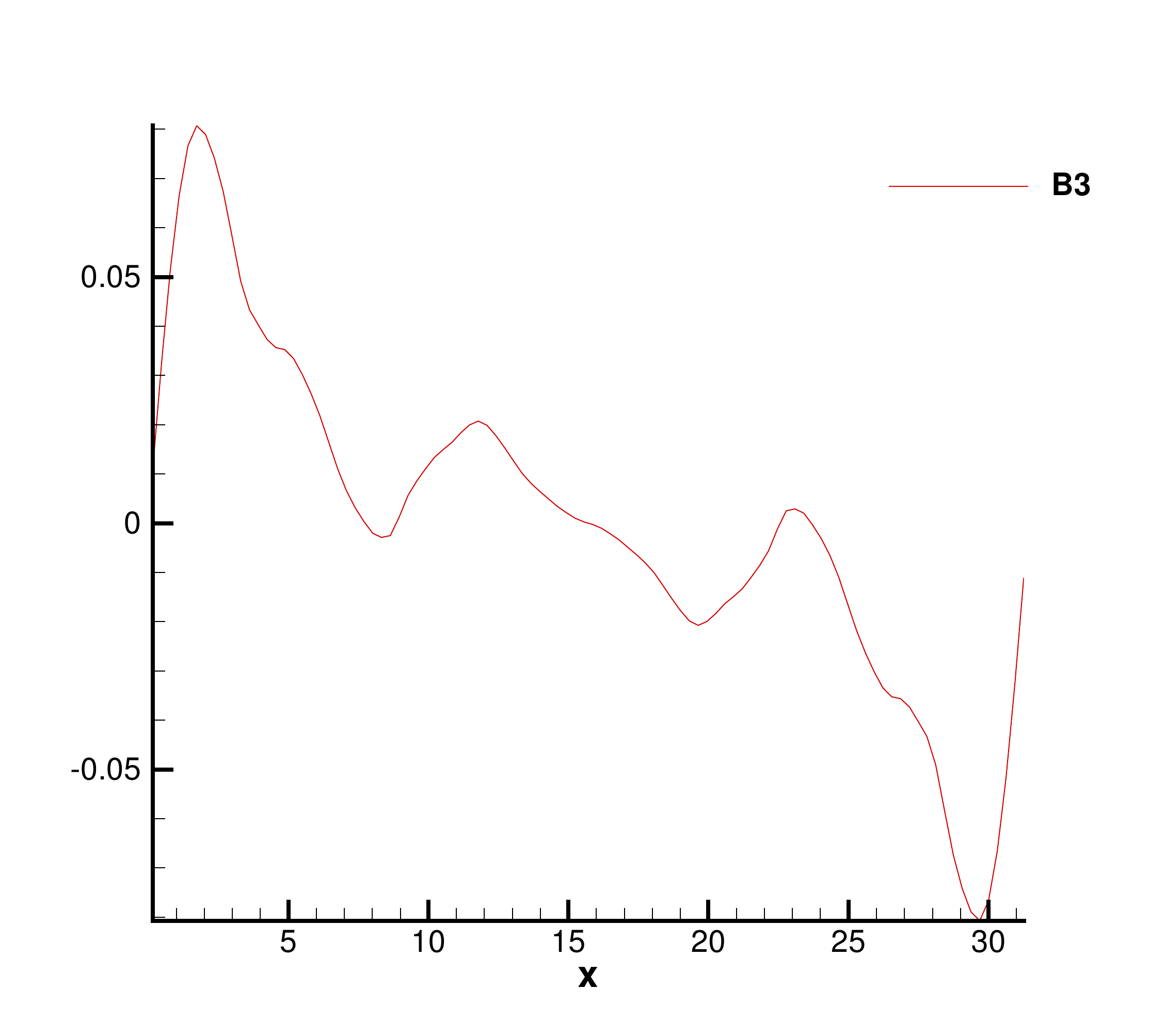}}
  \end{center}
    \caption{Streaming Weibel instability. The mesh is $100^3$ with piecewise quadratic polynomials. The electric and magnetic fields at $T=200$.}
\label{befield1}
\end{figure}

\begin{figure}[htb]
  \begin{center}
 \subfigure[ Parameter Choice 1. Kinetic energies]{\includegraphics[width=2.5in,angle=0]{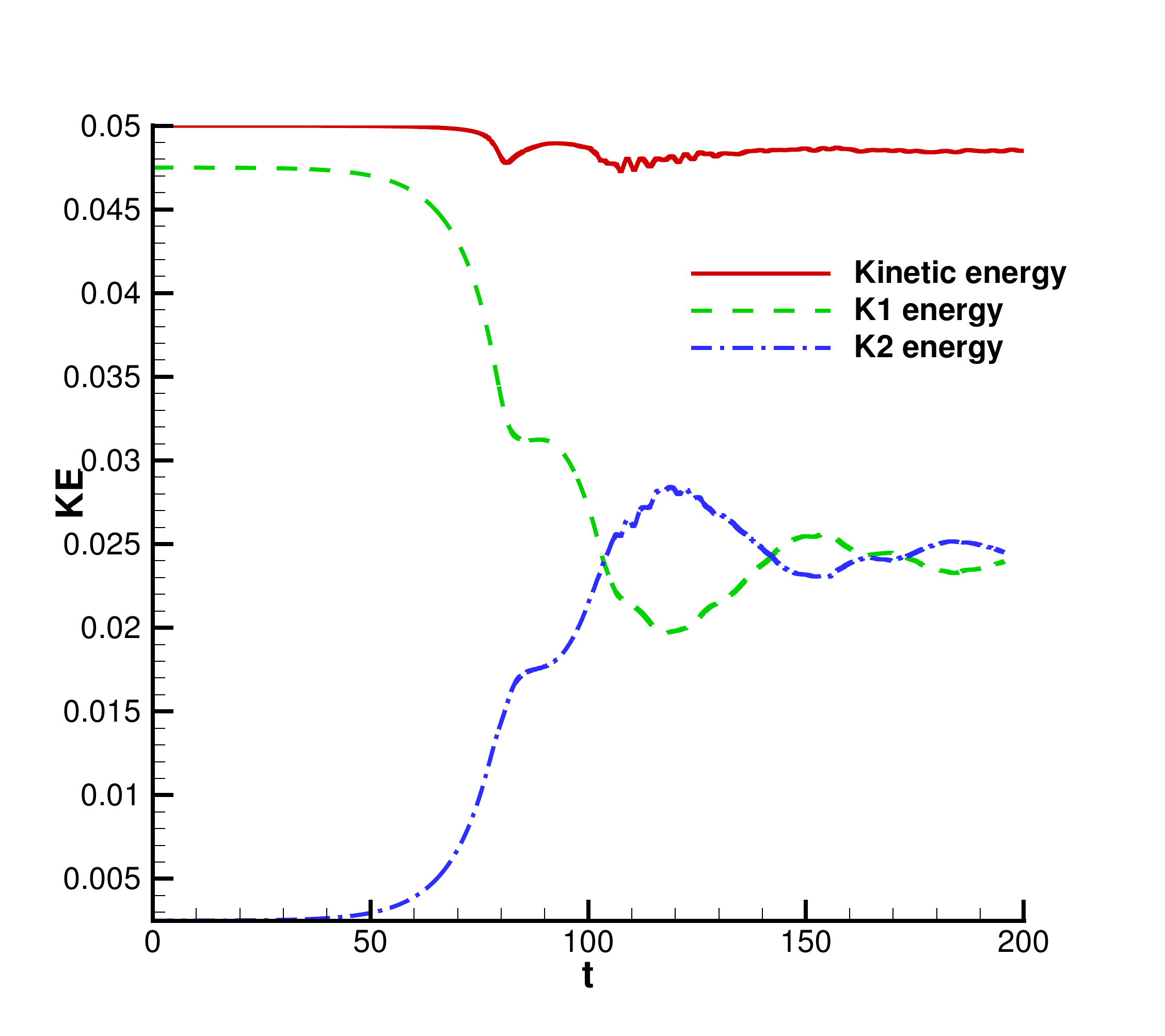}}
    \subfigure[Parameter Choice 1. Field energies]{\includegraphics[width=3.5in,angle=0]{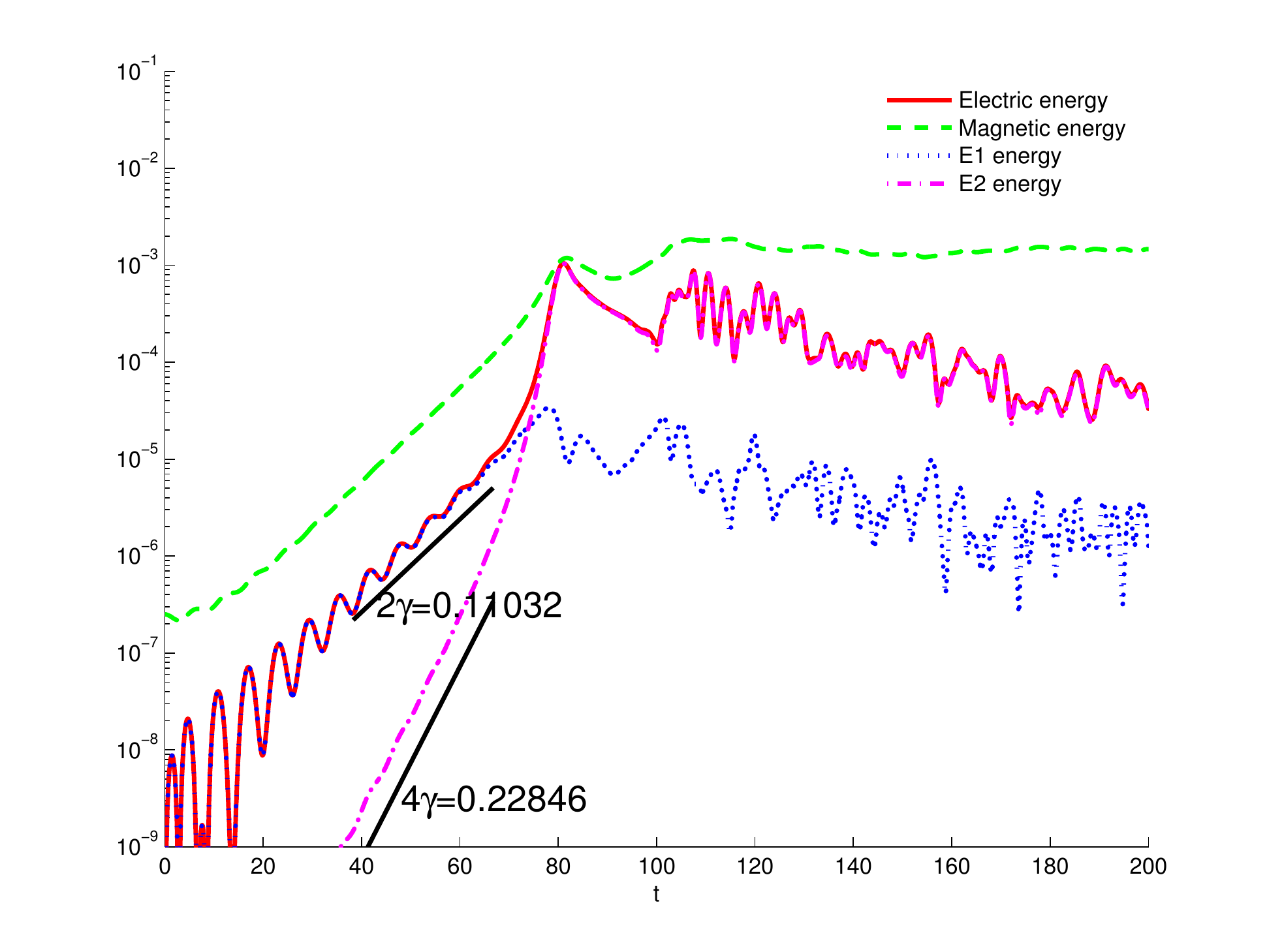}}
     \subfigure[Parameter Choice 2.  Kinetic energies]{\includegraphics[width=2.5in,angle=0]{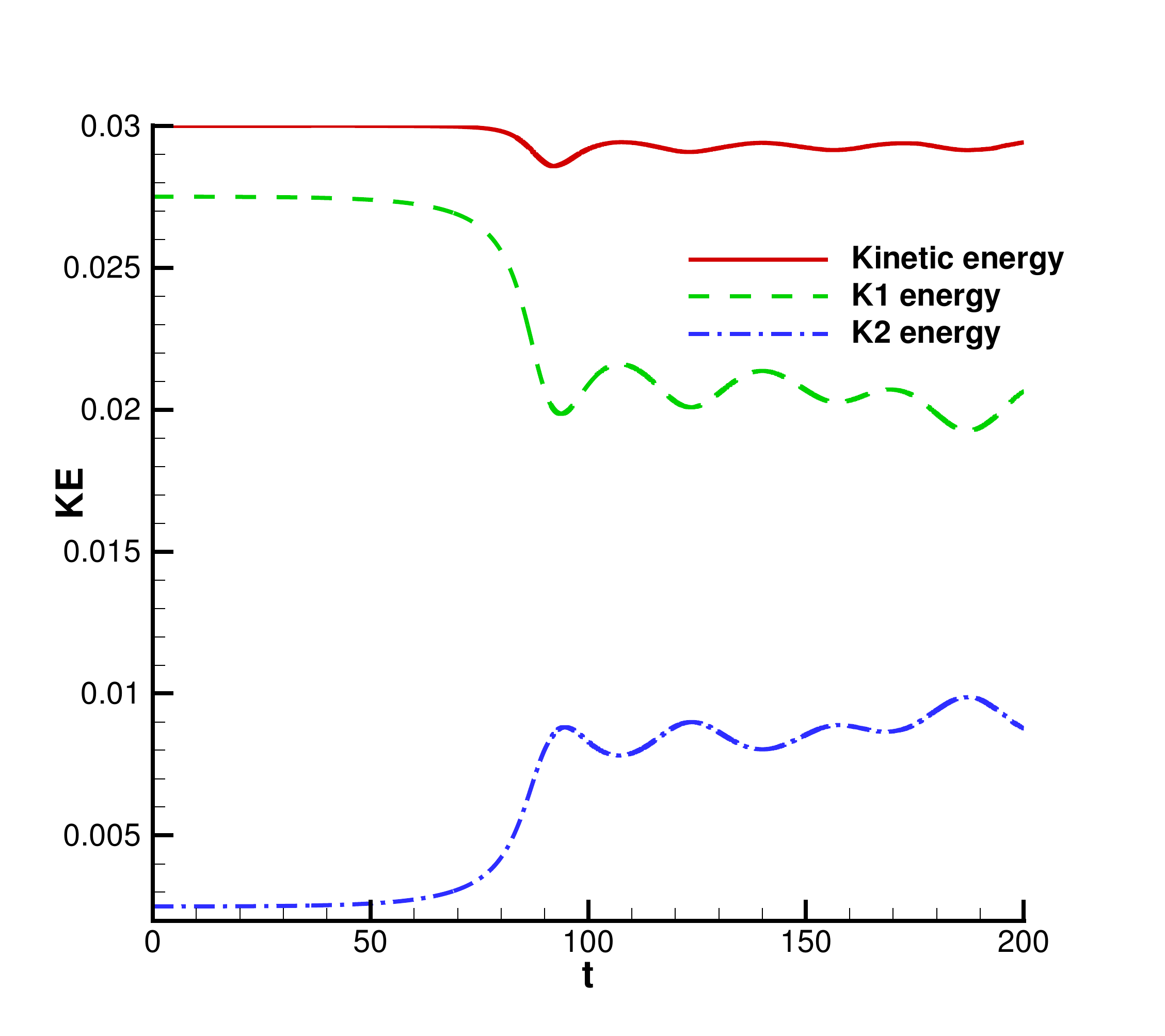}}
    \subfigure[Parameter Choice 2. Field energies]{\includegraphics[width=3.5in,angle=0]{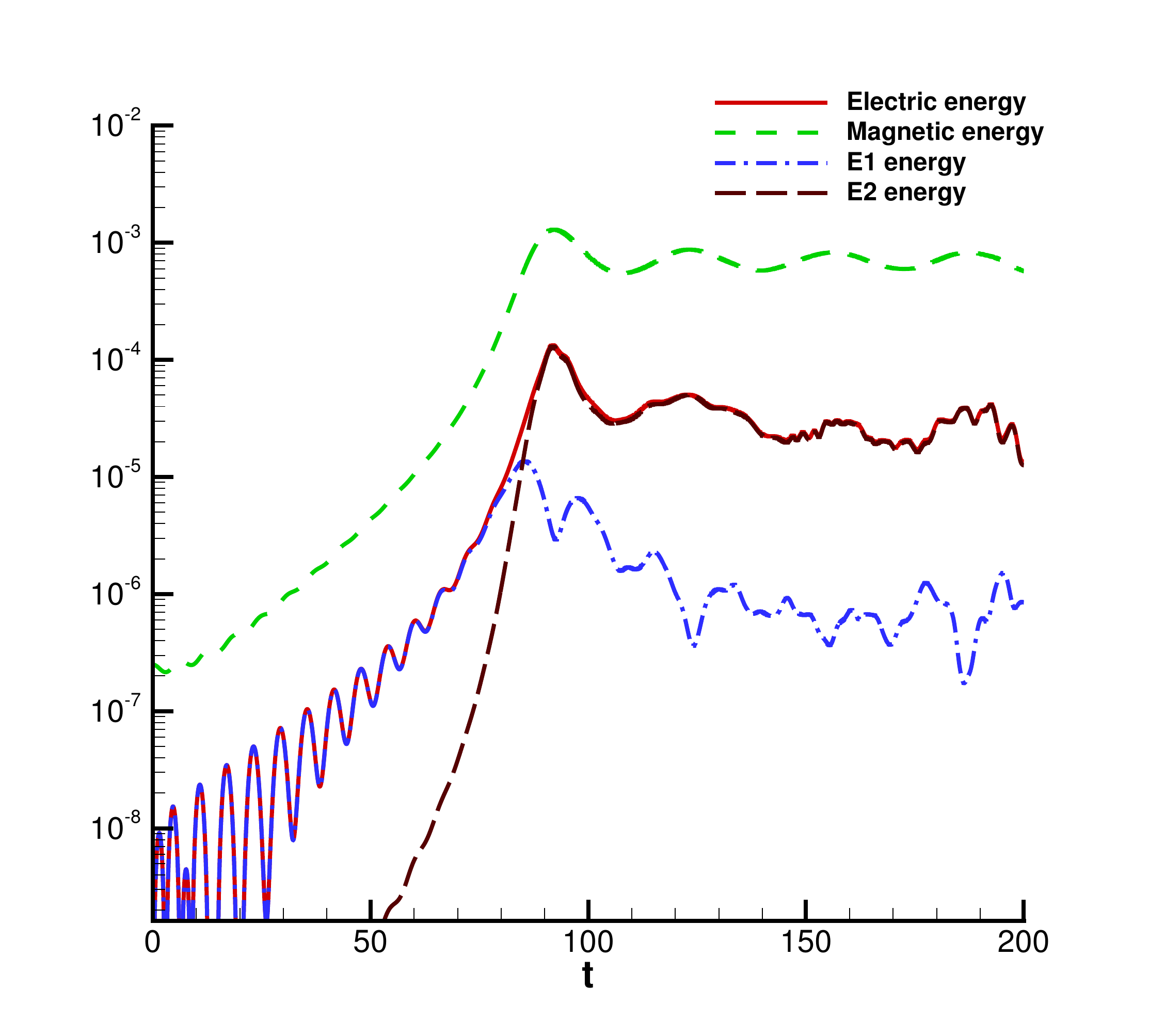}}
  \end{center}
    \caption{Streaming Weibel instability. The mesh is $100^3$ with piecewise quadratic polynomials. Time evolution of kinetic, electric  and magnetic energies  by alternating flux for the Maxwell's equations. }
\label{keeme1}
\end{figure}

\begin{figure}[htb]
  \begin{center}
    \subfigure[ Parameter Choice 1. $ t=55.$]{\includegraphics[width=2.9in,angle=0]{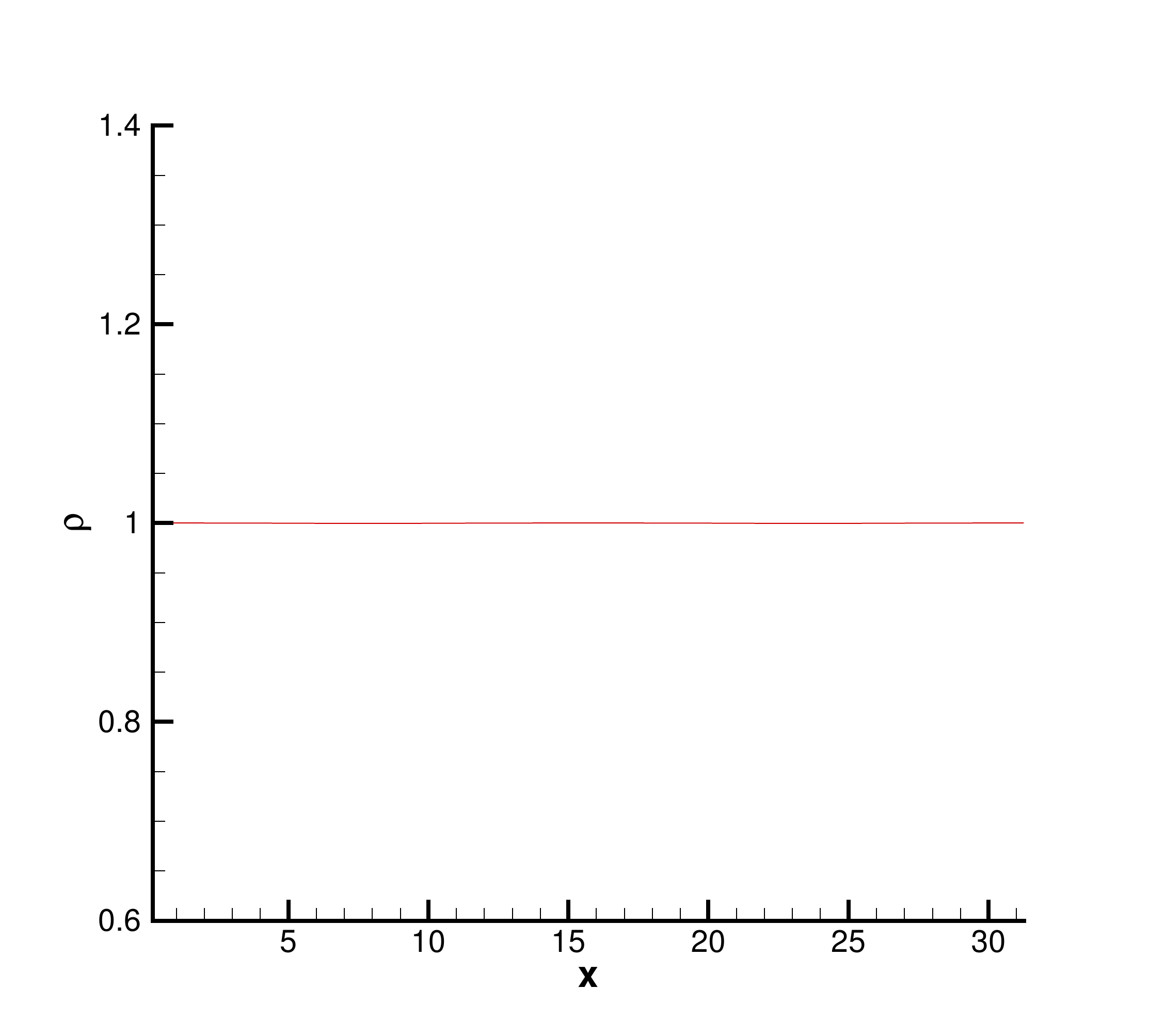}}
        \subfigure[ Parameter Choice 2. $ t=55.$]{\includegraphics[width=2.9in,angle=0]{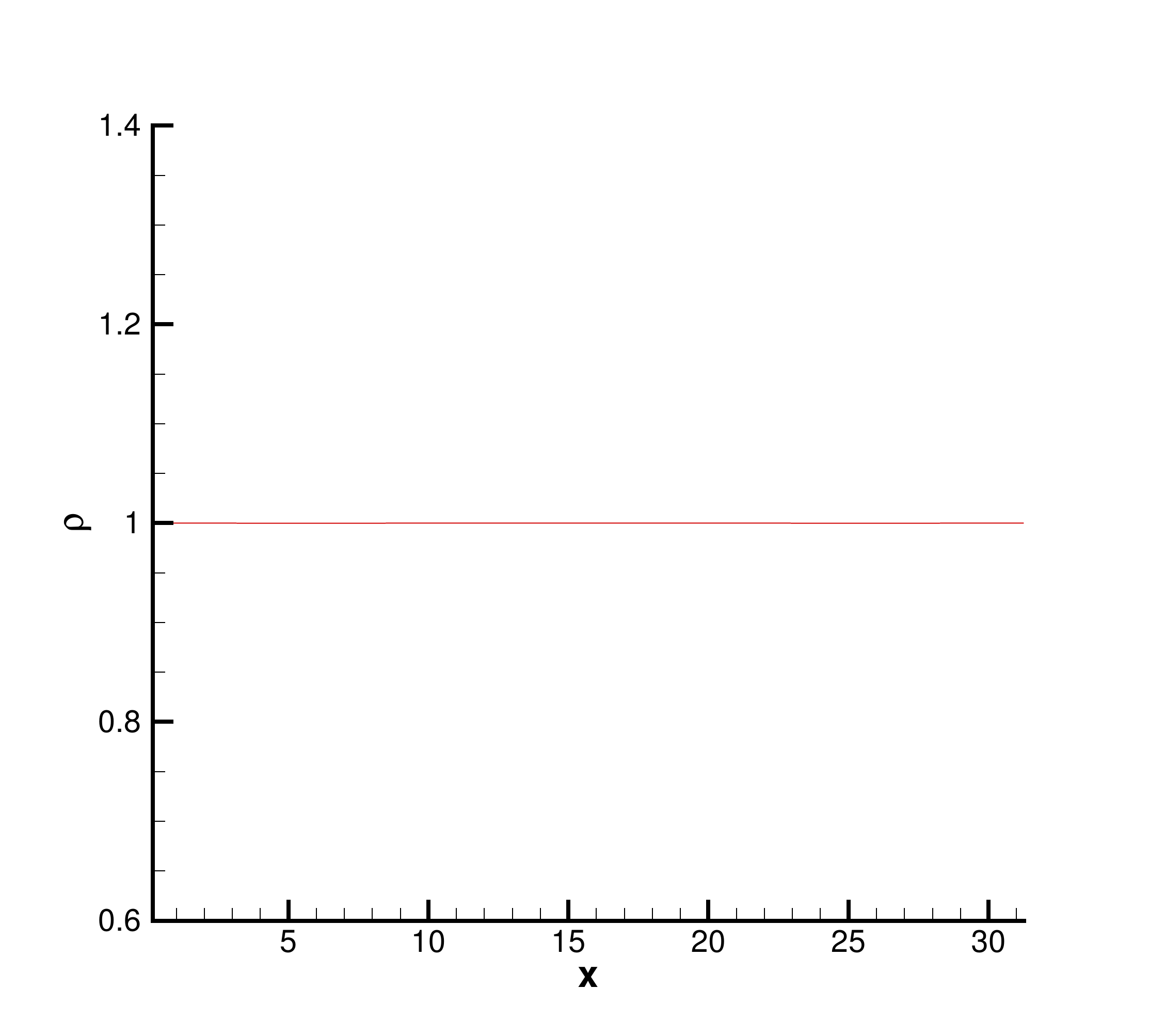}}
        \subfigure[ Parameter Choice 1. $ t=82.$]{\includegraphics[width=2.9in,angle=0]{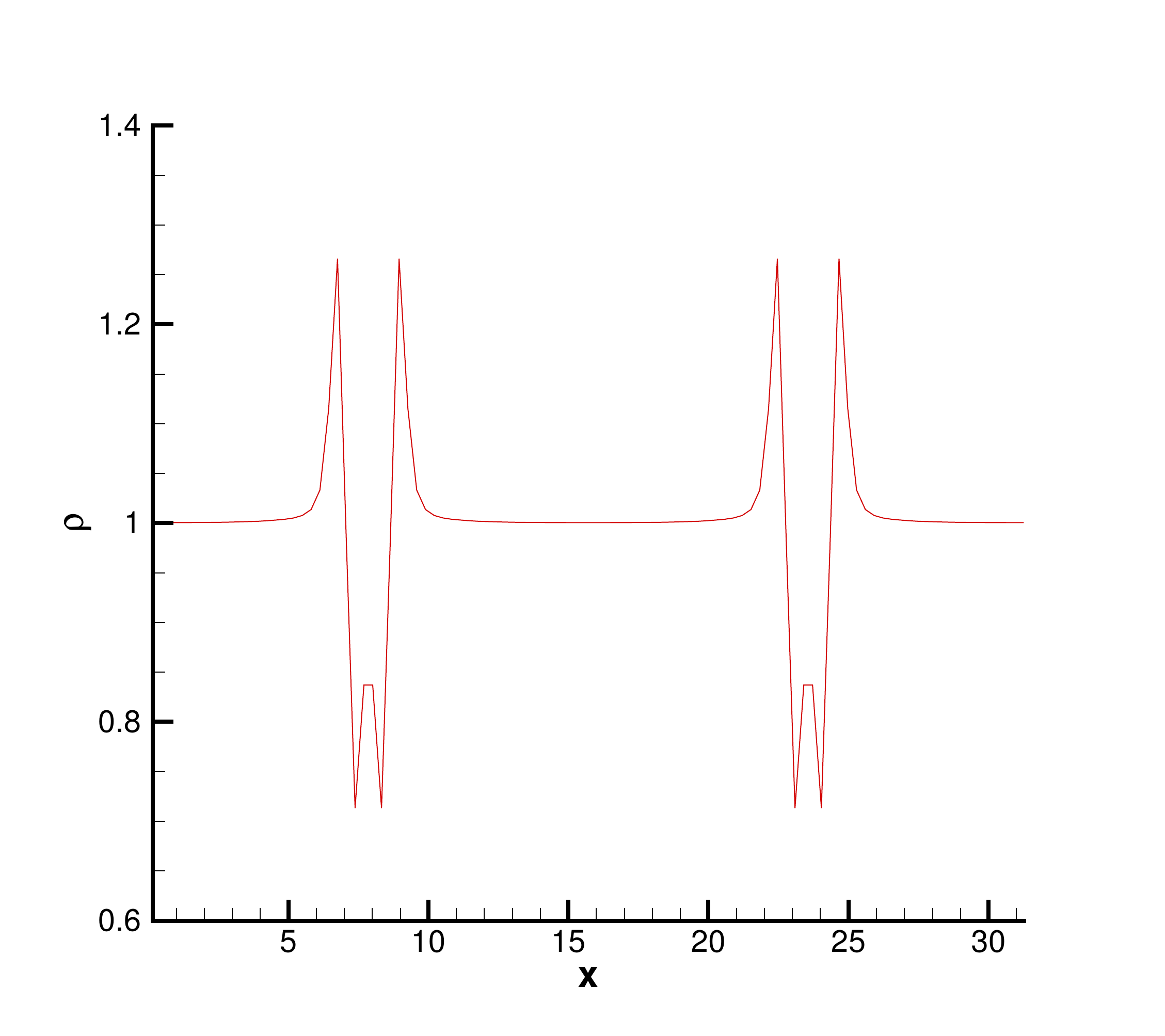}}
                \subfigure[ Parameter Choice 2. $ t=82.$]{\includegraphics[width=2.9in,angle=0]{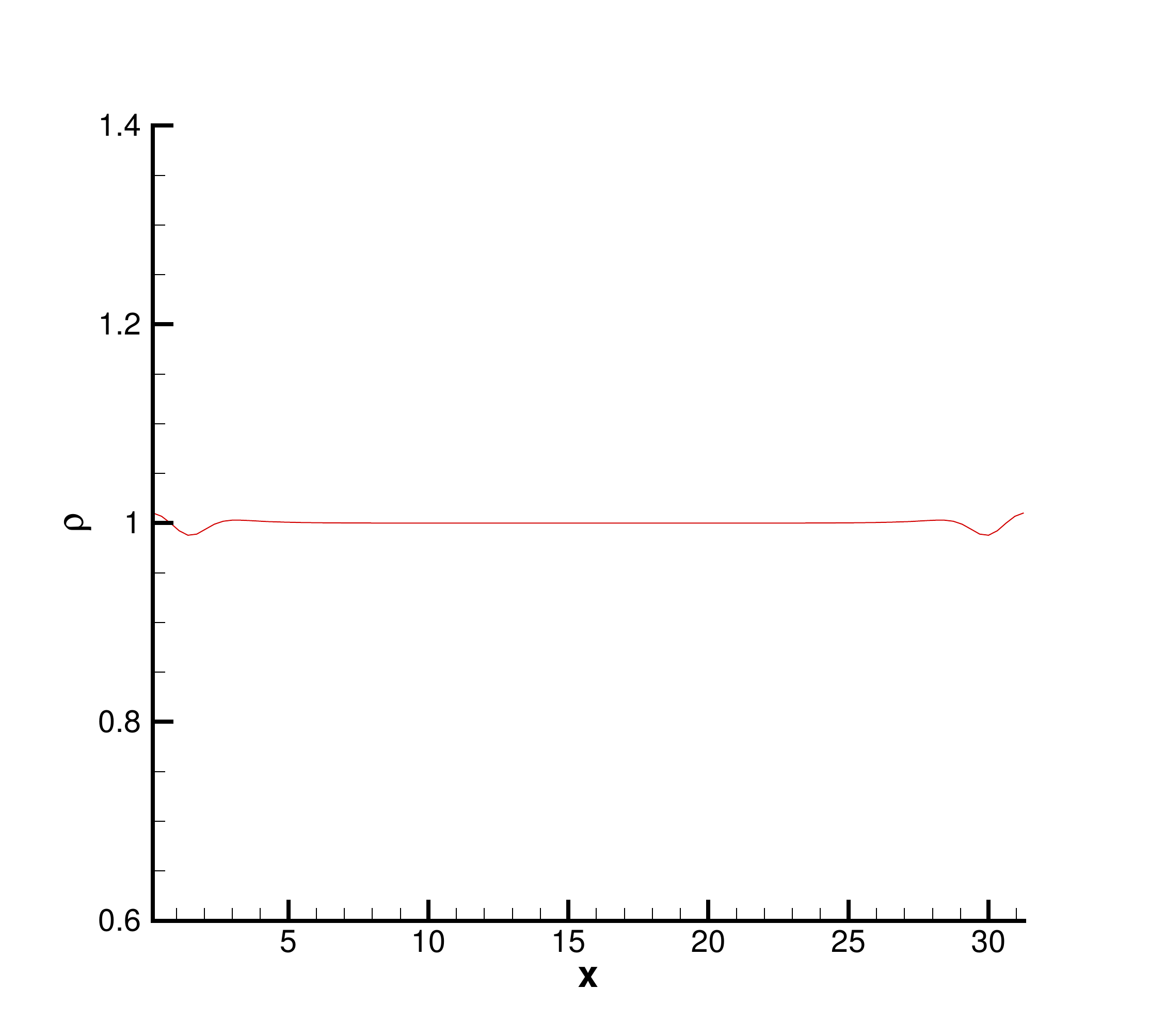}}
                \subfigure[ Parameter Choice 1. $ t=125.$]{\includegraphics[width=2.9in,angle=0]{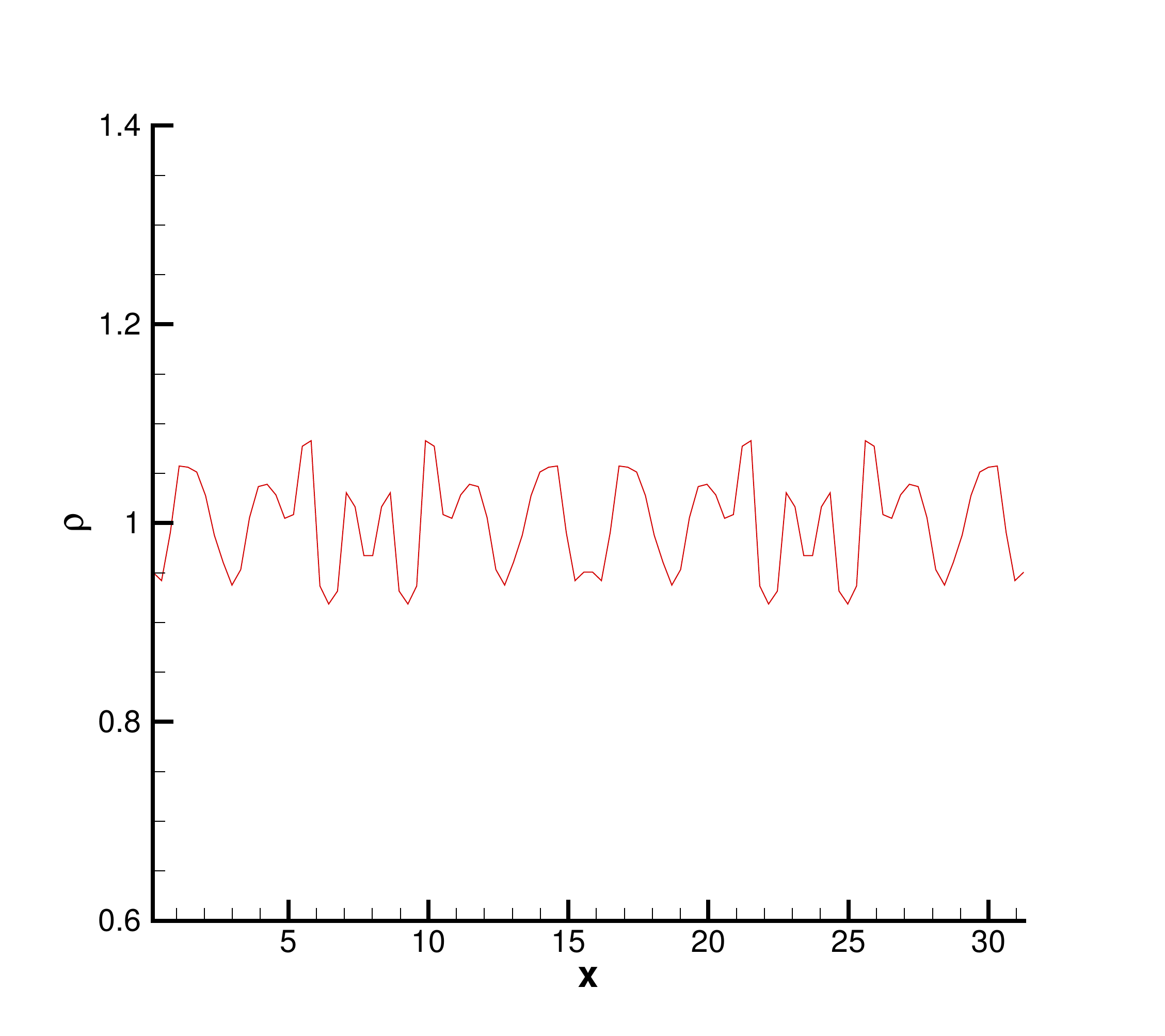}}
                                \subfigure[ Parameter Choice 2. $ t=125.$]{\includegraphics[width=2.9in,angle=0]{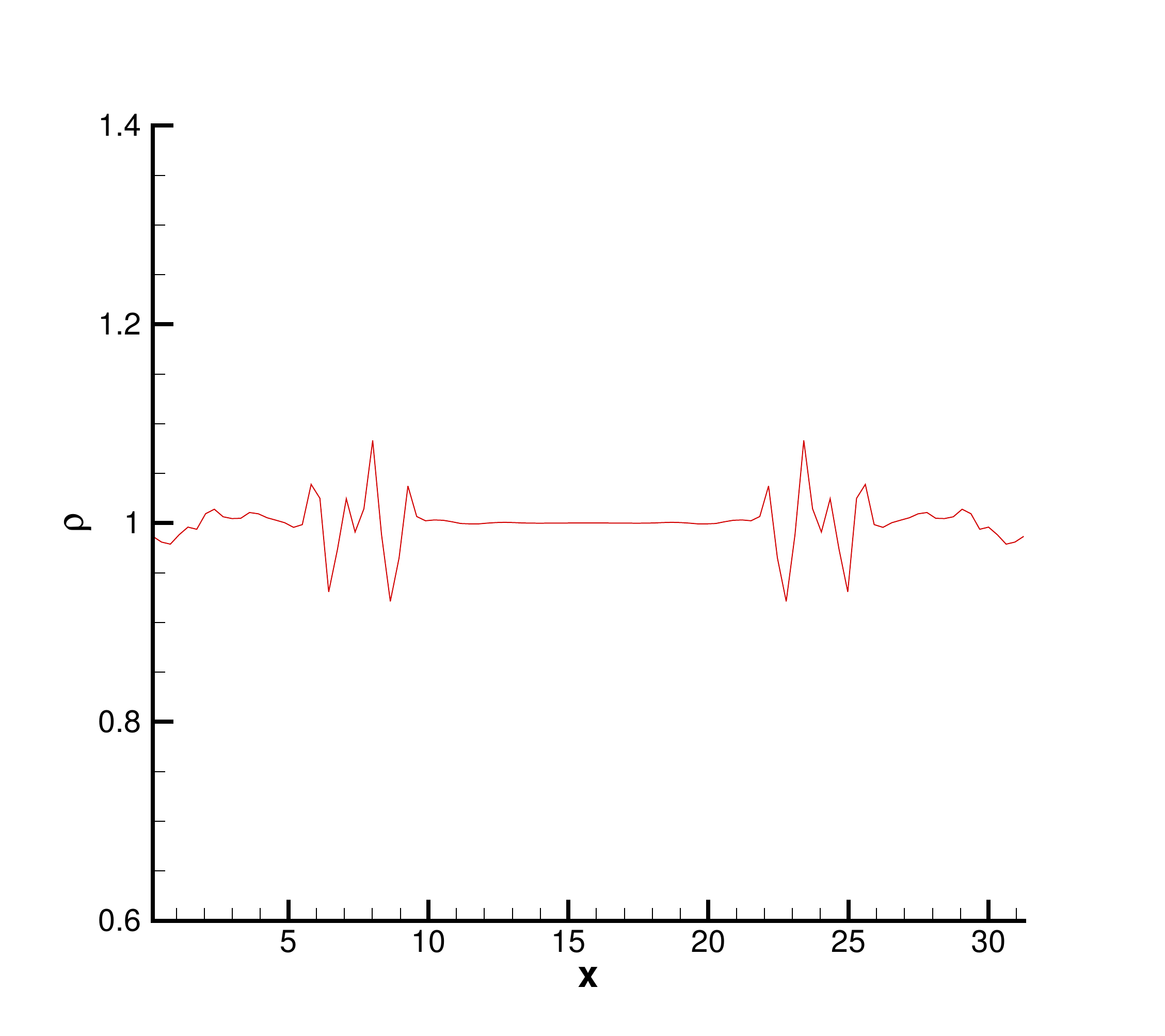}}
  \end{center}
  \caption{Plots of the computed density function $\rho_h$ for the streaming Weibel instability at selected  time $t$. The mesh is $100^3$ with piecewise quadratic polynomials. The upwind flux is applied. }
\label{density1}
\end{figure}

\begin{figure}[htb]
  \begin{center}
    \subfigure[Parameter Choice 1. Log Fourier modes of $E_1$]{\includegraphics[width=2.9in,angle=0]{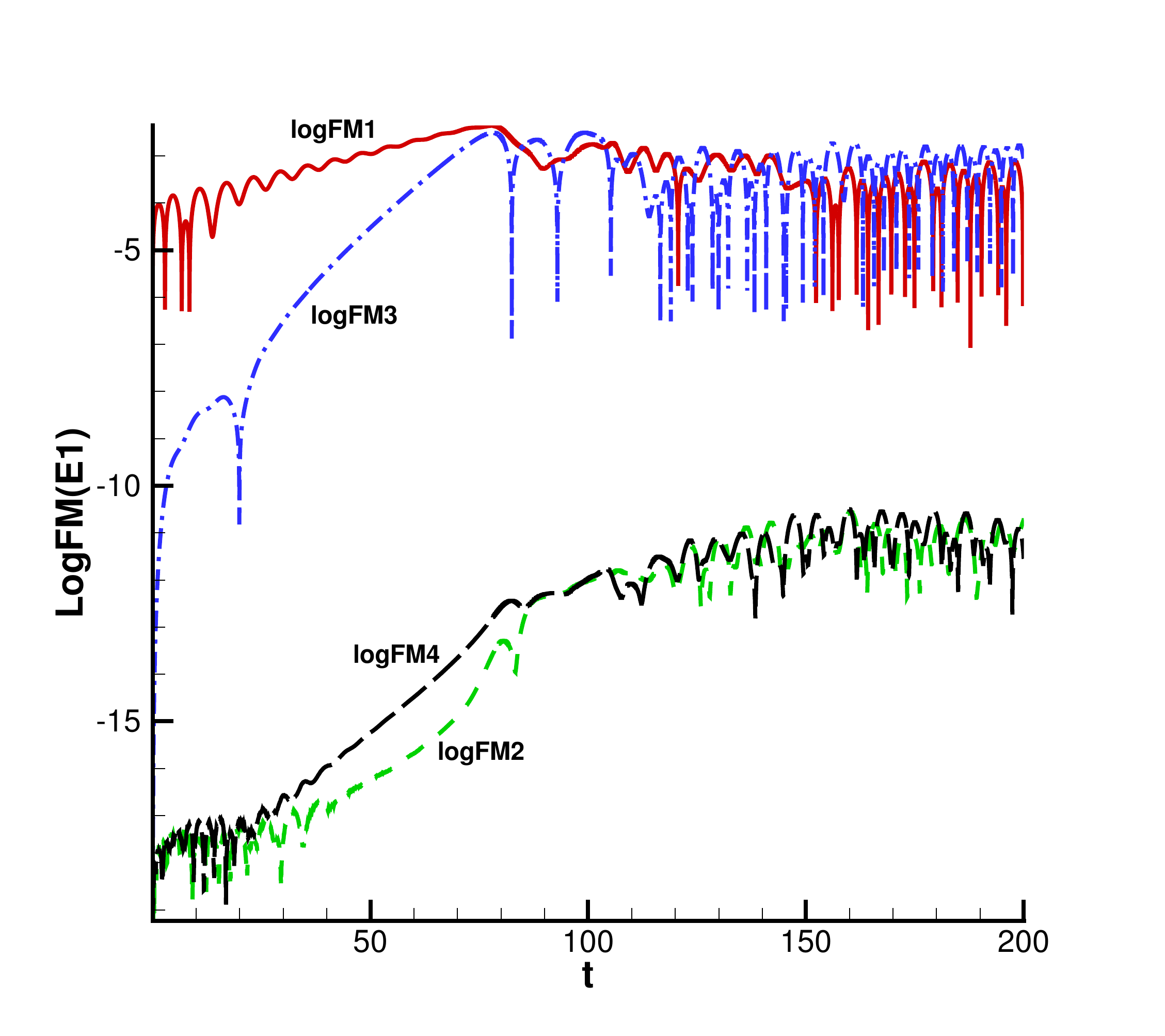}}
            \subfigure[Parameter Choice 2. Log Fourier modes of $E_1$]{\includegraphics[width=2.9in,angle=0]{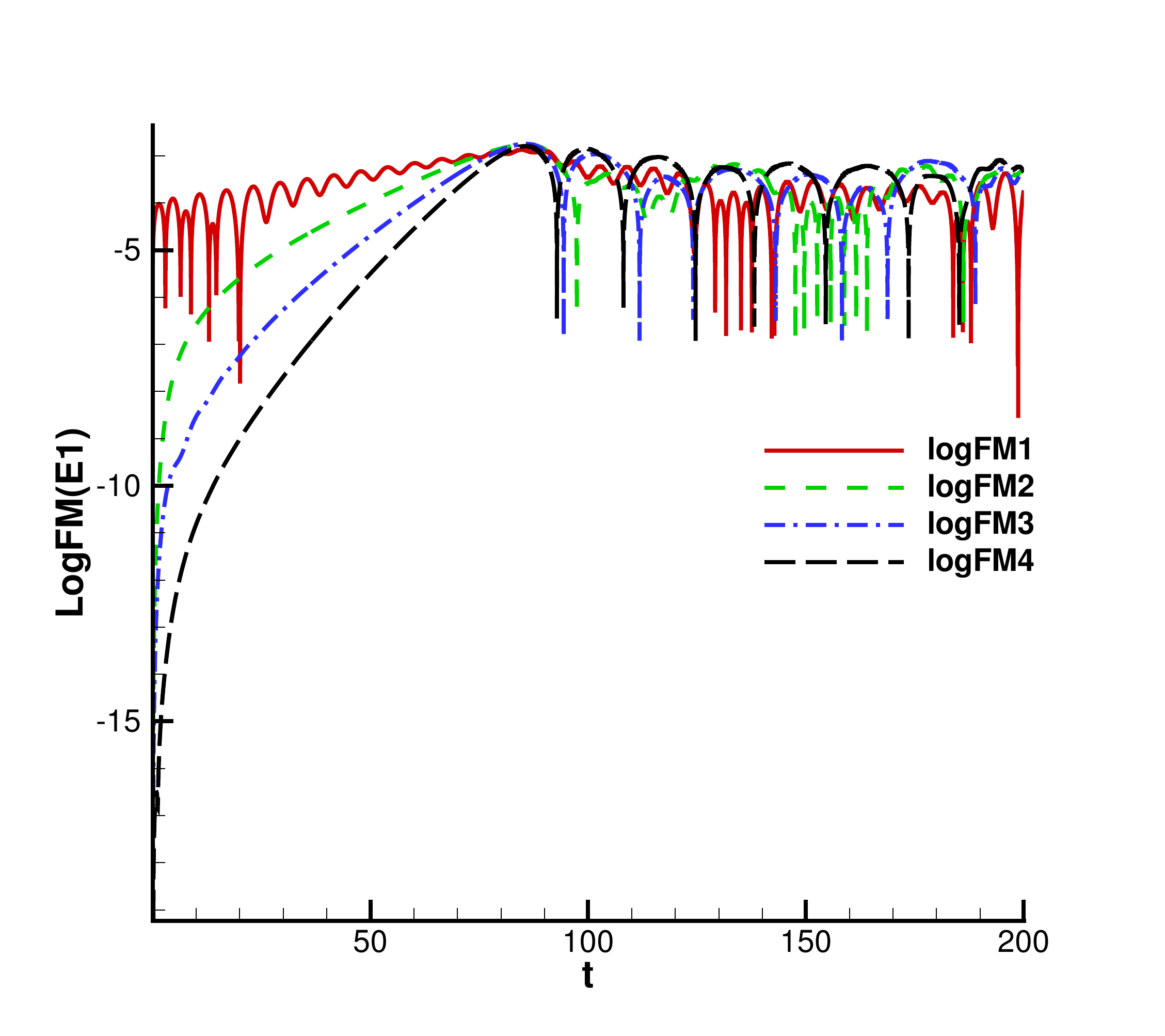}}
                \subfigure[Parameter Choice 1. Log Fourier modes of $E_2$]{\includegraphics[width=2.9in,angle=0]{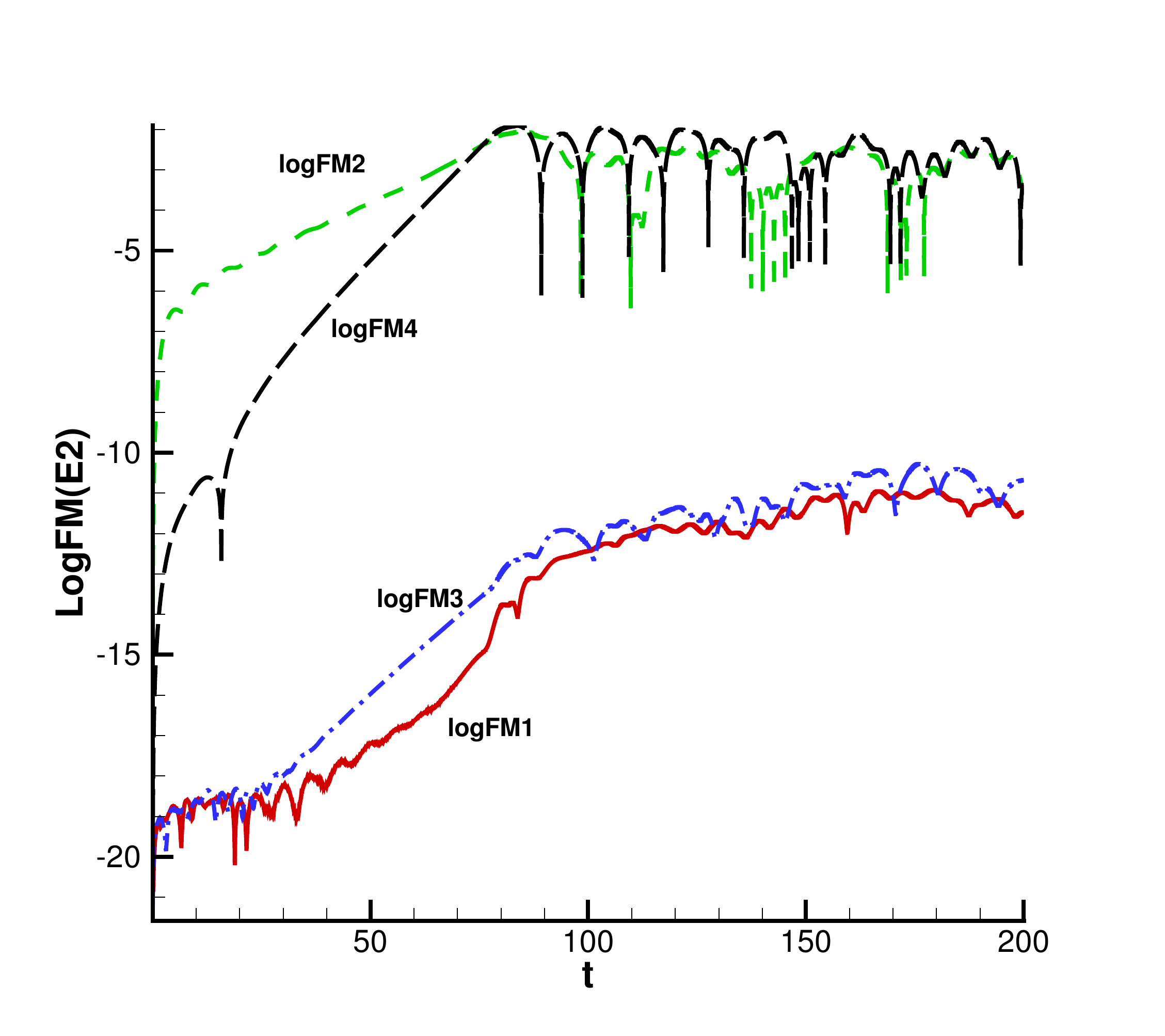}}
                    \subfigure[Parameter Choice 2. Log Fourier modes of $E_2$]{\includegraphics[width=2.9in,angle=0]{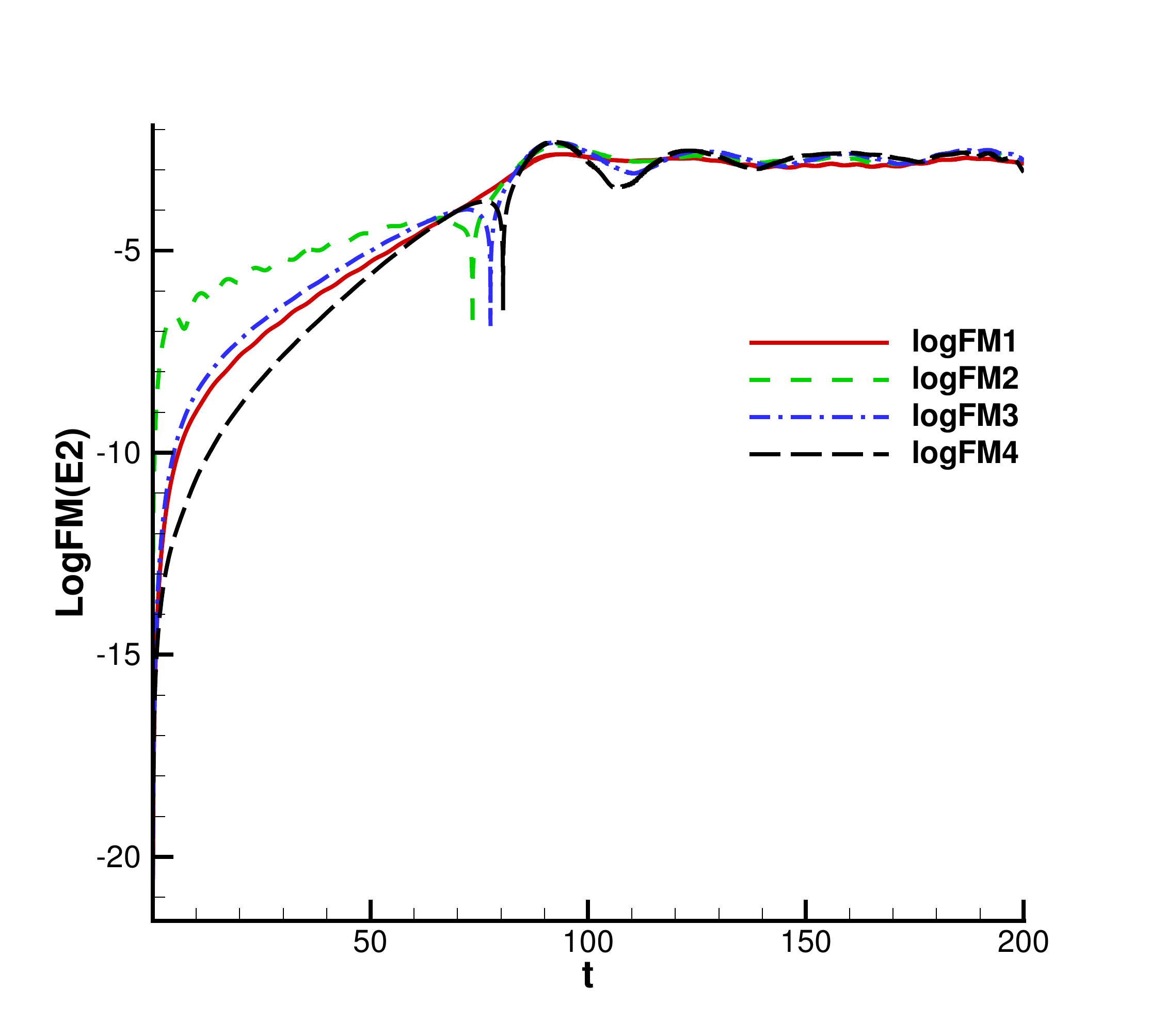}}
        \subfigure[Parameter Choice 1. Log Fourier modes of $B_3$]{\includegraphics[width=2.9in,angle=0]{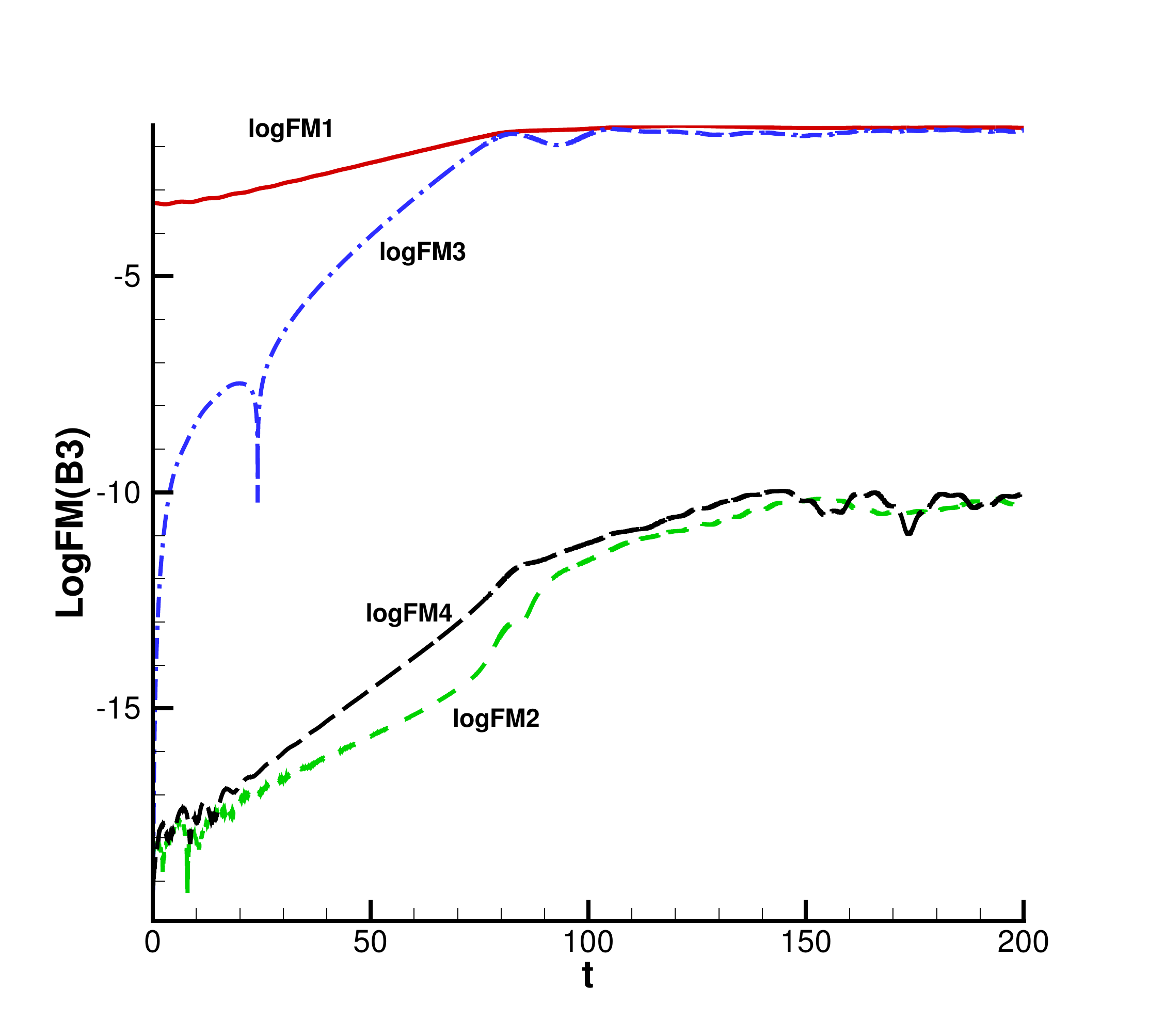}}
        \subfigure[Parameter Choice 2. Log Fourier modes of $B_3$]{\includegraphics[width=2.9in,angle=0]{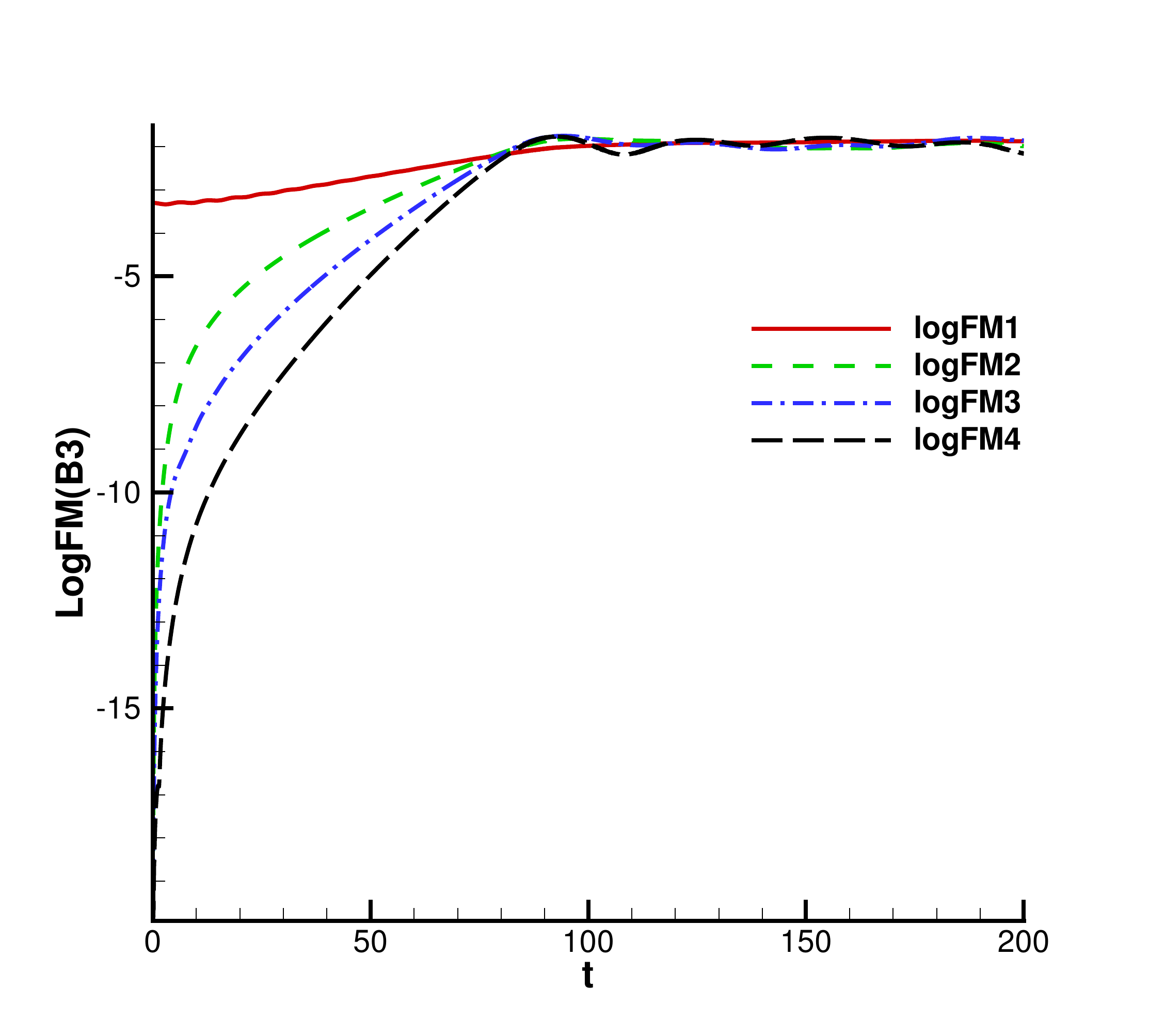}}
  \end{center}
    \caption{Streaming Weibel instability. The mesh is $100^3$ with piecewise quadratic polynomials. The first four Log Fourier modes of $E_1$, $E_2$, $B_3$ computed by the alternating flux for the Maxwell's equations.}
\label{logfm1}
\end{figure}

\begin{figure}[htb]
  \begin{center}
        \subfigure[ Parameter Choice 1. $x_2=0.05 \pi,  \, t=55.$]{\includegraphics[width=2.9in,angle=0]{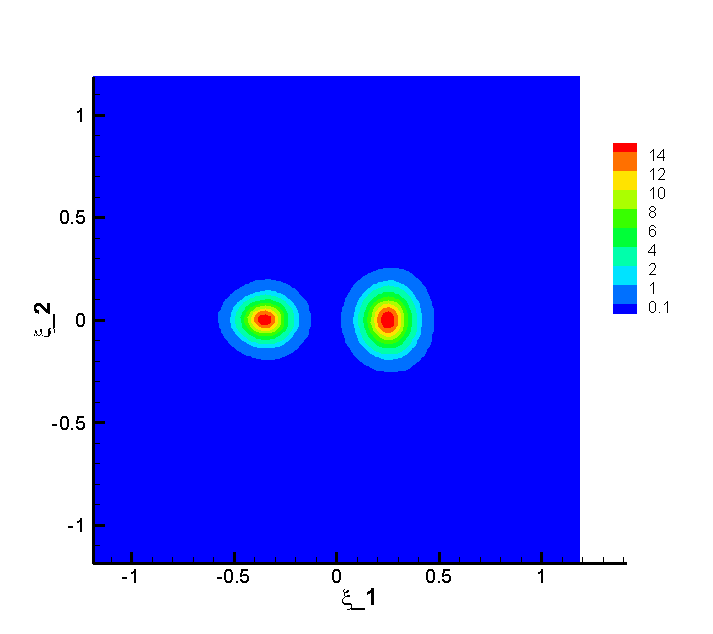}}
                \subfigure[ Parameter Choice 2. $x_2=0.05 \pi,  \, t=55.$]{\includegraphics[width=2.9in,angle=0]{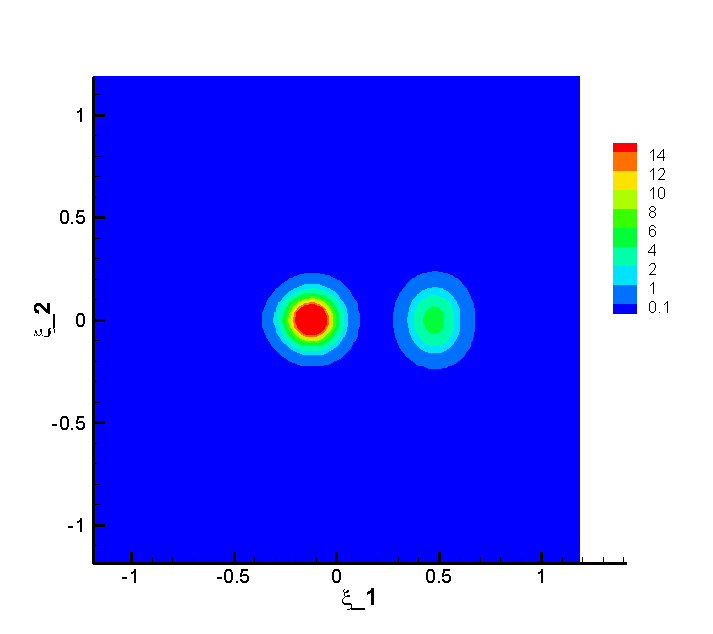}}
                \subfigure[Parameter Choice 1.  $x_2=0.05 \pi, \,  t=82.$]{\includegraphics[width=2.9in,angle=0]{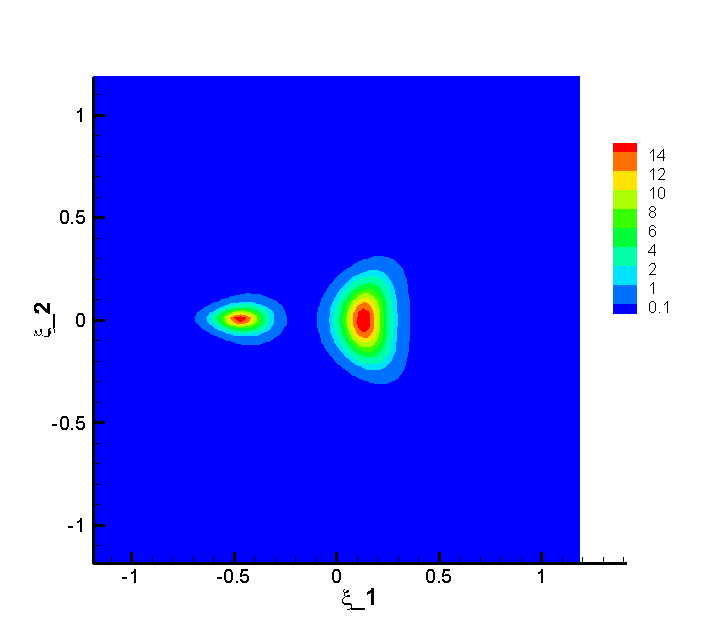}}
                                \subfigure[ Parameter Choice 2. $x_2=0.05 \pi, \,  t=82.$]{\includegraphics[width=2.9in,angle=0]{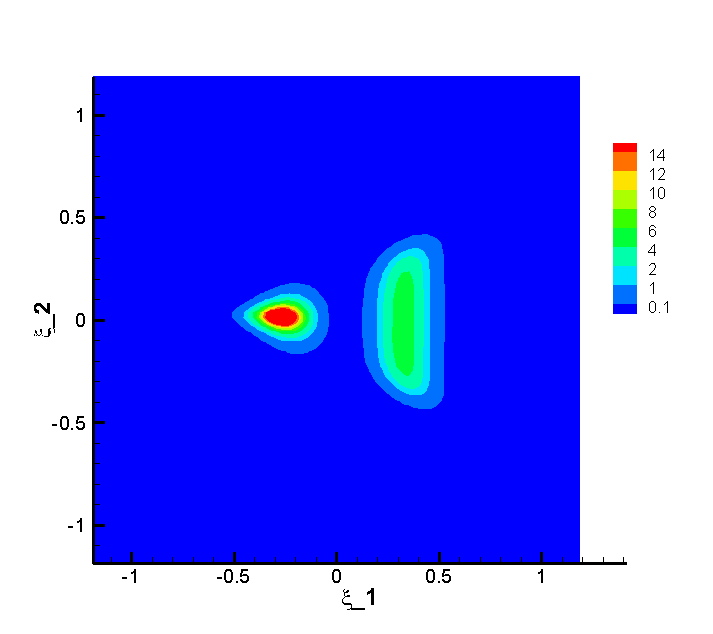}}
                            \subfigure[ Parameter Choice 1. $x_2=0.05 \pi,  \, t=125.$]{\includegraphics[width=2.9in,angle=0]{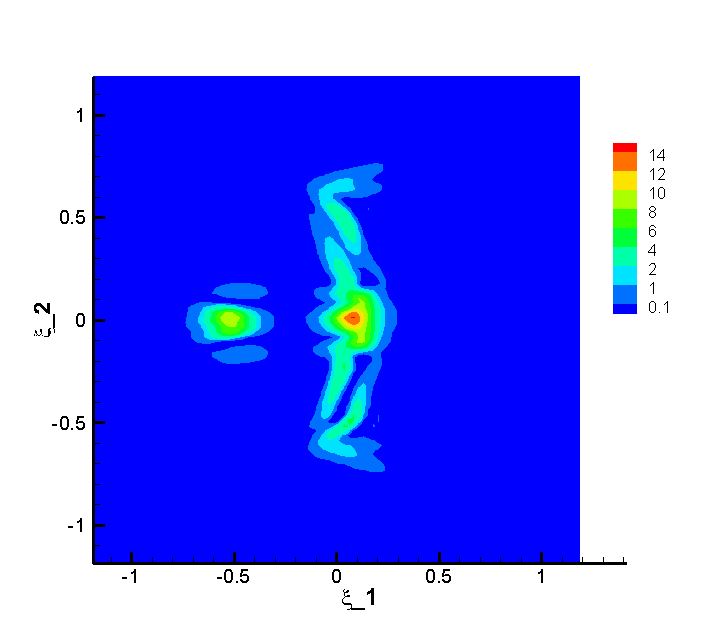}}
                                                        \subfigure[ Parameter Choice 2. $x_2=0.05 \pi,  \, t=125.$]{\includegraphics[width=2.9in,angle=0]{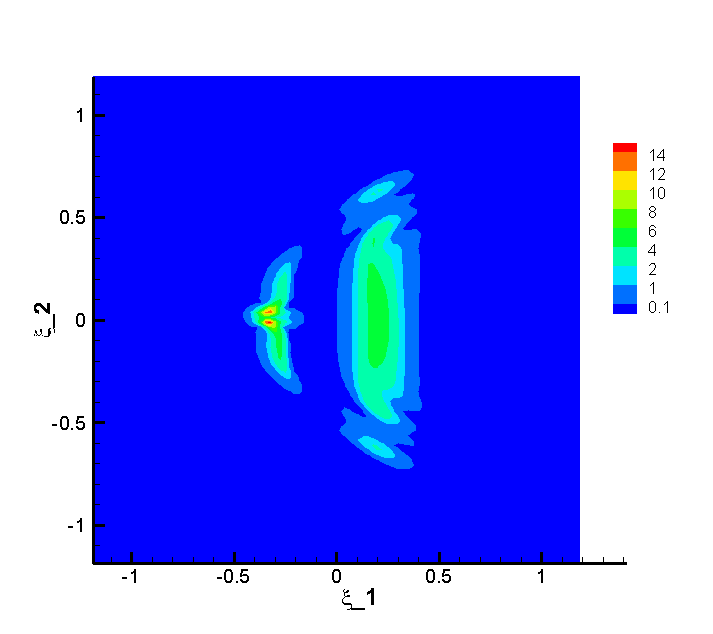}}
  \end{center}
  \caption{2D contour plots of the computed distribution function $f_h$ for the streaming  Weibel instability. The mesh is $100^3$ with piecewise quadratic polynomials. The upwind flux is applied. }
\label{contour1}
\end{figure}

For choice 2,  with the  nonsymmetric  parameter set, the results are included in Figures \ref{befield1}, \ref{keeme1},  \ref{density1}, \ref{logfm1},  and \ref{contour1},  juxtaposed  with those for  parameter choice 1.  
Insofar as we can make comparison with \cite{califano1998ksw}, our results are in reasonable agreement. Similar energy transfers take place, but the equipartition of the magnetic and electric energies at the peak is not achieved.  All modes saturate now at nearly the same values, evidently resulting from the broken symmetry.  Also, at long times, contours of  the distribution function are displayed.  Here the wrapping of the distribution function as two intertwined distorted  cylinders is observed  as in  \cite{califano1998ksw}, although for late times there is a loss of localization.

\section{Concluding Remarks}
\label{conclu}
In summary, we have  developed discontinuous Galerkin methods  for solving  the Vlasov-Maxwell system.  We have proven that the method is arbitrarily accurate, conserves charge,  can conserve energy, and  is stable.    Error estimates were  established for  several flux choices.  The scheme was tested on the streaming Weibel instability, where  the  order of accuracy  and  conservation properties  were  verified.
In the future, we will explore other time stepping methods to improve the efficiency of the overall algorithm.  In our development, the constraint equations of (\ref{eq:max:4})  were not considered;   in the future, we plan to investigate them together with some correction techniques for the continuity equation. The proposed method has been clearly established as sufficient for investigating the streaming Weibel instability, and the long time nonlinear physics of this system can be further investigated and modeled.  In the future,  we will also apply the method to  study other important plasma physics problems, especially those of higher dimension.


\section*{Acknowledgments}
Y.C. is supported by grant NSF DMS-1217563,  I.M.G. is supported by grant NSF DMS-1109625,  and  F.L. is partially supported by NSF CAREER award DMS-0847241 and an Alfred P. Sloan Research Fellowship.  P.J.M is supported by the US Department of Energy, grant DE-FG02-04ER54742;  He would like to thank  F. Pegoraro for  helpful correspondence.  Also, support from Department of Mathematics at Michigan State University and the Institute of Computational Engineering and Sciences at the University of Texas Austin are gratefully acknowledged.

\bibliographystyle{abbrv}
\bibliography{refer,ref_cheng_plasma_2}

\end{document}